\newif\ifprint%IM
\printfalse%IM off for contributors
\documentclass[twoside,10pt]{HandbookOfModuli}%IM

\usepackage{amsmath,amsthm}
\usepackage{amssymb}
\usepackage{latexsym}%IM
\usepackage[utf8]{inputenc}%IM
\usepackage{textcomp}%IM
\usepackage{color}%IM
\usepackage[pdftex]{graphicx}%IM
\usepackage{titlesec}
\usepackage{xspace}

\ifprint
	\input{giovanni} % off for contributors
	\usepackage[final]{microtype}%IM off for contributors
\fi
\usepackage{bm}
\renewcommand{\mathbf}[1]{\bm{#1}} % override \mathbf for use with Giovanni

\ifprint
	\definecolor{linkred}{rgb}{0,0,0} % black
	\definecolor{linkblue}{rgb}{0,0,0} % black
\else
	\definecolor{linkred}{rgb}{0.7,0.2,0.2}
	\definecolor{linkblue}{rgb}{0,0.2,0.6}
\fi

\numberwithin{equation}{section} % standardize numbering

\usepackage[
	hyperindex,
	pagebackref,
	pdftex,
	pdftitle={Handbook of Moduli},
	pdfdisplaydoctitle,
	pdfpagemode=UseNone,
	breaklinks=true,
	extension=pdf,
	bookmarks=false,
	plainpages=false,
	colorlinks,
	linkcolor=linkblue,
	citecolor=linkred,
	urlcolor=linkred,
	pdfmenubar=true,
	pdftoolbar=true,
	pdfpagelabels,
	pdfpagelayout=TwoPage,
	pdfview=Fit,
	pdfstartview=Fit
]{hyperref}%IM

\catcode`@=11 \baselineskip 4.5mm
\def\ps@handbook{\def\@oddhead{\hfill \leftmark \hfill\thepage }
\def\@evenhead{\thepage \hfill \rightmark \hfill}
\def\@oddfoot{}
\def\@evenfoot{}}
\def\@evenhead{}
\def\@oddfoot{}%\hfill \copyright\ China Higher Education Press}
\def\@evenfoot{\hfill\copyright\ China Higher Education Press}
\def\list#1#2{\ifnum \@listdepth >5\relax \@toodeep \else \global
\advance \@listdepth\@ne \fi \rightmargin \z@ \listparindent\z@
\itemindent\z@ \csname @list\romannumeral\the\@listdepth\endcsname
\def\@itemlabel{#1}\let\makelabel\@mklab \@nmbrlistfalse #2\relax
\@trivlist \parskip -\parsep \parindent\listparindent \advance
\linewidth -\rightmargin \advance\linewidth -\leftmargin \advance
\@totalleftmargin \leftmargin \parshape \@ne \@totalleftmargin
\linewidth \ignorespaces}
\catcode`@=12
\renewcommand\thesection{\arabic{section}}
\renewcommand\thesubsection{\arabic{subsection}}
\renewcommand\thesubsubsection{\arabic{subsubsection}}
\renewcommand{\theequation}{\thesection.\arabic{equation}}

\def\thebibliography#1{\section*{References}
\list{[\arabic{enumi}]}{\settowidth \labelwidth{[#1]} \leftmargin
\labelwidth \advance \leftmargin \labelsep \usecounter{enumi}}
\def\newblock{\hskip .11em plus .33em minus .07em} \sloppy
\clubpenalty 4000 \widowpenalty 4000 \sfcode`\.=1000 \relax}
\titleformat{\section}{\normalfont\large\bfseries}{\thesection.}{0.5em}{}
\titleformat{\subsection}{\normalfont\bfseries}{\thesection.\thesubsection.}{0.5em}{}
\titleformat{\subsubsection}[runin]{\normalfont\bfseries}{\thesection.\thesubsection.\thesubsubsection.}{0.5em}{}[\kern0.5em]%IM

\usepackage{amsxtra}
\usepackage{amscd}
\usepackage{longtable}
\usepackage{bigstrut}
\usepackage[curve,matrix,arrow,frame,tips]{xy}
\usepackage{mathdots}
\usepackage{verbatim}
\usepackage{graphicx}
\usepackage{xspace}
\usepackage{ifthen}
\numberwithin{equation}{subsection}
\renewcommand{\theequation}{\thesection.\thesubsection.\arabic{equation}}

\theoremstyle{plain}
\newtheorem{thm}[equation]{Theorem}
\newtheorem{prop}[equation]{Proposition}
\newtheorem{lemma}[equation]{Lemma}
\newtheorem{conjecture}[equation]{Conjecture}
\newtheorem{cor}[equation]{Corollary}

\theoremstyle{definition}
\newtheorem{Definition}[equation]{Definition}
\newtheorem{Variants}[equation]{Variants}

\newtheorem{Remark}[equation]{Remark}

\newtheorem{eg}[equation]{Example}

\newcommand{\cal}{\mathcal}

\newcommand{\A}{{\mathcal A}}

\newcommand{\B}{{\mathcal B}}

\newcommand{\Gg}{{\mathcal G}}

\newcommand{\Z}{{\bf Z}}
\newcommand{\F}{{\mathcal F}}
\newcommand{\ti}{\tilde}
 \renewcommand{\O}{{\mathcal O}}

\newcommand{\M}{{\mathcal M}}
\renewcommand{\L}{{\mathcal L}}

\newcommand{\La}{\Lambda}

\newcommand{\I}{{\mathcal I}}

\newcommand{\Tr}{{\rm Tr}}

\newcommand{\ov}{\overline}

\renewcommand{\P}{{\cal P}}

\DeclareMathOperator{\Ker}{Ker}

\newcommand{\mto}{\longmapsto}

\newcommand{\la}{\langle}
\newcommand{\ra}{\rangle}

\newcommand{\mA}{\ensuremath{\mathbb{A}}\xspace}
\newcommand{\mC}{\ensuremath{\mathbb{C}}\xspace}
\newcommand{\mF}{\ensuremath{\mathbb{F}}\xspace}
\newcommand{\mG}{\ensuremath{\mathbb{G}}\xspace}
\newcommand{\mP}{\ensuremath{\mathbb{P}}\xspace}
\newcommand{\mQ}{\ensuremath{\mathbb{Q}}\xspace}
\newcommand{\mR}{\ensuremath{\mathbb{R}}\xspace}
\newcommand{\mZ}{\ensuremath{\mathbb{Z}}\xspace}

\entrymodifiers={+!!<0pt,\fontdimen22\textfont2>}

\DeclareMathOperator{\Adm}{Adm}
\DeclareMathOperator{\Aff}{Aff}

\DeclareMathOperator{\charac}{char}

\DeclareMathOperator{\diag}{diag}
\DeclareMathOperator{\Gal}{Gal}
\DeclareMathOperator{\Gr}{Gr}
\DeclareMathOperator{\Lat}{Lat}
\DeclareMathOperator{\LGr}{LGr}

\DeclareMathOperator{\OGr}{OGr}
\DeclareMathOperator{\Perm}{Perm}
\DeclareMathOperator{\rank}{rank}
\DeclareMathOperator{\Res}{Res}
\DeclareMathOperator{\Spec}{Spec}
\DeclareMathOperator{\spn}{span}

\newcommand{\ad}{\ensuremath{\mathrm{ad}}\xspace}
\newcommand{\aff}{\ensuremath{\mathrm{a}}\xspace}
\newcommand{\aform}{$\langle$~,~$\rangle$\xspace}
\newcommand{\bs}{\backslash}
\newcommand{\can}{\ensuremath{\mathrm{can}}\xspace}
\newcommand{\der}{\ensuremath{\mathrm{der}}\xspace}
\newcommand{\Fr}{\ensuremath{\mathrm{Fr}}\xspace}
\newcommand{\id}{\ensuremath{\mathrm{id}}\xspace}
\newcommand{\loc}{\ensuremath{\mathrm{loc}}\xspace}
\newcommand{\naive}{\ensuremath{\mathrm{naive}}\xspace}
\newcommand{\ol}{\overline}
\newcommand{\R}{\ensuremath{\mathcal{R}}\xspace}
\renewcommand{\sc}{\ensuremath{\mathrm{sc}}\xspace}
\newcommand{\sep}{\ensuremath{\mathrm{sep}}\xspace}
\newcommand{\spin}{\ensuremath{\mathrm{spin}}\xspace}
\newcommand{\sform}{$($~,~$)$\xspace}

\newcommand{\un}{\ensuremath{\mathrm{un}}\xspace}
\renewcommand{\vert}{\ensuremath{\mathrm{vert}}\xspace}
\newcommand{\wh}{\widehat}
\newcommand{\wt}{\widetilde}
\SelectTips{cm}{}
\renewcommand{\to}{%
   \ifthenelse{\boolean{@display}}{\longrightarrow}{\rightarrow}%
   }
\newcommand{\inj}{\hookrightarrow}
\newcommand{\surj}{\twoheadrightarrow}
\let\shortmapsto\mapsto
\renewcommand{\mapsto}{%
   \ifthenelse{\boolean{@display}}{\longmapsto}{\shortmapsto}%
   }
\newcommand{\isoarrow}{%
\ifthenelse{\boolean{@display}}{\overset{\sim}{\longrightarrow}}{\xrightarrow\sim}%
   }
\newlength{\olen}
\newlength{\ulen}
\newlength{\xlen}
\newcommand{\xra}[2][]{%
   \ifthenelse{\boolean{@display}}{%
      \settowidth{\olen}{$\overset{#2}{\longrightarrow}$}%
      \settowidth{\ulen}{$\underset{#1}{\longrightarrow}$}%
      \settowidth{\xlen}{$\xrightarrow[#1]{#2}$}%
      \ifthenelse{\lengthtest{\olen > \xlen}}%
         {\underset{#1}{\overset{#2}{\longrightarrow}}}%
         {\ifthenelse{\lengthtest{\ulen > \xlen}}%
            {\underset{#1}{\overset{#2}{\longrightarrow}}}
            {\xrightarrow[#1]{#2}}}}{
      \xrightarrow[#1]{#2}}
   }

\newenvironment{altenumerate}
   {\begin{list}
      {\textup{(\theenumi)} }
      {\usecounter{enumi}
       \setlength{\labelwidth}{0pt}
       \setlength{\labelsep}{0pt}
       \setlength{\leftmargin}{0pt}
       \setlength{\itemsep}{\the\smallskipamount}
       \renewcommand{\theenumi}{\roman{enumi}}
      }}
   {\end{list}}
\newenvironment{altitemize}
   {\begin{list}
      {$\bullet$ }
      {\setlength{\labelwidth}{0pt}
       \setlength{\labelsep}{0pt}
       \setlength{\leftmargin}{0pt}
       \setlength{\itemsep}{\the\smallskipamount}
      }}
   {\end{list}}
   
\newsavebox{\lineone}
\newsavebox{\linetwo}
\newsavebox{\linethree}
\newlength{\lineonelen}
\newlength{\linetwolen}
\newlength{\linethreelen}
\newlength{\biggerlen}
\newcommand{\twolinestight}[2]{%
   \sbox{\lineone}{#1}%
   \sbox{\linetwo}{#2}%
   \settowidth{\lineonelen}{\usebox{\lineone}}%
   \settowidth{\linetwolen}{\usebox{\linetwo}}%
   \ifthenelse{\lengthtest{\lineonelen > \linetwolen}}%
      {\setlength{\biggerlen}{\the\lineonelen}}%
      {\setlength{\biggerlen}{\the\linetwolen}}%
   \begin{minipage}{\biggerlen}%
      \centering
      \usebox\lineone\\%
      \usebox\linetwo%
   \end{minipage}%
   }
\newcommand{\threelinestight}[3]{%
   \sbox{\lineone}{#1}%
   \sbox{\linetwo}{#2}%
   \sbox{\linethree}{#3}%
   \settowidth{\lineonelen}{\usebox{\lineone}}%
   \settowidth{\linetwolen}{\usebox{\linetwo}}%
   \settowidth{\linethreelen}{\usebox{\linethree}}%
   \ifthenelse{\lengthtest{\lineonelen > \linetwolen}}%
      {\ifthenelse{\lengthtest{\lineonelen > \linethreelen}}%
         {\setlength{\biggerlen}{\the\lineonelen}}%
         {\setlength{\biggerlen}{\the\linethreelen}}%
      }%
      {\ifthenelse{\lengthtest{\linetwolen > \linethreelen}}%
         {\setlength{\biggerlen}{\the\linetwolen}}%
         {\setlength{\biggerlen}{\the\linethreelen}}%
      }%   
   \begin{minipage}{\biggerlen}%
      \centering
      \usebox\lineone\\%
      \usebox\linetwo\\%
      \usebox\linethree\\
   \end{minipage}%
   }
   
\newcounter{dummy}

\begin{document}
\setcounter{page}{1}% to be adjusted when all contributions are assembled
%
% 3. Please fill in the template fields below 
% indicated by a \replace macro.
%    
\long\def\replace#1{#1}

%    
%    Title: you may add a short title for running heads, if needed as an optional
%    argument as in the example below
%    \title[Guide to Contributors]{Guide for Contributors to the Handbook of Moduli}
%    
\title[Local models of Shimura varieties, I.] {Local models of Shimura varieties, I. \\  Geometry and combinatorics}
%   
%    Author information--add or delete authors as needed
%    
\author[G. Pappas]{Georgios Pappas*}\thanks{*Partially supported by NSF grants DMS08-02686 and  DMS11-02208.}
\address{Department of Mathematics\\
Michigan State University\\
East Lansing, MI 48824-1027\\
USA}
\email{pappas@math.msu.edu}

\author[M. Rapoport]{Michael Rapoport}
\address{Mathematisches Institut der Universit\"at Bonn\\  
Endenicher Allee 60\\
53115 Bonn\\
Germany}
\email{rapoport@math.uni-bonn.de}

\author[B. Smithling]{Brian Smithling}
\address{University of Toronto\\
Department of Mathematics\\
40 St.\ George St.\\
Toronto, ON  M5S 2E4\\
Canada}
\email{bds@math.toronto.edu}

%    
%    Classification and abstract
%    
\subjclass[2000]{Primary 14G35, 11G18; Secondary 14M15}
\keywords{Shimura variety, local model, affine flag variety}
%%% The \date command has been redefined by Handbook formatting to produce
%%% ``Received by the editors \today''
\date{\today}

\begin{abstract}
We survey the theory of local models of Shimura varieties. In particular, we discuss their definition and illustrate it by examples. We give an overview of the results on their geometry and combinatorics obtained in the last 15 years. We also exhibit their connections to other classes of algebraic varieties. 
\end{abstract}  

\maketitle
% suppress page number on title pagesr
%
% 4. Place the body of your article here. 
% Please use \section as the top level division.
%

%%% Table of contents is commented out -- the titlesec package, which is required
%%% for Handbook formatting, screws it up!
%\tableofcontents

 \thispagestyle{empty}

\section*{Introduction}\label{s:intro}

Local models of Shimura varieties are projective algebraic varieties over the spectrum of 
a discrete valuation ring. Their  singularities are supposed to model the singularities that
arise in the reduction modulo $p$ of Shimura varieties, in the cases where the level structure at $p$ is of parahoric type.
The simplest case occurs for the modular curve with $\Gamma_0 (p)$-level structure.
In this example the local model is obtained by blowing up the projective line $\mP_{\mZ_p}^1$
over $\Spec\mZ_p$ at the origin 
$0$ of the special fiber $\mP_{\mF_p}^1 = \mP_{\mZ_p}^1\times_{\Spec\mZ_p}\Spec\mF_p$. Local models
for Shimura varieties are defined   in terms of linear algebra data inside the product of Grassmann 
varieties,  at least as far as type $A$, or $C$, or some cases of type $D$ are concerned. Another version of these
varieties arises as closures of Schubert varieties inside the Beilinson-Drinfeld-Gaitsgory deformation
of affine flag varieties.  It is the aim of this survey to discuss local models from various points of view, exhibit their connections to other classes of algebraic varieties, and give an overview of the results on them obtained in the last 15 years.

Why does such a survey have a place in the handbook of moduli? The reason is that Shimura varieties are often moduli spaces of abelian varieties
with additional structure. Therefore, determining the singularities of their reduction modulo $p$ is an inherent part of the theory of such moduli spaces. The archetypical example is the Shimura variety attached to the group of symplectic similitudes (and its canonical family of Hodge structures). In this case, the Shimura variety represents the moduli functor on $\mQ$-schemes of isomorphism classes of principally polarized abelian varieties of a fixed dimension, equipped with a level structure. In case the $p$-component of this level structure is of parahoric type, there is an obvious way to extend the moduli functor to a moduli functor on $\mZ_{(p)}$-schemes, which however will have bad reduction, unless the $p$-component of the level structure is {\it hyperspecial}. Local models  then serve to analyze the singularities in the special fibers of the $\mZ_{(p)}$-models thus defined. For instance,  natural questions of a local nature are  whether the $\mZ_{(p)}$-schemes that arise in this way are flat over  $\mZ_{(p)}$, or Cohen-Macaulay, or what the set of branches through a point in the reduction is.  All these questions, exactly because they are  of a local nature,  can be transferred to questions on  the corresponding local models.

We will not give  a sketch of the historical development of the theory here.  We only mention that the origin of these ideas lies in the work  of  Deligne and Pappas \cite{DP}, of Chai and Norman \cite{CN}, and of de Jong \cite{J} on specific Shimura varieties. The definitions of local models in the examples considered in these papers were formalized to some degree in the work of Rapoport and Zink in \cite{R-Z} with the introduction of what were subsequently termed {\it naive local models}. The paper \cite{P1} of Pappas pointed to the fact that naive local models are not always flat. Whereas the examples of Pappas arise due to the fact that the underlying group is non-split (in fact, split only after a ramified extension), it was later pointed out by Genestier \cite{Ge2} that a similar phenomenon also occurs for split orthogonal groups.  This then led to the definition of local models in the papers \cite{P-R1, P-R2, P-R4}, as it is presented here.  The local structure of local models was considered in papers by G\"ortz \cite{Go1, Go2, Go4}, Faltings \cite{Fa1, Fa2},    Arzdorf \cite{A}, Richarz \cite{Ri},  and others. At the same time, the combinatorics of the special fiber of local models (in particular, the $\{\mu\}$-admissible set in the Iwahori-Weyl group), was considered in papers by 
Kottwitz and Rapoport \cite{KR}, Haines and Ng\^o \cite{HN2}, G\"ortz \cite{Go4}, and Smithling \cite{Sm1, Sm2, Sm3, Sm4, Sm5}. Finally, we mention the papers
by Gaitsgory \cite{Ga}, Haines and Ng\^o \cite{HN1}, G\"ortz \cite{Go3}, Haines \cite{H1, H2, HP}, and Kr\"amer \cite{Kr},  addressing the problem of determining  the complex of nearby cycles for local models ({\it Kottwitz conjecture}).

It is remarkable that local models also appear in the study of singularities of   other moduli schemes. In \cite{Ki} Kisin constructs a kind of birational modification scheme of the universal {\it flat}  deformation of a finite flat group scheme over a discrete valuation ring of unequal characteristic $(0, p)$, and shows that the singularities in characteristic $p$ of these  schemes are modeled by certain local models that correspond to Shimura varieties of type $A$. Another context in which local models  appear is in the description of Faltings \cite{Fa3} of the singularities of the moduli space of vector bundles on semi-stable singular algebraic curves. 

The theory of local models falls fairly neatly into two parts. The first part is concerned with the local commutative algebra of local models, and the combinatorics of the natural stratification of their special fibers. This  part of the theory  is surveyed in the present paper. The second part is concerned with the cohomology of sheaves on local models, and will be presented in a sequel to this paper. More precisely, we will survey in a second installment the cohomology of coherent sheaves on local models (and in particular will explain the coherence conjecture of \cite {P-R3}), as well as the cohomology of $\ell$-adic sheaves on local models, and in particular the determination of the complex of nearby cycles. Of course, both parts are interrelated by various links between them, and we will try to make this plain in the sequel to this first installment. 

This survey consists of three parts of a rather distinct nature. In the first part (\S \ref{s:motivationalsection}), we give two  approaches to local models, each with a different audience in mind. It should be pointed out that only one of these approaches is the one with which we actually work, and which relates directly to the theory of Shimura varieties,
especially those which are of PEL type. The other approach points 
to a more general theory and shows the ubiquity of local models in other contexts, 
but is not completely worked out at this stage (although there is relevant work in preparation
by Pappas and Zhu).

In the second part (\S\S \ref{s:examples}--\ref{s:combinatorics}), we give an account of the results  on local models that have been obtained in the last 15 years, and we highlight open questions in this area.

In the third part  (\S\S \ref{s:no}--\ref{s:wc}) we explain the  relation of local models to other classes of algebraic varieties, such as nilpotent orbit closures, matrix equation varieties, quiver Grassmannians, and wonderful compactifications of symmetric spaces, that have been established in some cases. Especially as concerns the last section, this is still largely uncharted territory, which explains why this part is of a more informal nature. 

We are happy to acknowledge the important contributions of K.~Arzdorf, \mbox{C.-L.}~Chai, P.~Deligne, V.~Drinfeld, G.~Faltings, D.~Gaitsgory, A.~Genestier, U.~G\"ortz, T.~Haines, J.~de~Jong, R.~Kottwitz, L.~Lafforgue, P.~Norman, T.~Richarz, and Th.~Zink to the subject of the survey. In addition, we thank U.~G\"ortz, T.~Richarz, J.~Schroer,  and  X.~Zhu for their help with this survey,  G.~Farkas and I.~Morrison for inviting us to include our text  in  their handbook of moduli and for putting on us just the right amount of pressure  for a (un)timely delivery, and the referee for his/her suggestions.
 
\section{An object in search of a definition}\label{s:motivationalsection}

In this motivational section, we sketch two possible approaches to local models. It is the 
first  approach  that is directly related to the original purpose of local models, which is to
construct an elementary projective scheme over the ring of integers of a $p$-adic localization of
the reflex field of a Shimura variety,  whose singularities model those of certain integral
models of a Shimura variety. Unfortunately we cannot make the corresponding definition  in as great a generality as we
would like. It is the second approach which is most easily related to the
theory of algebraic groups. It is also  the most elegant, in the sense that it is uniform. In a preliminary subsection we list the 
formal properties that we have come to expect from local models.

\subsection{Local models in an ideal world}\label{ss:lmideal}

The ideal situation presents itself as follows. Let $F$ be a discretely
valued field. We denote by $\O_F$ its ring of integers and by $k=k_F$ its  residue field which we assume to be perfect.
Let $G$ be a connected reductive group over $F$, and let $\{\mu\}$ be a geometric
conjugacy class of one-parameter subgroups of $G$, defined over an algebraic closure
$\overline{F}$ of $F$. Let $E$ be the field of definition of $\{\mu\}$, a finite extension
of $F$ contained in $\overline{F}$ (the \emph{reflex field} of the pair $(G, \{\mu\})$). Finally, let
$K$ be a parahoric subgroup of $G(F)$ in the sense of \cite{BTII}, see also \cite{T}.
These subgroups are (``up to connected component", see \cite{BTII} for a precise definition) the stabilizers of
points in the Bruhat-Tits building of the group $G(F)$. 
 We denote by $\mathcal G$ the smooth group
scheme over $\mathcal O_F$ with generic fiber $G$ and with connected special fiber such that
$K = \mathcal G (\mathcal O_F)$. The existence of a canonical  group scheme 
$\Gg$ with these properties is one of the main results of  \cite{BTII}.

To these data, one would like to associate the {\it local model}, a projective scheme
$M^{\rm loc} = M^{\rm loc} (G, \{\mu\})_K$ over $\Spec\mathcal O_E$, at least when $\{\mu\}$ is a conjugacy class of minuscule%
\footnote{Recall that a coweight $\mu$ is \emph{minuscule} if $\la \alpha,\mu \ra \in \{-1,0,1\}$ for every root $\alpha$ of $G_{\ol F}$.}
coweights. It should be equipped with an action of
$\mathcal G_{\mathcal O_E}=\mathcal G\otimes_{\mathcal O_F}{\mathcal O_E}$. At least when $\{\mu\}$ is minuscule,
$M^{\rm loc}$ should have the following properties.

\smallskip

\noindent (i) $M^{\rm loc}$ is flat over $\Spec\mathcal O_E$ with generic fiber isomorphic to $G/P_{\mu}$.
Here $G/P_{\mu}$ denotes the variety over $E$ of parabolic subgroups of $G$ of type $\{\mu\}$.

\smallskip

\noindent (ii) There is an identification of the geometric points of the special fiber,
\begin{equation*}
M^{\rm loc} (\overline{k}_E)= \bigl\{\, g\in G(L)/\widetilde{K} \bigm| \widetilde{K} g \widetilde{K}\in {\rm Adm}_{\widetilde{K}} (\{\mu\})\,\bigr\}\ .
\end{equation*}
Here $L$ denotes the completion of the maximal unramified extension of $F$ in $\overline{F}$,
and $\widetilde{K} = \mathcal G (\mathcal O_L)$ the parahoric subgroup of $G(L)$ corresponding to $K$.
Finally,
\begin{equation*}
{\Adm}_{\widetilde{K}}(\{\mu\})\subset\widetilde{K}\backslash G(L)/\widetilde{K}
\end{equation*}
is the finite subset of $\{\mu\}$-admissible elements \cite{R}, cf.\ Definition \ref{def:mu-admissible} below.

\smallskip

\noindent (iii) For any  inclusion of parahoric subgroups $K\subset K'$ of $G (F)$, there should be
a morphism
\begin{equation}
M^{\rm loc}_K\to M^{\rm loc}_{K'}\  ,
\end{equation}
which induces the identity (via (i)) on the generic fibers. For a central isogeny $G\to G'$,
and compatible conjugacy classes  $\{\mu\}$ and $\{\mu'\}$, and compatible parahoric subgroups $K\subset G(F)$,  resp.\
$K'\subset G'(F)$, one should have an identification
\begin{equation}
M^{\rm loc} (G, \{\mu\})_K = M^{\rm loc} (G', \{{\mu}'\})_{K'}\  .
\end{equation}
More generally, if $\varphi \colon G\to G'$ is a homomorphism, and $\{\mu'\} = \{\varphi\circ\mu\}$, and if
$\varphi (K)\subset K'$, there should be a morphism
\begin{equation}
M^{\rm loc} (G, \{\mu\})_K\to M^{\rm loc} (G', \{{\mu}'\})_{K'}\otimes_{\mathcal O_{E'}}\mathcal O_E\  ,
\end{equation}
which induces in the generic fiber the natural morphism $G/P_{\mu}\to (G'/P_{\mu'})\otimes_{E'} E$. Here $E'\subset E$ is the
reflex field of $(G', \{\mu'\})$. 

\smallskip

\noindent (iv) Let $F'$ be a finite extension of $F$ contained in $\overline{F}$. Let
$G' = G\otimes_F F'$, and regard $\{\mu\}$ as a geometric conjugacy class of
one-parameter subgroups of $G'$. Let $K'\subset G'(F')$ be a parahoric subgroup
with $ K = K'\cap G (F)$. Note that the reflex field of $(G', \{\mu\})$ is equal to
$E' = F' E$. Under these circumstances one should expect a morphism of local models
\begin{equation}\label{basechlocmod}
M^{\rm loc}(G, \{\mu\})_K\otimes_{\mathcal O_E}\mathcal O_{E'}\to M^{\rm loc} (G', \{\mu\})_{K'}\  ,
\end{equation}
which induces the natural morphism
\begin{equation*}
(G/P_{\mu})\otimes_E E'\to G' / P'_{\mu}\  
\end{equation*}
in the generic fibers. Furthermore, if $F'/F$ is unramified, then the morphism (\ref{basechlocmod}) should be an isomorphism.

\smallskip

\noindent (v) Suppose that $G=\prod_{i=1}^nG_i$, $K=\prod_{i=1}^n K_i$, and $\mu=\prod_{i=1}^n \mu_i$ are all products.
Then ${\mathcal G}=\prod_{i=1}^n {\mathcal G}_i$ and the reflex fields $E_i$, $1\leq i\leq n$, generate the reflex field $E$. We  then
expect an equivariant isomorphism of local models
\begin{equation}\label{prodlocmod}
M^{\rm loc}(G, \{\mu\})_K\xrightarrow{\sim} \prod\nolimits_i M^{\rm loc} (G_i, \{\mu_i\})_{K_i}\otimes_{\mathcal O_{E_i}}\mathcal O_{E}\  ,
\end{equation}
which induces the natural isomorphism
\begin{equation*}
(G/P_{\mu})= \prod\nolimits_i (G_i / P_{\mu_i})\otimes_{E_i} E 
\end{equation*}
in the generic fibers. 

\smallskip

Here we should point out that it is not clear that the above listed properties
are enough to characterize the local models $M^{\rm loc}(G, \{\mu\})_K$ 
up to isomorphism. In fact,
a general abstract (i.e.\ ``group theoretic") definition of local models is still lacking,
although, as we will explain in \S 1.3, there is now some  progress on this problem.

We now sketch two different approaches 
to  the concept of local models.

%Unfortunately, we are not yet able to define local models in general with these
%properties. Instead, we now sketch two different possible approaches 
%to the concept of local models.

\subsection{Local models arising from Shimura varieties}\label{ss:Shimuralocmod}

Let $Sh_{\bold K} = Sh ({\bold G}, \{h\}, {\bold K})$ denote a Shimura variety \cite{D1} attached to the triple consisting of  a connected reductive group
$\bold G$ over $\mQ$, a family of Hodge structures $h$ and a compact open subgroup 
$\bold K\subset \bold G(\mA_f)$. We fix a prime number $p$ and assume that $\bold K$ factorizes as
$\bold K = K^p\cdot K_p\subset {\bold G}(\mA_f^p)\times {\bold G} (\mQ_p)$. In fact we assume in addition that $K=K_p$ is
a parahoric subgroup of $\bold G (\mQ_p)$.

Let $E\subset\mC$ denote the reflex field of  $({\bold G}, \{h\})$, i.e.,\ the field of definition 
of the geometric conjugacy class of one-parameter subgroups $\{\mu\} = \{\mu_h\}$
attached to $\{h\}$, cf.~\cite{D1}. Then $E$ is a subfield of the field of algebraic numbers $\overline{\mQ}$, of finite degree over $\mQ$.
Fixing an embedding $\overline{\mQ}\to\overline{\mQ}_p$ determines a place $\wp$
of $E$ above $p$. We denote by the same symbol the canonical model of $Sh_{\bold K}$ over $E$
and its base change to $E_{\wp}$. It is then an interesting problem to define a suitable
model $\mathcal S_{\bold K}$ of $Sh_{\bold K}$ over $\Spec\mathcal O_{E_{\wp}}$. Such a model should be projective if $Sh_{\bold K}$
is (which is the case when $\bold G_{\rm ad}$ is $\mQ$-anisotropic), and should always have 
manageable singularities. In particular, it should be flat over $\Spec\mathcal O_{E_{\wp}}$, and its local 
structure should only depend on the localized group $G = \bold G\otimes_{\mQ}\mQ_p$, the 
geometric conjugacy class $\{\mu\}$ over $\overline{\mQ}_p$, and the parahoric subgroup
$K = K_p$ of $G (\mQ_p)$. Note that, due to the definition of a Shimura variety, the conjugacy
class $\{\mu\}$ is minuscule. 

More precisely, we expect the local model 
$M^{\rm loc} (G, \{\mu\})_K$ to model the singularities of the model $\mathcal S_{\bold K}$, in the following
sense. We would like to have a {\it local model diagram} of $\mathcal O_{E_\wp}$-schemes, in the sense of \cite{R-Z}, 
$$\xymatrix{
& {\widetilde{\mathcal S}}_{\bold K}\ar[ld]_{\pi}\ar[rd]^{\widetilde{\varphi}} & \\
{\quad\quad \mathcal S_{\bold K}\quad\quad} & & {M^{\rm loc} (G, \{\mu\})_K}\, ,
 } $$
in which $\pi$ is a principal homogeneous space (p.h.s.) under the algebraic group $\mathcal G_{\mathcal O_{E_\wp}}=\mathcal G\otimes_{\mZ_p}\mathcal O_{E_\wp}$,  and
in which $\widetilde{\varphi}$ is smooth of relative dimension $\dim G$. Equivalently, using
the language of algebraic stacks, there should be a smooth morphism  of algebraic stacks of relative dimension
$\dim G$ to the stack quotient,
\begin{equation*}
\mathcal S_{\bold K}\to [M^{\rm loc} (G, \{\mu\})_K/\mathcal G_{\mathcal O_{E_\wp}}]\, .
\end{equation*}
In particular, for every geometric point $x\in \mathcal S_{\bold K} (\overline{\mF}_p)$, there exists a geometric
point $\overline{x}\in M^{\rm loc} (G, \{\mu\})_K (\overline{\mF}_p)$, unique up to the action of
$\mathcal G (\overline{\mF}_p)$, such that the strict henselizations of $\mathcal S_{\bold K}$ at $x$ and of $M^{\rm loc}$
at $\overline{x}$ are isomorphic.

Note that the generic fiber $G/P_{\mu} = M^{\rm loc} (G, \{\mu\})_K\otimes_{\mathcal O_{E_\wp}} E_\wp$
is nothing but the {\it compact dual} of the hermitian symmetric domain corresponding to the Shimura
variety $Sh (\bold G, \{h\},\bold  K)$ (after extending scalars from $E$ to $E_\wp$). From this perspective,
the local model $M^{\rm loc} (G, \{\mu\})_K$ is an $\mathcal O_{E_\wp}$-integral model of the compact
dual of the Shimura variety $Sh (\bold G, \{h\},\bold  K)$.

The problems of defining a model of $Sh_{\bold K}$ over $\mathcal O_{E_\wp}$ and of defining a local model
$M^{\rm loc} (G, \{\mu\})_K$ are closely intertwined (although not completely equivalent, as the
example of a ramified unitary group shows \cite{P-R4}). Let us explain this and also briefly review
the general procedure for the construction of   local models $M^{\rm loc} (G, \{\mu\})_K$
in some  cases where the Shimura variety is of PEL type.
Recall that in the PEL cases treated in \cite{R-Z} one first constructs a ``naive"
integral model $\mathcal S^{\rm naive}_{\bold K}$ of the Shimura variety ${Sh}_{\bold K}$; this
is given by a moduli space description and affords a corresponding ``naive local model" $M^{\rm naive}$
together with a smooth morphism
\begin{equation*}
\mathcal S^{\rm naive}_{\bold K}\to [M^{\rm naive}/\mathcal G_{\mathcal O_{E_{\wp}}}].
\end{equation*}
As we mentioned in the introduction,   these naive models $M^{\rm naive}$ and $\mathcal S_{\bold K}$
are often not even flat over $\mathcal O_{E_{\wp}}$ \cite{P1,Ge2}. Then, in most cases, the (non-naive) local model is a $\mathcal G_{\mathcal O_{E_{\wp}}}$-invariant closed subscheme $M^{\rm loc}:=M^{\rm loc} (G, \{\mu\})_K$ of $M^{\rm naive}$ with the same 
generic fiber which is  brutally defined as the flat closure. 
The general idea then is that, from $M^{\rm loc} (G, \{\mu\})_K$, one  also
obtains a good (i.e at least flat) integral model $\mathcal S_{\bold K}$ 
of the Shimura variety via the cartesian diagram
\[
   \xymatrix{
      \mathcal S_{\bold K} \ar[r] \ar[d]
         & [M^{\rm loc}  /\mathcal G_{\mathcal O_{E_\wp}}] \ar[d] \\
      \mathcal S^{\rm naive}_{\bold K} \ar[r]
         & [M^{\rm naive}  /\mathcal G_{\mathcal O_{E_\wp}}].
   }
\]
% \begin{equation*}
% \begin{matrix} \mathcal S_{\bold K} &\xrightarrow {\ } &[M^{\rm loc}  /\mathcal G_{\mathcal O_{E_\wp}}]\\
% \downarrow &&\downarrow \\
% \mathcal S^{\rm naive}_{\bold K} &\xrightarrow {\ } &[M^{\rm naive}  /\mathcal G_{\mathcal O_{E_\wp}}].
% \end{matrix}
% \end{equation*}
Unfortunately, in general,  the schemes $M^{\rm loc} (G, \{\mu\})_K$ and  $\mathcal S_{\bold K}$,
 do not have a reasonable moduli theoretic interpretation. Nevertheless, there are  still
 (proven or conjectural) moduli  descriptions  in many interesting cases
 \cite{P1,Go1, Go2, P-R1, P-R2, P-R4}. All these issues are explained 
in more detail in \S \ref{s:examples}.

We mention here that taking the expected functorialities
(i)--(v) of local models into account, we may,  in constructing a local model for the data 
$( G, \{\mu\}, K)$,  make the following hypotheses. We may assume that the adjoint group of $G$
is simple; we may extend scalars to an unramified extension $F$ of $\mQ_p$. If we insist
that $\{\mu\}$ be minuscule, this  reduces the number of possible cases to an essentially finite list.
Let us explain this in more detail. We assume  $G_{\rm ad}$ is  simple and denote by $\mu_{\rm ad}$ the corresponding minuscule cocharacter of $G_{\rm ad}(\ol\mQ_p)$.
Let $\mQ_p^{\rm un}$ be the completion of
the maximal unramified extension of $\mQ_p$; by Steinberg's theorem 
every reductive group over $\mQ_p^{\rm un}$ is quasi-split.
We can write 
$$
{G_{\rm ad}}_{/\mQ_p^{\rm un}}= {\rm Res}_{L/\mQ^{\rm un}_p} (H)\  , \quad \mu=\{\mu_\sigma\}_{\sigma:\, L\rightarrow \ol\mQ_p}\  ,
$$
where $H$ is absolutely simple adjoint over $L$,  $\sigma$ runs over embeddings 
of $L$ over $ \mQ^{\rm un}_p$, and $\mu_\sigma$ are minuscule cocharacters of $H(\ol\mQ_p)$.
The group $H$ over $L$ is also quasi-split.

The possible cases for the pairs $(H, \mu_\sigma)$ are given in
the table below  which can be obtained by combining the table of types of quasi-split, residually split, absolutely
simple groups from 
\cite[p.\ 60--61]{T} with the lists of minuscule coweights in \cite{B} which are dominant relative to the choices of positive roots in \cite{B}.  In the local Dynkin diagrams, $h$ denotes a hyperspecial vertex, $s$ a special (but not hyperspecial) vertex, and $\bullet$ a nonspecial vertex.  We refer to \cite[1.8]{T} for the explanation of the notation in the diagrams.  There are $n+1$ vertices in each diagram that explicitly depends on $n$, i.e.\ aside from the diagrams for $A_1$, $A_2^{(2)}$, $D^{(3)}_4$, $E_6$, $E^{(2)}_6$, and $E_7$.

\bigskip
\bigskip
\bigskip
 
%\vfill\eject

{\small
\setlongtables
\setlength{\bigstrutjot}{\bigskipamount}
\begin{longtable}{|c|c|c|}
\hline
Type of $H(L)$ & \  Local Dynkin diagram & 
   \setlength{\bigstrutjot}{2pt} \twolinestight{\bigstrut[t] Nonzero dominant minuscule}{\bigstrut[b] coweights for $H(\ol\mQ_p)$}   \\
\hline \hline
$A_1$\bigstrut &
   $\xy
      (-2.5,0)*{h}="h1";
      (2.5,0)*{h}="h2";
      {\ar@{-} "h1";"h2"};
      {\ar@{-} "h1";"h2"<.35pt>};
      {\ar@{-} "h1";"h2"<-.35pt>};
      {\ar@{-} "h1";"h2"<.7pt>};
      {\ar@{-} "h1";"h2"<-.7pt>};
      {\ar@{-} "h1";"h2"<1.05pt>};
      {\ar@{-} "h1";"h2"<-1.05pt>};
   \endxy$ & $\varpi^\vee_1$\\
\hline
$A_n$, $n \geq 2$ &
      $\xy
      (0,2.5)*{h}="h1";
      (0,2.5)*{\bigstrut};
      (-7.5,-2.5)*{h}="h2";
      (-2.5,-2.5)*{h}="h3";
      (2.5,-2.5)*{h}="h4";
      (7.5,-2.5)*{h}="h5";
      "h5"*{\bigstrut};
      {\ar@{-} "h1";"h2"};
      {\ar@{-} "h2";"h3"};
      {\ar@{.} "h3";"h4"};
      {\ar@{-} "h4";"h5"};
      {\ar@{-} "h5";"h1"}
      \endxy$
   & $\varpi^\vee_i$, $1\leq i\leq n$  \\
\hline
\bigstrut $A_2^{(2)}$ ($C$-$BC_1$) &
   $\xy
      (-2.5,0)*{s}="h1";
      (2.5,0)*{s}="h2";
      (0,0)*{<};
      {\ar@{-} "h1";"h2"};
      {\ar@{-} "h1";"h2"<.35pt>};
      {\ar@{-} "h1";"h2"<-.35pt>};
      {\ar@{-} "h1";"h2"<.7pt>};
      {\ar@{-} "h1";"h2"<-.7pt>};
      {\ar@{-} "h1";"h2"<1.05pt>};
      {\ar@{-} "h1";"h2"<-1.05pt>};
   \endxy$ & $\varpi^\vee_1,\varpi^\vee_2$\\
\hline
\twolinestight
   {\bigstrut[t] $A^{(2)}_{2n}$ ($C$-$BC_n$),}
   {\bigstrut[b] $n \geq 2$} & 
      $\xy 
      (-10,0)*{s}="s1";
      (-5,0)*{\bullet};
      (0,0)*{\bullet};
      (5,0)*{\bullet};
      (10,0)*{\bullet};
      (15,0)*{s}="s2";
      (-7.5,0)*{<};
      (12.5,0)*{<};
      {\ar@2{-} "s1";(-5,0)};
      {\ar@{-} (-5,0);(0,0)};
      {\ar@{.} (0,0);(5,0)};
      {\ar@{-} (5,0);(10,0)};
      {\ar@2{-} (10,0);"s2"};
      \endxy$
   & $\varpi^\vee_i$, $1\leq i\leq 2n$ \\
\hline
\twolinestight
   {\bigstrut[t] $A^{(2)}_{2n-1}$ ($B$-$C_n$),}
   {\bigstrut[b] $n \geq 3$} &
      $\xy 
      (-10,0)*{\bullet}="v1";
      (-5,0)*{\bullet}="v2";
      (0,0)*{\bullet}="v3";
      (5,0)*{\bullet}="v4";
      (10,0)*{\bullet}="v5";
      (15,3)*{s}="s1";
      (15,-3)*{s}="s2";
      (-7.5,0)*{>};
      {\ar@2{-} (-10,0);(-5,0)};
      {\ar@{-} (-5,0);(0,0)};
      {\ar@{.} (0,0);(5,0)};
      {\ar@{-} (5,0);(10,0)};
      {\ar@{-} (10,0);"s1"};
      {\ar@{-} (10,0);"s2"};
      \endxy$
   & $\varpi^\vee_i$, $1\leq i\leq 2n-1$  \\
\hline\hline
$B_n$, $n \geq 3$ & 
   $\xy 
      (-10,0)*{\bullet}="v1";
      (-5,0)*{\bullet}="v2";
      (0,0)*{\bullet}="v3";
      (5,0)*{\bullet}="v4";
      (10,0)*{\bullet}="v5";
      (15,3)*{h}="s1";
      "s1"*{\bigstrut};
      (15,-3)*{h}="s2";
      "s2"*{\bigstrut};
      (-7.5,0)*{<};
      {\ar@2{-} (-10,0);(-5,0)};
      {\ar@{-} (-5,0);(0,0)};
      {\ar@{.} (0,0);(5,0)};
      {\ar@{-} (5,0);(10,0)};
      {\ar@{-} (10,0);"s1"};
      {\ar@{-} (10,0);"s2"};
      \endxy$
   & $\varpi^\vee_1$ \\ 
\hline\hline
$C_n$, $n \geq 2$\bigstrut  &
      $\xy 
      (-10,0)*{h}="s1";
      (-5,0)*{\bullet};
      (0,0)*{\bullet};
      (5,0)*{\bullet};
      (10,0)*{\bullet};
      (15,0)*{h}="s2";
      (-7.5,0)*{>};
      (12.5,0)*{<};
      {\ar@2{-} "s1";(-5,0)};
      {\ar@{-} (-5,0);(0,0)};
      {\ar@{.} (0,0);(5,0)};
      {\ar@{-} (5,0);(10,0)};
      {\ar@2{-} (10,0);"s2"};
      \endxy$
   & $\varpi^\vee_n$    \\ 
\hline\hline
$D_{n}$, $n \geq 4$ & 
      $\xy 
      (-10,3)*{h}="h1";
      "h1"*{\bigstrut};
      (-10,-3)*{h}="h2";
      (-5,0)*{\bullet}="v2";
      (0,0)*{\bullet}="v3";
      (5,0)*{\bullet}="v4";
      (10,0)*{\bullet}="v5";
      (15,3)*{h}="s1";
      (15,-3)*{h}="s2";
      "s2"*{\bigstrut};
      {\ar@{-} "h1";(-5,0)};
      {\ar@{-} "h2";(-5,0)};
      {\ar@{-} (-5,0);(0,0)};
      {\ar@{.} (0,0);(5,0)};
      {\ar@{-} (5,0);(10,0)};
      {\ar@{-} (10,0);"s1"};
      {\ar@{-} (10,0);"s2"};
      \endxy$
   & $\varpi^\vee_1, \varpi^\vee_{n-1}, \varpi^\vee_{n}$  \\ 
\hline
\twolinestight
   {\bigstrut[t] $D^{(2)}_{n+1}$ ($C$-$B_n$),}
   {\bigstrut[b] $n \geq 2$} & 
      $\xy 
      (-10,0)*{s}="s1";
      (-5,0)*{\bullet};
      (0,0)*{\bullet};
      (5,0)*{\bullet};
      (10,0)*{\bullet};
      (15,0)*{s}="s2";
      (-7.5,0)*{<};
      (12.5,0)*{>};
      {\ar@2{-} "s1";(-5,0)};
      {\ar@{-} (-5,0);(0,0)};
      {\ar@{.} (0,0);(5,0)};
      {\ar@{-} (5,0);(10,0)};
      {\ar@2{-} (10,0);"s2"};
      \endxy$
   & $\varpi^\vee_1, \varpi^\vee_{n}, \varpi^\vee_{n+1}$  \\ 
\hline
$D^{(3)}_4$, $D^{(6)}_4$ ($G_2^\mathrm{I}$)\bigstrut &
   $\xy
   (-5,0)*{s}="s";
   (0,0)*{\bullet};
   (5,0)*{\bullet};
   (2.5,0)*{<};
   {\ar@{-} "s";(0,0)};
   {\ar@3{-} (0,0);(5,0)};
   \endxy$ & $\varpi^\vee_1,\varpi^\vee_3,\varpi^\vee_4$\\
\hline\hline
$E_6$ & 
   $\xy
      (-10,0)*{h}="h1";
      (-5,0)*{\bullet};
      (0,0)*{\bullet};
      (5,3)*{\bullet};
      (10,6)*{h}="h2";
      "h2"*{\bigstrut};
      (5,-3)*{\bullet};
      (10,-6)*{h}="h3";
      "h3"*{\bigstrut};
      {\ar@{-} "h1";(-5,0)};
      {\ar@{-} (-5,0);(0,0)};
      {\ar@{-} (0,0);(5,3)};
      {\ar@{-} (5,3);"h2"};
      {\ar@{-} (0,0);(5,-3)};
      {\ar@{-} (5,-3);"h3"};
   \endxy$
   & $\varpi^\vee_1, \varpi^\vee_6$ \\ 
\hline
$E^{(2)}_6$ $(F_4^\mathrm{I})$ \bigstrut &
   $\xy
   (-10,0)*{s}="s";
   (-5,0)*{\bullet};
   (0,0)*{\bullet};
   (5,0)*{\bullet};
   (10,0)*{\bullet};
   (2.5,0)*{<};
   {\ar@{-} "s";(-5,0)};
   {\ar@{-} (-5,0);(0,0)};
   {\ar@2{-} (0,0);(5,0)};
   {\ar@{-} (5,0);(10,0)}
   \endxy$ &
   $\varpi^\vee_1, \varpi^\vee_6$\\
\hline
$E_7$ & 
   $\xy
      (-15,0)*{h}="h1";
      "h1"*{\bigstrut};
      (-10,0)*{\bullet};
      (-5,0)*{\bullet};
      (0,0)*{\bullet};
      (0,5)*{\bullet};
      (0,5)*{\bigstrut};
      (5,0)*{\bullet};
      (10,0)*{\bullet};
      (15,0)*{h}="h2";
      {\ar@{-} "h1";(-10,0)};
      {\ar@{-} (-10,0);(-5,0)};
      {\ar@{-} (-5,0);(0,0)};
      {\ar@{-} (0,0);(5,0)};
      {\ar@{-} (0,0);(0,5)};
      {\ar@{-} (5,0);(10,0)};
      {\ar@{-} (10,0);"h2"};
   \endxy$
   & $\varpi^\vee_1$ \\ 
\hline 
\end{longtable}
}

Note that the minuscule coweights are for $H(\ol\mQ_p)$
and so they only depend on the absolute type over $\ol\mQ_p$.
There are no nonzero minuscule coweights for $E_8$, $F_4$, $G_2$ types.
\begin{comment}In addition to fixing  a field $L/\mQ_p^{\rm un}$ and a collection 
 of such pairs,   the data needed to define a local model 
over $\mQ_p^{\rm unr}$  also includes
a choice of a parahoric subgroup $\wt K$  of the group $G(\mQ_p^{\rm unr})$.
%(At least, when $G$ is simply connected, conjugacy classes of such parahoric subgroups are in one-to-one 
%correspondence with subsets of nodes of the corresponding local Dynkin diagram.)
\end{comment}
Of course, there is no simple description
of the local model for ${\rm Res}_{L/\mQ^{\rm un}_p}(H)$ in terms of a local model
for $H$. For example, see the case of $H=GL_n$ in \S\ref{s:examples}.\ref{ss:ResGL_n} below. However, we expect that most properties
of local models  for  a group which is the restriction of scalars ${\rm Res}_{L/\mQ^{\rm un}_p}(H)$ 
will only depend on $H$, the degree of $L$ over  $\mQ_p^{\rm un}$, the 
combinatorial data describing $\{\mu_\sigma\}_\sigma$ and the type (conjugacy class)
of the parahoric subgroup $K\subset H(L)$ (and not on the particular 
choice of the field $L$).

Recall that to each such pair $(H, \mu)$ with $H$ absolutely simple adjoint and $\mu$ minuscule as above, %$(\hbox{\rm type}, \hbox{\rm minuscule coweight})$
we associate a homogeneous space $H/P_\mu$. Following Satake, in \cite[1.3]{D2} Deligne studies
faithful symplectic representations $\rho\colon H'\to GSp_{2g}$, where
$H'\to H$ is a central isogeny, with the property  that the coweight $\rho_{\rm ad}\circ\mu_{\rm ad}$ 
is the (unique) minuscule coweight $\varpi^\vee_g$ in type $C_g$.
Such symplectic representations exist for all the pairs in the table,  except
for those corresponding to exceptional groups.  
Hence, for all classical pairs, we can obtain an embedding of $H/P_\mu$ in 
the  Grassmannian of Lagrangian subspaces of rank $g$ in symplectic $2g$-space.
As we will see in the rest of the paper, the local model is often defined using such an
embedding.
By  loc.\ cit.,\  Shimura varieties of ``abelian type" \cite{M,M2} produce pairs that support  such symplectic representations.
Among them, the pairs $(B_n, \varpi^\vee_1)$, $(D_n, \varpi^\vee_1)$, $(D^{(2)}_n, \varpi^\vee_1)$, $(D^{(3)}_4, \varpi^\vee_1)$, $(D^{(6)}_4, \varpi^\vee_1)$
do not appear,  when we are just considering
Shimura varieties of PEL type. For these pairs, the corresponding homogeneous spaces
$H/P_\mu$ are forms of quadric hypersurfaces 
in projective space. So far, local models involving these pairs  and the exceptional 
pairs have not been the subject of a systematic investigation. The construction in \S \ref{s:motivationalsection}.\ref{ss:BGm} 
applies to some of these local models, but we will  otherwise
  omit their discussion in this survey.

 \begin{eg}
 Let us consider the Siegel case, i.e.,\ the Shimura variety of principally polarized abelian varieties of dimension $g$ with level $\bold K$-structure, where the $p$-component $K_p$ of $\bold K$ is the 
 parahoric subgroup of $Gp_{2g}(\mQ_p)$ which is the stabilizer of a selfdual periodic lattice chain $\L$ in the standard symplectic vector space of dimension
 $2g$ over $\mQ_p$.  In this case the Shimura field is equal to $\mQ$, and  a model $\mathcal S_{\bold K}$ over $\mZ_{(p)}$ is given as the moduli scheme  of principally polarized chains of abelian varieties of dimension $g$ of type corresponding to $\L$,  with a level structure prime to $p$. In this case, the local model is given inside the product  of  finitely many copies of the Grassmannian of subspaces of dimension $g$ in a $2g$-dimensional vector space,   which satisfy two conditions: a periodicity condition, and a self-duality condition. This example is discussed in \S \ref{s:examples}.\ref{ss:GSp_2g}. 
 \end{eg}
 
 \begin{eg}\label{eg:picard}
Let us consider Shimura varieties related to the   Picard moduli schemes of principally polarized abelian varieties of
dimension $n$ with complex multiplication of the ring of integers $\mathcal O_k$ in an imaginary-quadratic field $k$ of signature $(r, s)$
(cf.\ \cite[ \S 4]{KuR} for precise definitions). Here the $\mQ_p$-group $G$ is the group of unitary
similitudes for the quadratic extension $k\otimes\mQ_p$ of $\mQ_p$. Three alternatives present themselves.

\smallskip

\noindent (i)\emph{ $p$ splits in $k$}. Then $G\simeq GL_n\times\mG_m$, and $\{\mu\}$ is the conjugacy class of a cocharacter of the form
$\bigl((1^{(r)}, 0^{(s)}); 1\bigr)$.  Here, for $n = r + s$, we write $(1^{(r)},0^{(s)})$ for the cocharacter
\[
   x \mapsto \diag(\underbrace{x,\dotsc,x}_r,\underbrace{1,\dotsc,1}_s)
\]
of $GL_n$.  The parahoric subgroup $K_p$ is of the form $K_p^0\times\mZ_p^\times$,
where $K_p^0$  is a parahoric supgroup of $GL_n (\mQ_p)$.

\smallskip

\noindent (ii) \emph{$p$ is inert  in $k$}. Then $G$ becomes isomorphic to $ GL_n\times\mG_m$
after the unramified base extension $\otimes_{\mQ_p} k_p$. Hence, by the expected general property (iv)
of local models (which is true in the case at hand), the local models in cases (i) and (ii) become
isomorphic after extension of scalars from $E_\wp$ to $E'_\wp = k_p\cdot E_\wp$. Note that, if
$r\neq s$, then $E_\wp$ can be identified with $k_p$, and hence $E'_\wp = E_\wp$.

\smallskip

\noindent (iii) \emph{$p$ is ramified in $k$}. Again $G\otimes_{\mQ_p} k_p =  GL_n\times\mG_m$. There is
a morphism of local models
\begin{equation*}
M^{\rm loc} (G, \{\mu\})_K\otimes_{\mathcal O_{E_\wp}}\mathcal O_{E'_\wp}\to M^{\rm loc} ( GL_n\times\mG_m , \{\mu\})_{K'}\, ,
\end{equation*}
for any parahoric subgroup $K'\subset GL_n (k_p)\times k_p^\times$ with intersection $K$ with $G (\mQ_p)$.
However, in general this is not an isomorphism.

The Picard moduli problems lead to  local models defined in terms of linear algebra, similar to the Siegel case above. The local models relating to the first two cases   are  discussed in \S \ref{s:examples}.\ref{ss:GL_n}; the local models of the last case  is discussed in \S \ref{s:examples}.\ref{ss:GU_n}.

\end{eg}

As is apparent from this brief discussion, the definitions of the local models related to the last two kinds of Shimura varieties strongly use the natural representations of the classical groups in question (the group of symplectic similitudes in the Siegel case, the general linear group in the Picard case for unramified $p$, and the group of unitary similitudes in the Picard case for ramified $p$). They are therefore
not purely group-theoretical.  In the next section, we will give,
in some case, a purely group-theoretical construction of local models.

\subsection{Local models in the Beilinson-Drinfeld-Gaitsgory style}\label{ss:BGm}

The starting point of the construction is a globalized version of the affine Grassmannian as in \cite{BD}.
Let $\mathcal O$ be a complete discrete valuation ring, with fraction field $F$ and residue field $k$. Let $X=\Spec  \mathcal O[t]$ be the affine line over $\mathcal O$. Let $G$ be a split reductive algebraic group. 
We consider the following functor on $(Sch/X)$. Let  $S\in (Sch/X)$, with structure morphism $y\colon S\to X$, and define
\begin{equation}\label{grGX}
   Gr_{G, X}(S)= \biggl\{\,\twolinestight{iso-classes of pairs}{$( \mathcal F, \beta)$} \biggm|
      \twolinestight{$\mathcal F$ a $G$-bundle on $X\times S$,}
            {$\beta$ a trivialization of $\mathcal F |_{ (X\times S)\setminus\Gamma_y}$}\,\biggr\}\, . 
\end{equation}
Here $\Gamma_y\subset X\times S$ denotes the graph of $y$, and the fiber products are over $\Spec\mathcal O$. 

Then $Gr_{G, X}$ is representable by an ind-scheme over $X$.  The relation of this ind-scheme to the usual affine Grassmannian is as follows.  

Recall that to $G$ and any field $\kappa$, there is associated   its positive loop group 
$L^+ G $ over $\kappa$,   its loop group $L G $, and its affine Grassmannian $Gr_G = L G/L^+ G$ (quotient
of $fpqc$-sheaves on $\kappa$-schemes). Here $L^+ G$ is the affine group scheme on $\Spec \kappa$
representing the functor on $\kappa$-algebras
\begin{equation*}
R\mto L^+ G (R) = G ( R[[T]])\, ,
\end{equation*}
and $L G$ is the ind-group scheme over $\Spec \kappa$ representing the functor
\begin{equation*}
R\mto L G (R) = G \big(R((T))\big)\, ,
\end{equation*}
and $Gr_G$ is the ind-scheme over $\Spec \kappa$ representing the functor
\[
   R\mto Gr_G (R) = \biggl\{\,\twolinestight{iso-classes of pairs}{$(\mathcal F, \beta)$} \biggm|
   \twolinestight{$\mathcal F$ a $G$-bundle on $\Spec R[[T]]$,}
   {$\beta$ a trivialization of $\mathcal F|_{\Spec R((T))}$}\,\biggr\}\, ,
\]
comp.\ \cite{B-L}, cf.\ also \S \ref{s:afv}.\ref{ss:afv} below. When 
we wish to emphasize that we are working over the field $\kappa$, we will denote
the affine Grassmannian  by $Gr_{G, \kappa}$.

\begin{lemma}[Gaitsgory {\cite[Lem.\ 2]{Ga}}]\label{globalGr}
Let $x\in X(\kappa)$, where $\kappa$ is either the residue field of $\mathcal O$, or the fraction field of $\mathcal O$, and identify the completed local 
ring $\mathcal O_x$ of $X\times \Spec \kappa$ with $\kappa[[T]]$, using the local parameter $T=t-x$. Then the restriction morphism induces an isomorphism of ind-schemes over $\Spec \kappa$,  
\begin{equation*}
i_x^\ast \colon Gr_{G, X}\times_{X, x} \Spec \kappa\to Gr_{G, \kappa}\, .
\end{equation*}
Here $Gr_{G, \kappa}$ denotes the affine Grassmannian of $G$ over $\kappa$. \qed
\end{lemma}

%For $x\in X(k)$, let $Gr_{G, X\setminus \{x\}}$ denote the pre-image of $X\setminus \{x\}$ in
%$Gr_{G, X}$. Similarly, let $Gr_{G, x} = i_x^\ast (Gr_{G, X})$.

We next  construct a degeneration of $Gr_{G, F}$ to the affine flag
variety $Fl_{G, k} = L G/\B$ over $k$, where $\B$ denotes the Iwahori subgroup scheme of $L^+ G$
given as the inverse image under the reduction map of a fixed Borel subgroup $B$ of $G$,
 \[
      \xymatrix@R=2.5ex{
         L^+ G \ar[r] \ar@{}[d] |*{\rotatebox{90}{$\subset$}}&
           G \ar@{}[d] |*{\rotatebox{90}{$\subset$}}\\
        \B \ar@{-->}[r] &
           B.
      }
   \]
Denote by  $0\in X(\mathcal O)$ the zero section. Let $Fl_{G, X}$ be the ind-scheme over $X$ which represents the following
functor on $X$-schemes,
\begin{equation}
   S\mto \biggl\{\,\text{iso-classes of triples } (\mathcal F, \beta, \varepsilon)
   \biggm|
   \twolinestight{$(\mathcal F, \beta)\in Gr_{G, X} (S)$,}
      {$\varepsilon$ a reduction of $\mathcal F|_{\{0\}\times S}$ to $B$}\,\biggr\}\, .
\end{equation}
Let $\pi_X \colon Fl_{G, X}\to Gr_{G, X}$ be the forgetful morphism, which is a smooth proper
morphism with typical fiber $G/B$. 

Now fix a uniformizer $\pi\in\mathcal O$. We denote by $\delta$ the section of $X$ over $\mathcal O$ defined by $\delta^*(t)=\pi$. Let $Fl_{G, \mathcal O}$, resp.\ $Gr_{G, \mathcal O}$,  be the pull-back via $\delta$ of $Fl_{G, X}$, resp.\ $Gr_{G, X}$, 
to $\Spec\mathcal O$, and let  
\begin{equation}
\pi_{\mathcal O}\colon Fl_{G, \mathcal O}\to Gr_{G, \mathcal O}
\end{equation}
 be the pull-back of $\pi_X$. Note that the section $\delta$ gives by Lemma \ref{globalGr} identifications of the generic fiber of $Gr_{G, \mathcal O}$ with $Gr_{G, F}$, and of the special fiber of $Gr_{G, \mathcal O}$ with $Gr_{G, k}$. 

\begin{lemma} \label{lemma135} The morphism $\pi_{\mathcal O}$ induces
\begin{altitemize}
\item over $F$ a canonical isomorphism
\begin{equation*}
Fl_{G, \mathcal O}\times \Spec F\simeq Gr_{G, F}\times G/B\ ,
\end{equation*}
\item over $k$ a canonical isomorphism
\begin{equation*}
Fl_{G, \mathcal O}\times \Spec k\simeq Fl_{G, k}\ .
\end{equation*}
\end{altitemize}
\end{lemma}

\begin{proof}(cf.\ {\cite[Prop. 3]{Ga}}) If $S$ is a $F$-scheme, then 
$$X\times S\setminus\Gamma_y = \Spec_S\mathcal O_S[t, (t-\pi)^{-1}],$$
and the trivialization $\beta$ induces  a trivialization of $\F$ along the section $t=0$. Hence the reduction $\varepsilon$ to $B$ corresponds to a section of $G/B$ over $S$, which provides the claimed identification of the generic fiber.  

If $S$ is a $k$-scheme, then the identification of $Gr_{G, \mathcal O}\times \Spec k$ with $Gr_{G, k}$ is via the origin $t=0$, in the sense of Lemma  \ref{globalGr}. Hence the reduction $\varepsilon$ to $B$ corresponds to the choice of a compatible flag in the non-constant $G$-bundle $\F|_{t=0}$ over $S$, hence the triple $(\F, \beta, \varepsilon)$ corresponds to a lifting of the $S$-valued point $(\F, \beta)$ of $Gr_{G, k}$ to an $S$-valued point of $Fl_{G, k}$, which gives the claimed identification of the special fiber. 
\end{proof}

Next  we recall that the orbits of $L^+ G$ on $Gr_G$ are parametrized by the dominant
coweights, cf.\ Remark \ref{rk:identWK}.
% cf.\ \S\ref{s:combinatorics}.\ref{ss:I-W_gp}. 
More precisely, if $A$ denotes a maximal split torus in $B$, and $X_\ast (A)_+$
denotes the dominant coweights with respect to $B$, then the map
\begin{equation*}
\lambda\mto \big(L^+ G\cdot\lambda (T)\cdot L^+ G\big) \big/ L^+ G = \mathcal O_\lambda
\end{equation*}
defines a bijection between $X_\ast (A)_+$ and the set of orbits. Furthermore, $\mathcal O_\lambda$ is a
quasi-projective variety of dimension $\langle 2\varrho , \lambda\rangle$, and
$\mathcal O_\mu\subset\overline{\mathcal O}_\lambda$ if and only if $\mu{\leq}\lambda$
(i.e.,\ $\lambda -\mu$ is a non-negative integral sum of positive co-roots), cf.\ Proposition \ref{dimspec}.
% \S\ref{s:combinatorics}.\ref{ss:I-W_gp}. 
In particular,
$\mathcal O_\lambda$ is a projective variety if and only if $\lambda$ is a minuscule
coweight.

Now we may define a version of local models in this context.
\begin{Definition}
The \emph{local model} attached to $\lambda\in X_\ast (A)_+$ in the Beilinson-Drinfeld-Gaitsgory
context is the scheme-theoretic closure $M_{G, \lambda}$  in $Fl_{G, \mathcal O}$ of the locally
closed subset $\mathcal O_\lambda\times\{e\}$ of $Gr_{G, F}\times_{\Spec F} G/B$.
\end{Definition}
This definition is essentially independent of the choice of the uniformizer $\pi$ of $\mathcal O$. Indeed, any two uniformizers differ by a unit, which may be used to construct a canonical  isomorphism between the corresponding local models. 

It follows from the definition that $M_{G, \lambda}$ is a projective scheme flat of relative
dimension $\langle 2\varrho , \lambda\rangle$ over $\mathcal O$. If $\lambda$ is minuscule,
then the generic fiber of $M_{G, \lambda}$ is projective and smooth. The theory of local models is
concerned with the structure of the schemes $M_{G, \lambda}$. Natural questions that arise in this connection  are the following. When is the special fiber
$M_{G, \lambda}\otimes_\O k$ reduced? What are its singularities, and how can one enumerate its irreducible components?

\begin{Variants} (i) Replacing the Borel subgroup $B$ by a parabolic subgroup $P$ containing
$B$, and the Iwahori subgroup $\B$ by the parahoric subgroup $\P$ corresponding
to $P$ under the reduction morphism, we obtain a scheme $Fl_{G, P, \mathcal O}$ with
generic fiber equal to $Gr_{G, F}\times G/P$ and with special fiber equal to the affine partial flag
variety $L G /\P$. Correspondingly we define local models $M_{G, P, \lambda}$ over $\Spec\mathcal O$ 
for $\lambda\in X_\ast (A)_+$, with generic fiber contained in $Gr_{G, F}$, and with special fiber contained in the partial flag variety $LG/L^+\P$ over $k$. For an inclusion $P\subset P'$ of two \emph{standard}
parabolic subgroups of $G$, we obtain a morphism between local models
\begin{equation*}
M_{G, P, \lambda}\to M_{G, P', \lambda}\ ,
\end{equation*}
which induces an isomorphism in the generic fibers. In the extreme case $P = G$,
the scheme $Fl_{G, P, \mathcal O}$  has generic fiber $Gr_{G, F}$ and special fiber $Gr_{G, k}$,  and the local model $M_{G, G, \lambda}$ ``looks constant"  over $\Spec\mathcal O$,  with generic fiber the Schubert variety $\overline{\mathcal O}_\lambda$ in $Gr_{G, F}$,  and special fiber the Schubert variety $\overline{\mathcal O}_\lambda$ in $Gr_{G, k}$. If $\lambda$ is minuscule,  then $M_{G, G, \lambda}$ is projective and smooth over $\mathcal O$. 
\smallskip

\noindent (ii) The preceding considerations generalize without substantial changes to the case when $G$ is a quasi-split reductive group over $\mathcal O$. 

\smallskip

\noindent (iii) An alternative definition of $M_{G, P, \lambda}$ can be given as follows. 
Starting from the Chevalley form of $G$ over $\O$
and a parabolic subgroup $P$ as above, we can construct a smooth  ``parahoric group
scheme" $\wh{\mathcal G}$ over $\Spec \O[[t]]$. The generic, resp.\  special, fiber  of $\wh\Gg\rightarrow \Spec \O$ is isomorphic to the  smooth affine ``parahoric group
scheme"  $\wh\Gg_\kappa$ over $\Spec \kappa[[t]]$ with $\kappa=F$, resp.\ $k$, given by Bruhat-Tits theory.
(These are characterized by requiring that
   $\wh\Gg_\kappa(\kappa^{\rm sep}[[t]])$ is equal to the group of elements of  $G(\kappa^{\rm sep}[[t]])$ with reduction modulo $t$ contained in  $P(\kappa^{\rm sep})$.)
    For example, $\wh\Gg$ can be obtained by applying the constructions
of \cite[3.2,  3.9.4]{BTII} to the two dimensional base $\Spec \O[[t]]$ by picking appropriate
schematic root data given by ideals generated by powers of  $t$, see also \cite[p.\ 147]{P-R4}.  
The base change $\wh{\mathcal G}\times _{\Spec \O[[t]]}\Spec \O((t))$ 
is identified with the Chevalley group scheme $G\times _{\Spec \O}\Spec \O((t))$.  We can now glue  
the ``constant" group scheme $G$ over $\Spec \O[t, t^{-1}]$ with  $\wh\Gg$ over $\Spec \O[[t]]$ to
produce a ``Bruhat-Tits group scheme" $\mathcal G$  over the affine line $X=\Spec \O[t]$, cf.\   \cite{P-R5,Heinloth}. 
Let us define the functor $ Gr_{{\mathcal G}, X}$
exactly as in (\ref{grGX}) above,  except that  $ G$-torsors are now replaced by $\mathcal G$-torsors. Also as above, set 
$$
Gr_{{\mathcal G},\O}=Gr_{{\mathcal G}, X}\times_{X, \delta} \Spec \O
$$
where $\delta\colon \Spec \O\to X$ is given by $t\mapsto \pi$. Note that  ${\mathcal G}\times_{X,\delta}\Spec \O$ 
is the parahoric group scheme associated to the subgroup of elements of  $G(\O^{\rm un})$ with reduction modulo $\pi$ contained in  $P(k^{\rm sep})$.
Similar to  Lemma \ref{lemma135}, we can see that $ Gr_{{\mathcal G}, \O}\times \Spec F=Gr_{G, F}$
is the affine Grassmannian of the loop group of $G$ over $F$, while $Gr_{\mathcal G, \O}\times \Spec k=LG/L^+\wh{\mathcal G}_k$ is the affine flag 
variety corresponding to the parahoric subgroup $\wh{\mathcal G}(k[[t]])$ over $k$. The rest of the construction proceeds the same way: we define $M_{G,P,\lambda}=M_{{\mathcal G}, \lambda}$ to be the Zariski closure of the orbit $\O_\lambda$.  

This construction extends beyond  the split case and is used in \cite{PZ}  to provide a definition 
of local models $M^{\rm loc}(G, \{\mu\})_K$ under some rather general assumptions. 
Indeed, one can deal   with all reductive groups $G$ that split over a tamely ramified 
extension of $F$ and with general parahoric subgroups. The technical details of the construction of the group scheme
$\Gg$ over $X=\Spec \O[t]$ and of the global affine Grassmannian $Gr_{\Gg, X}$ in the general (tamely ramified)
case are quite involved and we will not attempt to report on them here.
Instead, we  will refer the reader to the forthcoming
article \cite{PZ}.
 
\end{Variants} 

\smallskip

In the rest of this survey we will only discuss the models that are directly related to (mostly PEL) Shimura varieties, as sketched in \S \ref{s:motivationalsection}.\ref{ss:Shimuralocmod}. However, especially after Gaitsgory's paper \cite{Ga}, it is likely  that methods from elsewhere, such as from the theory of the Geometric Langlands Correspondence, will have an impact on the problems discussed in this report.\footnote{In this respect, we refer to very recent work of 
X.~ Zhu \cite{Zh} on the coherence conjecture \cite{P-R3} and to the forthcoming \cite{PZ}.}
We hope that our loose discussion  above  can help in this respect to attract people from these other areas to the theory of local models.

\section{Basic examples}\label{s:examples}

In this section we make explicit the definition of the local model in the style of \S\ref{s:motivationalsection}.\ref{ss:Shimuralocmod} in the most basic cases. Let $F$ be a discretely valued field, 
$\O_F$ its ring of integers, $\pi \in \O_F$ a uniformizer, and $k = \O_F/\pi\O_F$ its residue field which we assume
is perfect. Let $n$ be a positive integer.  A \emph{lattice chain in $F^n$} is a collection 
of $\O_F$-lattices in $F^n$ totally ordered under inclusion.  A lattice chain $\L$ is \emph{periodic} if $a\Lambda \in \L$ for 
every $\Lambda \in \L$ and $a \in F^\times$.  For $i = na+j$ with $0 \leq j < n$, we define the $\O_{F}$-lattice
\begin{equation}\label{disp:Lambda_i}
   \Lambda_i := \sum_{l=1}^j \pi^{-a-1}\O_F e_l + \sum_{l=j+1}^{n} \pi^{-a}\O_F e_l \subset F^n,
\end{equation}
where $e_1,\dotsc,e_n$ denotes the standard ordered basis in $F^n$.  Then the $\Lambda_i$'s form a periodic lattice chain
\begin{equation}\label{disp:std_lattice_chain}
   \dotsb \subset \Lambda_{-2} \subset \Lambda_{-1} \subset \Lambda_0 \subset \Lambda_1 \subset \Lambda_2 \subset \dotsb,
\end{equation}
which we call the \emph{standard lattice chain.}

Given a partition $n = r + s$, we recall the cocharacter $\bigl(1^{(r)},0^{(s)}\bigr)$ of $GL_n$ defined in Example \ref{eg:picard}(i); we shall also regard this as a cocharacter of certain subgroups of $GL_n$ (e.g.\ $GSp_{2g}$, $GO_{2g},\dots$), as appropriate.

In each case except for \S\S\ref{s:examples}.\ref{ss:ResGL_n} and \ref{s:examples}.\ref{ss:ResGSp}, we give the types of the adjoint group and nontrivial minuscule coweights under consideration, in the sense of the table in \S\ref{s:motivationalsection}.\ref{ss:Shimuralocmod}.

\subsection{Split unitary, i.e.\ $GL_n$  (types $(A_{n-1}, \varpi^\vee_r)$, $1 \leq r \leq n-1$)}\label{ss:GL_n}
We refer to \S\ref{s:motivationalsection}.\ref{ss:Shimuralocmod} for an explanation of why we lump the cases $GL_n$ and the split unitary group relative to an unramified quadratic extension together. 

Let $G := GL_n$ over $F$ and let $\L$ be a periodic lattice chain in $F^n$.  Fix an integer $r$ with $0 \leq r \leq n$, let $\mu$ denote the cocharacter $\bigl( 1^{(r)}, 0^{(n-r)} \bigr)$ of the standard maximal torus of diagonal matrices in $G$, and let $\{\mu\}$ denote the geometric conjugacy class of $\mu$ over $\ol F$.  The \emph{local model $M^\loc_{G,\{\mu\},\L}$} attached to the triple $(G,\{\mu\},\L)$ is the functor on the category of $\O_F$-algebras that assigns to each $\O_F$-algebra $R$ the set of all families $(\F_\Lambda)_{\Lambda\in\L}$ such that
\begin{altenumerate}
\item\label{it:lm_rank}
   (\emph{rank}) for every $\Lambda \in \L$, $\F_\Lambda$ is an $R$-submodule of $\Lambda \otimes_{\O_F} R$ which Zariski-locally on $\Spec R$ is a direct summand of rank $n-r$;
\item\label{it:lm_functoriality}
   (\emph{functoriality}) for every inclusion of lattices $\Lambda \subset \Lambda'$ in $\L$, the induced map $\Lambda \otimes_{\O_F} R \to \Lambda' \otimes_{\O_F} R$ carries $\F_\Lambda$ into $\F_{\Lambda'}$:
   \[
      \xymatrix@R=2.5ex{
         \Lambda \otimes_{\O_F} R \ar[r] \ar@{}[d] |*{\rotatebox{90}{$\subset$}}&
            \Lambda' \otimes_{\O_F} R \ar@{}[d] |*{\rotatebox{90}{$\subset$}}\\
         \F_\Lambda \ar@{-->}[r] &
            \F_{\Lambda'};
      }
   \]
\item\label{it:lm_periodicity}
\setcounter{dummy}{\value{enumi}}
   (\emph{periodicity}) for every $a \in F^\times$ and every $\Lambda\in\L$, the isomorphism $\smash{\Lambda \xra[\sim] a a\Lambda}$ identifies $\F_\Lambda \isoarrow \F_{a\Lambda}$.
\end{altenumerate}

It is clear that $M^\loc_{G,\{\mu\},\L}$ is representable by a closed subscheme of a product of finitely many copies of $\Gr(n-r,n)_{\O_F}$, the Grassmannian of $(n-r)$-planes in $n$-space; and that $M^\loc_{G,\{\mu\},\L}$ has generic fiber isomorphic to $\Gr(n-r,n)_F$.  The fundamental result of G\"ortz's paper \cite{Go1} is the following.

\begin{thm}[G\"ortz {\cite[4.19, 4.21]{Go1}}]\label{st:GL_n_loc_mod_flat}
For any $\mu = \bigl( 1^{(r)}, 0^{(n-r)} \bigr)$ and periodic lattice chain $\L$, $M^\loc_{G,\{\mu\},\L}$ is flat over $\Spec \O_F$ with reduced special fiber.  The irreducible components of its special fiber are normal with rational singularities, so in particular are Cohen-Macaulay. Furthermore, $M^\loc_{G,\{\mu\},\L}$ has semi-stable reduction when $\mu=\bigl(1, 0^{(n-1)}\bigr)$. \qed
\end{thm}

Here a normal variety having ``rational singularities'' is meant in the strongest sense, i.e.,\ there exists a birational proper morphism from a smooth variety to it such that the higher direct images of the structure sheaf and of the dualizing sheaf vanish.

\begin{eg}\label{exn=2}
The simplest nontrivial example occurs for $n = 2$, $\mu = (1,0)$, and $\L$ the standard lattice chain (the Iwahori case).  The most interesting point on the local model is the $k$-point $x$ specified by the lines $k \ol e_1 \subset \ol\Lambda_0$ and $k \ol e_2 \subset \ol\Lambda_1$, where we use a bar to denote reduction mod $\pi$.  In terms of standard affine charts of the Grassmannian, we find that $x$ has an affine neighborhood $U$ in the local model consisting of all points of the form
\[
   \F_{\Lambda_0} = \spn\{e_1 + Xe_2\}
   \quad\text{and}\quad
   \F_{\Lambda_1} = \spn \{Y\pi^{-1}e_1 + e_2\}
\]
such that $XY = \pi$.  Hence $U \cong \Spec \O_F[X,Y]/(XY-\pi)$.  Hence $U$ visibly satisfies the conclusions of Theorem \ref{st:GL_n_loc_mod_flat}; its special fiber consists of two copies of $\mA_k^1$ meeting transversely at $x$.  In fact, globally the special fiber of the local model consists of two copies of $\mP_k^1$ meeting at $x$.  By contrast, taking $\L$ to be the homothety class of $\Lambda_0$ or of $\Lambda_1$ (the maximal parahoric case), the local model is tautologically isomorphic to $\mP_{\O_F}^1$.

In \S\ref{s:matrix}.\ref{ss:redmat} we shall consider various analogs of the scheme $U$ both for higher rank and for other groups, which we broadly refer to as schemes of matrix equations.  Note that $U$ is exactly the scheme $Z_{1,2}$ appearing in Theorem \ref{st:gencirc_flat}.
\end{eg}

\begin{Remark}\label{rem: 2.1.3}
In light of Theorem \ref{st:GL_n_loc_mod_flat}, it is an interesting question whether the special fiber, as a whole, of the local model is Cohen-Macaulay; by the flatness result, this is equivalent to the local model itself being Cohen-Macaulay.  If Cohen-Macaulayness holds, then, since by the theorem above the special
fiber is generically smooth, we can apply Serre's criterion to deduce that the local model is also normal.  

In \cite[\S4.5.1]{Go1},  G\"ortz proposes to attack the question of Cohen-Macaulayness of the special fiber by means of a purely combinatorial problem in the affine Weyl group, which however appears to be difficult, at least when $\L$ is the full standard lattice chain.  In this way he has found that the special fiber is Cohen-Macaulay for $n \leq 4$ by hand calculations, and for $n \leq 6$ by computer calculations. 
Cohen-Macaulayness can also be shown via his approach for any $n$, in the case that the lattice chain $\L$ consists of the multiples of only two lattices. Similar remarks apply to local models for any group, whenever the special fiber of the local model can be identified with a union of Schubert varieties in an affine flag variety. 
See Remark \ref{rem:2.2.5} for another case where this property can be shown. By contrast, we know of no experimental evidence for Cohen-Macaulayness of the special fiber in any other Iwahori, i.e.\ ``full lattice chain,''  cases. We shall discuss embedding the special fiber of local models in affine flag varieties in \S\ref{s:afv}.\ref{ss:embedding}.  The question of Cohen-Macaulayness and normality of local models is a major open problem in the field.
\end{Remark}

\subsection{Split symplectic (types $(C_g, \varpi^\vee_g)$)}\label{ss:GSp_2g}
Let $n = 2g$, and let \aform denote the alternating $F$-bilinear form on $F^{2g}$ whose matrix with respect to the standard ordered basis is
\begin{equation}\label{disp:sympl}
   J_n:=\begin{pmatrix}
            &  H_g\\
      -H_g  &
   \end{pmatrix},
\end{equation}
where $H_g$ denotes the $g \times g$ matrix
\begin{equation}\label{disp:antidiag_1}
   H_g :=
   \begin{pmatrix}
     &  &  1\\
     \phantom{\iddots} & \iddots & \phantom{\iddots}\\
     1
   \end{pmatrix}.
\end{equation}
Given a lattice $\Lambda$ in $\L$, we denote by $\wh \Lambda$ its \aform-dual,
\[
   \wh\Lambda := \bigl\{\, x\in F^{2g} 
                      \bigm| \langle\Lambda,x\rangle \subset \O_F \,\bigr\}.
\]
Then \aform induces a perfect bilinear pairing of $\O_F$-modules
\begin{equation}\label{disp:pairing}
   \Lambda \times \wh\Lambda \to \O_F.
\end{equation}
We say that a lattice chain $\L$ in $F^{2g}$ is \emph{self-dual} if $\wh\Lambda \in \L$ for all $\Lambda \in \L$.

Let $G := GSp_{2g} := GSp\bigl(\text{\aform}\bigr)$ over $F$, let $\mu$ denote the cocharacter $\bigl( 1^{(g)}, 0^{(g)} \bigr)$ of the standard maximal torus of diagonal matrices in $G$, and let $\{\mu\}$ denote its geometric conjugacy class over $\ol F$.  Let $\L$ be a periodic self-dual lattice chain in $F^{2g}$.  The \emph{local model $M^\loc_{G,\{\mu\},\L}$} is the closed $\O_F$-subscheme of $M^\loc_{GL_{2g},\{\mu\},\L}$ whose $R$-points, for each $\O_F$-algebra $R$, satisfy the additional condition
\begin{altenumerate}
\setcounter{enumi}{\value{dummy}}
\item\label{it:lm_duality}
   (\emph{perpendicularity}) for all $\Lambda \in \L$, the perfect $R$-bilinear pairing
   \[
      \bigl(\Lambda\otimes_{\O_F} R\bigr) \times 
         \bigl(\wh\Lambda \otimes_{\O_F} R\bigr) \to R
   \]
   obtained by base change from \eqref{disp:pairing} identifies $\F_\Lambda^\perp \subset \wh\Lambda \otimes_{\O_F} R$ with $\F_{\wh\Lambda}$.
\end{altenumerate}

This time the local model $M^\loc_{G,\{\mu\},\L}$ has generic fiber $\LGr(2g)_F$, the Grassmannian of Lagrangian subspaces in $F^{2g}$.  The fundamental result of G\"ortz's paper \cite{Go2} is the following.

\begin{thm}[G\"ortz {\cite[2.1]{Go2}}]\label{st:GSp_loc_mod_flat}
For any periodic self-dual lattice chain $\L$, $M^\loc_{G,\{\mu\},\L}$ is flat over $\Spec \O_F$ with reduced special fiber.  The irreducible components of its special fiber are normal with rational singularities, so in particular are Cohen-Macaulay.\qed
\end{thm}

\begin{Remark}\label{rem:2.2.5}
In the case that the lattice chain $\L$ consists of multiples of two lattices $\Lambda$ and $\Lambda'$
such that $\wh\Lambda=\Lambda$ and $\wh\Lambda'=\pi\Lambda$, one can obtain a   better
result, namely that the whole special fiber is Cohen-Macaulay and that the local model is normal.
This was first shown in \cite{CN}.
See Theorem \ref{Thm6.1.6} and the discussion after its statement. 
\end{Remark}

\subsection{Split orthogonal (types $(D_g, \varpi_{g-1}^\vee)$, $(D_g,  \varpi^\vee_g)$)}\label{ss:GO_2g}

In this example we assume $\charac k \neq 2$.
Let $n = 2g$, and let \sform denote the symmetric $F$-bilinear form on $F^{2g}$ whose matrix with respect to the standard ordered basis is $H_{2g}$ \eqref{disp:antidiag_1}.  Let $\wh\Lambda$ denote the \sform-dual of any lattice $\Lambda$ in $F^{2g}$.  Analogously to the previous subsection, \sform induces a perfect pairing $\Lambda \times \wh\Lambda \to \O_F$ for any lattice $\Lambda$.  We again say that a lattice chain in $F^{2g}$ is \emph{self-dual} if it is closed under taking duals.

Let $G := GO_{2g} := GO\bigl(\text{\sform}\bigr)$ over $F$,%
\footnote{Note that $G$ is \emph{disconnected,} so that it does not honestly fit into the framework of \S\ref{s:motivationalsection}. See the discussion after Remark \ref{rk:oriflamme}.}
let $\mu$ denote the cocharacter $\bigl( 1^{(g)}, 0^{(g)} \bigr)$ of the standard maximal torus of diagonal matrices in $G$, and let $\{\mu\}$ denote its $G(\ol F)$-conjugacy class over $\ol F$.  Let $\L$ be a periodic self-dual lattice chain $F^{2g}$.  The \emph{naive local model $M^\naive_{G,\{\mu\},\L}$} is the closed  $\O_F$-subscheme of $M^\loc_{GL_{2g},\{\mu\},\L}$ defined in exactly the same way as for $GSp_{2g}$, that is, we impose condition \eqref{it:lm_duality} with the understanding that all notation is taken with respect to \sform.

Analogously to the symplectic case, $M^\naive_{G,\{\mu\},\L}$ has generic fiber
$\OGr(g,2g)_F$, the orthogonal Grassmannian of totally isotropic $g$-planes in
$F^{2g}$. But contrary to the symplectic and linear cases --- and the reason
here for the adjective ``naive'' --- the naive local model is typically \emph{not flat over $\O_F$}, as was first observed by Genestier \cite{Ge2}.

\label{ref:not_flat} A major source of trouble is the fact that the orthogonal Grassmannian is
not connected, but has two connected components. To fix ideas, let us suppose
that $\L$ contains a self-dual lattice $\Lambda'$ and a lattice $\Lambda''
\supset \Lambda'$ with $\dim_k \Lambda''/\Lambda' = g$; then $\wh {\Lambda''} =
\pi \Lambda''$. Given an $R$-point $(\F_{\Lambda})_{\Lambda\in\L}$ of
$M^\naive_{G,\{\mu\},\L}$, the perpendicularity condition requires that
$\F_{\Lambda'}$ be totally isotropic for $\text{\sform}_R$, and the
perpendicularity and periodicity conditions require that $\F_{\Lambda''}$ be
totally isotropic for $\bigl(\pi^{-1}\text{\sform}\bigr)_R$, where we use a
subscript $R$ to denote base change to $R$. Hence we get a map
\begin{equation}\label{disp:OGr_mapping}
   \begin{matrix}
   \xymatrix@R=0ex{
      M^\naive_{G,\{\mu\},\L} \ar[r]  &  \OGr(g,2g)_{\O_F} \times \OGr(g,2g)_{\O_F}\\
      (\F_{\Lambda})_{\Lambda\in\L} \ar@{|->}[r]  &  (\F_{\Lambda'}, \F_{\Lambda''}).
   }
   \end{matrix}
\end{equation}

Now, quite generally, a scheme $X$ over a regular, integral, $1$-dimensional base scheme is flat if and only if the scheme-theoretic closure in $X$ of the generic fiber of $X$ is equal to $X$.  In our present situation, the target space in \eqref{disp:OGr_mapping} has $4$ connected components,  $2$ of which contain the image of the $2$ connected components of the generic fiber of $M^\naive_{G,\{\mu\},\L}$.  But for any $g \geq 1$, the image of $M^\naive_{G,\{\mu\},\L}$ always meets another component; see \cite[8.2]{P-R4} for a simple example which is easy to generalize to higher rank. Hence the generic fiber of $M^\naive_{G,\{\mu\},\L}$ is not dense in $M^\naive_{G,\{\mu\},\L}$, so that $M^\naive_{G,\{\mu\},\L}$ is not flat.

To correct for non-flatness of the naive local model, one simply defines the true local model $M^\loc_{G,\{\mu\},\L}$ to be the scheme-theoretic closure in $M^\naive_{G,\{\mu\},\L}$ of its generic fiber.  Then $M^\loc_{G,\{\mu\},\L}$ is flat essentially by definition, but a priori it carries the disadvantage of not admitting a ready moduli-theoretic description.  In \cite{P-R4} a remedy for this disadvantage is proposed in the form of a new condition, called the \emph{spin condition}, which is added to the moduli problem defining $M^\naive_{G,\{\mu\},\L}$.  Unfortunately the spin condition is a bit technical to formulate; we refer to \cite[\S\S7.1, 8.2]{P-R4} or \cite[\S2.3]{Sm1} for details.  In the simple case that $\L$ consists of the homothety classes of a self-dual lattice $\Lambda'$ and a lattice $\Lambda''$ satisfying $\wh{\Lambda''} = \pi \Lambda''$, the map \eqref{disp:OGr_mapping} is a closed embedding, and the effect of the spin condition is simply to intersect $M^\naive_{G,\{\mu\},\L}$ with the two connected components of $\OGr(g,2g)_{\O_F} \times \OGr(g,2g)_{\O_F}$ marked by the generic fiber of $M^\naive_{G,\{\mu\},\L}$.  For more general $\L$, the spin condition becomes more complicated.

In general, let $M^\spin_{G,\{\mu\},\L}$ denote the closed subscheme of $M^\naive_{G,\{\mu\},\L}$ that classifies points satisfying the spin condition.  The inclusion $M^\spin_{G,\{\mu\},\L} \subset M^\naive_{G,\{\mu\},\L}$ is shown in \cite{P-R4} to be an isomorphism on generic fibers, and we then have the following.  

\begin{conjecture}[{\cite[Conj.\ 8.1]{P-R4}}]
For any periodic self-dual lattice chain $\L$, $M^\spin_{G,\{\mu\},\L} = M^\loc_{G,\{\mu\},\L}$, that is, $M_{G,\{\mu\},\L}^\spin$ is flat over $\Spec \O_F$.
\end{conjecture}

Hand calculations show that $M^\spin_{G,\{\mu\},\L}$ is indeed flat with reduced special fiber for $n \leq 3$; see \cite[\S8.3]{P-R4} for some explicit examples for $n = 1$ and $2$.  The main result of \cite{Sm1} is the following weakened form of the conjecture (for arbitrary $n$), the full version of which is still open.  Recall that a lattice chain is \emph{complete} if all successive quotients have $k$-dimension $1$, and that a scheme over a regular, integral, $1$-dimensional base scheme is \emph{topologically flat} if its generic fiber is dense.

\begin{thm}[{\cite[7.6.1]{Sm1}}]\label{st:GO_loc_mod_top_flat}
For any complete periodic self-dual lattice chain $\L$, $M^\spin_{G,\{\mu\},\L}$ is topologically flat over $\Spec \O_F$; or in other words, the underlying topological spaces of $M^\spin_{G,\{\mu\},\L}$ and $M^\loc_{G,\{\mu\},\L}$ coincide.\qed
\end{thm}

\begin{Remark}\label{rk:oriflamme}
For sake of unformity assume $g \geq 4$.  In proving Theorem \ref{st:GO_loc_mod_top_flat}, it suffices to take $\L$ to be the standard lattice chain $\Lambda_\bullet$ \eqref{disp:std_lattice_chain}.  But from a building-theoretic perspective, it is more natural to instead consider the periodic self-dual lattice \emph{oriflamme}
\[
   \xy
      (-47,7)*+{\dotsb};
      (-42.5,3.5)*+{\rotatebox{-45}{$\subset$}};
      (-42.5,10.5)*+{\rotatebox{45}{$\subset$}};
      (-38,14)*+{\Lambda_0};
      (-38,0,)*+{\Lambda_{0'}};
      (-34,3.5)*+{\rotatebox{45}{$\subset$}};
      (-34,10.5)*+{\rotatebox{-45}{$\subset$}};
      (-19,7)*+{\Lambda_2 \subset \dotsb \subset \Lambda_{g-2}};
      (-4.5,3.5)*+{\rotatebox{-45}{$\subset$}};
      (-4.5,10.5)*+{\rotatebox{45}{$\subset$}};
      (0,14)*+{\Lambda_g};
      (0,0)*+{\Lambda_{g'}};
      (4,3.5)*+{\rotatebox{45}{$\subset$}};
      (4,10.5)*+{\rotatebox{-45}{$\subset$}};
      (21.5,7)*+{\Lambda_{g+2} \subset \dotsb \subset \Lambda_{2g-2}};
      (38.5,3.5)*+{\rotatebox{-45}{$\subset$}};
      (38.5,10.5)*+{\rotatebox{45}{$\subset$}};
      (43.8,14)*+{\Lambda_{2g}};
      (43.8,0)*+{\Lambda_{2g'}};
      (48.5,3.5)*+{\rotatebox{45}{$\subset$}};
      (48.5,10.5)*+{\rotatebox{-45}{$\subset$}};
      (53,7)*+{\dotsb};
   \endxy,
\]
where, for $a \in \mZ$,
\[
   \Lambda_{2ga'} := \pi^{-a-1}\O_F e_1 + \biggl(\sum_{l = 2}^{2g-1} \pi^{-a}\O_F e_l\biggr) + \pi^{-a+1}\O_F e_{2g}
\]
and
\[
   \Lambda_{(g + 2ga)'} := \biggl(\sum_{l = 1}^{g-1} \pi^{-a-1} \O_F e_l\biggr)
                           + \pi^{-a} \O_F e_{g}
                           + \pi^{-a-1} \O_F e_{g+1}
                           + \sum_{l=g+2}^{2g} \pi^{-a} \O_F e_l.
\]
Then the lattice-wise fixer of $\Lambda_\bullet$ in $G(F)$ is the same as that for the displayed oriflamme, namely the Iwahori subgroup $B$ of elements in $G(\O_F)$ which are upper triangular mod $\pi$;%
\footnote{One verifies easily that in fact $B \subset G^\circ(\O_F)$.}
and the parahoric subgroups of $G^\circ(F)$ that contain $B$ are precisely the parahoric stabilizers of periodic, self-dual subdiagrams of the displayed oriflamme.

One can define a naive local model for the displayed oriflamme just as we have done for lattice chains, namely by specifying a locally direct summand of rank $g$ for each lattice in the oriflamme, subject to functoriality, periodicity, and perpendicularity conditions.  However this naive local model again fails to be flat:  this time the four lattices $\Lambda_0$, $\Lambda_{0'}$, $\Lambda_g$, and $\Lambda_{g'}$ are all self-dual up to scalar, and one can see in a way very similar to what we discussed on p.~\pageref{ref:not_flat} that the naive local model is not even topologically flat.  One can see as in \cite[\S8.2]{P-R4} that it is necessary to impose a version of the spin condition, and we conjecture that the resulting spin local model is flat.
\end{Remark}

As already noted, the treatment of the local model in this subsection does not honestly fall under the framework set out in \S\ref{s:motivationalsection}, since $GO_{2g}$ is disconnected.  But if we take the philosophy of \S\ref{s:motivationalsection} seriously, then we should expect to have local models for the connected group $GO_{2g}^\circ$ (or its adjoint quotient $PGO_{2g}^\circ$) and each of its minuscule coweights $\varpi_{g-1}$ and $\varpi_g$.  Here we can simply define these two local models to be the respective Zariski closures of each of the two components of $\OGr(g,2g)_F$ in $M^\naive_{G,\{\mu\},\L}$.  In this way the local model $M^\loc_{G,\{\mu\},\L}$ for $GO_{2g}$ is just the disjoint union of these two local models for $GO_{2g}^\circ$. %(but note that we are unable to canonically distinguish between them).

% \textbf{BS/GP:  will (I think) need to add a remark to this section that $GO_{2g}$ doesn't honestly fit into the framework of \S1, since we'll probably only consider connected groups there.  But assuming the conjecture that $M^\spin = M^\loc$, then I think one sees a posteriori that $M^\loc$ is really the disjoint union of two ``\S1-style'' local models for $GO_{2g}^\circ$, one for each of the two ``PEL'' minuscule dominant cocharacters.}
% {\bf MR: YES, ADD THIS!! GP: I wrote something very similar in the non-split orthogonal paragraph. The two statements should coordinate.
% In fact, if you look at \cite[8.2]{P-R4} you will see that the local model $M^{\rm loc}_{G,\{\mu\}, \L}$ is really defined as a scheme over $\Spec(\O_F\oplus \O_F)$.
% I think that the correct definition should be as follows: take the Zariski closure of each component of $\OGr(g,2g)_F$ into $M^\naive_{G,\{\mu\},\L}$
% {\sl separately}. Then the local model is the disjoint union of the two resulting Zariski closures.}

\subsection{Weil restriction of $GL_n$}\label{ss:ResGL_n}
We now begin to consider the simplest examples of local models for \emph{nonsplit} groups.
Let $F_0$ be a discretely valued field with ring of integers $\O_{F_0}$ and residue field $k_0$.  We suppose that $F$ is a finite extension of $F_0$ contained in a separable closure $F_0^\sep$ of $F_0$.  Let $d := [F:F_0]$ and let $e$ denote the ramification index of $F/F_0$, so that $e \mid d$.

In this subsection we generalize our discussion of $GL_n$ in \S\ref{s:examples}.\ref{ss:GL_n} to the group $G := \Res_{F/F_0} GL_n$ over $F_0$.  As in \S\ref{s:examples}.\ref{ss:GL_n}, we place no restrictions on the characteristic of $k$.  Let $\L$ be a periodic $\O_F$-lattice chain in $F^n$.  Let $F^{\Gal}$ denote the Galois closure of $F$ in $F_0^\sep$. Then we have the standard splitting upon base change to $F^{\Gal}$,
\begin{equation}\label{disp:ResGL_n_splitting}
   G_{F^{\Gal}} \cong \prod_{\varphi\colon\! F \rightarrow F_0^\sep} GL_n,
\end{equation}
where the product runs through the set of $F_0$-embeddings $\varphi\colon F \to F_0^\sep$.  For each such $\varphi$, choose an integer $r_\varphi$ with $0 \leq r_\varphi \leq n$; let
\[
   r := \sum_{\varphi}r_\varphi;
\]
let $\mu_{\varphi}$ denote the cocharacter $\bigl( 1^{(r_\varphi)}, 0^{(n-r_\varphi)} \bigr)$ of the standard maximal torus of diagonal matrices in $GL_n$; let $\mu$ denote the geometric cocharacter of $G$ whose $\varphi$-component, in terms of \eqref{disp:ResGL_n_splitting}, is $\mu_\varphi$; and let $\{\mu\}$ denote the geometric conjugacy class of $\mu$.  Let $E$ denote the reflex field of $\{\mu\}$; this is easily seen to be the field of definition of $\mu$, that is, the fixed field in $F_0^\sep$ of the subgroup of the Galois group
\[
   \bigl\{\, \sigma\in \Gal(F_0^\sep/F_0) \bigm|
      r_{\sigma \circ \varphi} = r_\varphi 
      \text{ for all } \varphi\colon F \rightarrow F_0^\sep \,\bigr\}.
\]
Plainly $E \subset F^{\Gal}$.  Let $\O_E$ denote the ring of integers in $E$.

The \emph{naive local model $M^\naive_{G,\{\mu\},\L}$} attached to the triple $(G,\{\mu\},\L)$ is the functor on the category of $\O_E$-algebras that assigns to each $\O_E$-algebra $R$ the set of all families $(\F_\Lambda)_{\Lambda\in\L}$ such that
\begin{altenumerate}
\item
   for every $\Lambda \in \L$, $\F_\Lambda$ is an $(\O_F \otimes_{\O_{F_0}} R)$-submodule of $\Lambda \otimes_{\O_{F_0}} R$ which Zariski-locally on $\Spec R$ is a direct summand as an $R$-module of rank $dn-r$;
\item
   for every inclusion of lattices $\Lambda \subset \Lambda'$ in $\L$, the induced map $\Lambda \otimes_{\O_{F_0}} R \to \Lambda' \otimes_{\O_{F_0}} R$ carries $\F_\Lambda$ into $\F_{\Lambda'}$;
\item
   for every $a \in F^\times$ and every $\Lambda\in\L$, the isomorphism $\smash{\Lambda \xra[\sim] a a\Lambda}$ identifies $\F_\Lambda \isoarrow \F_{a\Lambda}$; and
\item
   (\emph{Kottwitz condition}) for every $a \in \O_F$ and every $\Lambda\in\L$, the element $a \otimes 1 \in \O_F \otimes_{\O_{F_0}} R$ acts on the quotient $(\Lambda\otimes_{\O_{F_0}} R)/\F_\Lambda$ as an $R$-linear endomorphism with characteristic polynomial
   \[
      {\rm char}_R \bigl(\,a \otimes 1 \bigm| (\Lambda\otimes_{\O_{F_0}} R)/\F_\Lambda\,\bigr) 
      = \prod_{\varphi\colon\! F \rightarrow F_0^\sep} \bigl(X- \varphi(a)\bigr)^{r_\varphi}. 
   \]
  %  
  %  for every $a \in \O_F$ and every $\Lambda\in\L$, we have
  %  \[
  %     {\textstyle \det_R} (a \otimes 1 \mid \F_\Lambda) = \prod_{\varphi\colon\! F \rightarrow F_0^\sep} \varphi(a)^{n-r_\varphi},
  %  \]
  %  where $a \otimes 1 \in \O_F \otimes_{\O_{F_0}} R$, which we interpret as an equality of polynomials with coefficients in $R$ in the sense of \cite[\S5]{Ko1} or \cite[\S3.23(a)]{R-Z}. Equivalently, this condition can be expressed in terms of the characteristic polynomial of the action of $a\in\O_F$, as
  % \[
  %     {\textstyle {\rm char}_R} (a \otimes 1 \mid \F_\Lambda) = \prod_{\varphi\colon\! F \rightarrow F_0^\sep}(X- \varphi(a))^{n-r_\varphi}. 
  %  \]
\end{altenumerate}

Note that in the statement of the Kottwitz condition the polynomial $\prod_{\varphi} \bigl(X-\varphi(a)\bigr)^{r_\varphi}$ can first be regarded as a polynomial with coefficients in $\O_E$ by definition of $E$, and then as a polynomial with coefficients in $R$ via its $\O_E$-algebra structure.  We remark that in \cite[\S5]{Ko1} and \cite[\S3.23(a)]{R-Z}, the Kottwitz condition is formulated in a different (but equivalent) way as a ``determinant'' condition.  As always, $M^\naive_{G,\{\mu\},\L}$ is plainly representable by a projective $\O_E$-scheme.

When $F$ is \emph{unramified} over $F_0$, upon base change to $F^{\Gal}$, $M^\naive_{G,\{\mu\},\L}$ becomes isomorphic to a product of local models for $GL_n$ of the form considered in \S\ref{s:examples}.\ref{ss:GL_n}.  Hence \eqref{st:GL_n_loc_mod_flat} implies the following.

\begin{thm}[G\"ortz {\cite[4.25]{Go1}}]\label{st:ResGL_n_unram}
Suppose $F$ is unramified over $F_0$.  Then for any $\mu$ as above and any periodic $\O_F$-lattice chain $\L$, $M^\naive_{G,\{\mu\},\L}$ is flat over $\Spec \O_E$ with reduced special fiber.  The irreducible components of its special fiber are normal with rational singularities, so in particular are Cohen-Macaulay.\qed
\end{thm}

In general, i.e.\ in the presence of ramification, the naive local model  need not be flat.  As in the orthogonal case, the honest local model $M^\loc_{G,\{\mu\},\L}$ is then defined to be the scheme-theoretic closure in $M^\naive_{G,\{\mu\},\L}$ of its generic fiber; thus $M^\loc_{G,\{\mu\},\L} = M^\naive_{G,\{\mu\},\L}$ when $F$ is unramified.  Unfortunately, in contrast to the orthogonal case, it appears to be unreasonable to hope to give a simple, explicit, purely moduli-theoretic description of $M^\loc_{G,\{\mu\},\L}$ in general; see however \cite[Th.~5.7]{P-R1}, where just such a description is given in special cases under the hypothesis of the correctness of the conjecture of DeConcini and Procesi on equations defining the closures of nilpotent conjugacy classes in $\mathfrak{gl}_n$.  To better focus on the issues at hand, we shall suppose henceforth that $F$ is \emph{totally ramified} over $F_0$, i.e.~that $e = d$.

Although there seems to be no simple moduli-theoretic description of $M^\loc_{G,\{\mu\},\L}$, there are at least two nontrivial descriptions of it that bear mention.  For the first, note that whenever $\L'$ is a subchain of $\L$, there is a natural forgetful morphism
\[
   \rho_{\L'}\colon M^\naive_{G,\{\mu\},\L} \to M^\naive_{G,\{\mu\},\L'}.
\]
In particular, for every lattice $\Lambda \in \L$,  we may consider its homothety class $[\Lambda] \subset \L$ and the projection $M^\naive_{G,\{\mu\},\L} \to M^\naive_{G,\{\mu\},[\Lambda]}$; here the target space corresponds to the \emph{maximal parahoric case,} as the stabilizer of $\Lambda$ in $G(F_0)$ is a maximal parahoric subgroup.  In \cite[\S8]{P-R1} it is proposed to describe $M^\loc_{G,\{\mu\},\L}$ by first taking the local model $M^\loc_{G,\{\mu\},[\Lambda]}$ in the sense of the previous paragraph for each homothety class $[\Lambda] \subset \L$, and then defining%
\footnote{Note that \cite{P-R1} uses the notation $M^\loc$ to denote what we call $\smash{M^\vert_{G,\{\mu\},\L}}$, which a priori is different from our definition of $M^\loc_{G,\{\mu\},\L}$.}
\[
   M^\vert_{G,\{\mu\},\L} := \bigcap_{[\Lambda]\subset \L} 
                        \rho^{-1}_{[\Lambda]} \bigl(M^\loc_{G,\{\mu\},[\Lambda]}\bigr).
\]
In the maximal parahoric case, the special fiber of $M^\loc_{G,\{\mu\},[\Lambda]}$ is integral and normal with only rational singularities \cite[5.4]{P-R1}.  On the other hand, G\"ortz \cite[Prop.~1]{Go4} has shown that $M^\vert_{G,\{\mu\},\L}$ is topologically flat.  These results can be combined to yield the following.

\begin{thm}[G\"ortz {\cite[\S1 Th.]{Go4}}; {\cite[7.3]{P-R2}}, {\cite[5.4]{P-R1}}]\label{st:ResGL_n_tot_ram}
For any $\mu$ as above and any periodic $\O_F$-lattice chain $\L$, $M^\vert_{G,\{\mu\},\L} = M^\loc_{G,\{\mu\},\L}$, that is, $M^\vert_{G,\{\mu\},\L}$ is flat over $\Spec \O_E$.  The special fiber of $M^\loc_{G,\{\mu\},\L}$ is reduced and its irreducible components are  normal with rational singularities, so in particular are Cohen-Macaulay.  When $\L$ consists of a single lattice homothety class, the special fiber of $M^\loc_{G,\{\mu\},\L}$ is moreover irreducible.\qed
\end{thm}

Note that if a moduli-theoretic description of the local models $M^\loc_{G,\{\mu\},[\Lambda]}$ can be found, then there would clearly also be a moduli-theoretic description of $ M^\vert_{G,\{\mu\},\L}$.  The definition of $M^\vert_{G,\{\mu\},\L}$ is closely related to the combinatorial notion of \emph{vertexwise admissibility,} which we shall take up in \S\ref{s:combinatorics}.\ref{ss:vwise_adm}.

The second description of $M^\loc_{G,\{\mu\},\L}$ makes use of the \emph{splitting model} $\M_{G,\{\mu\},\L}$ defined in \cite[\S5]{P-R2}.  We shall not recall the details of the definition here.  Roughly speaking, $\M_{G,\{\mu\},\L}$ is a projective scheme defined over the ring of integers $\O_{F^{\Gal}}$ in $F^{\Gal}$ which represents a rigidified version of the moduli problem defining $M^\naive_{G,\{\mu\},\L}$.  There are canonical morphisms
\[
   \M_{G,\{\mu\},\L} \to M^\naive_{G,\{\mu\},\L} \otimes_{\O_E} \O_{F^{\Gal}}
      \to M^\naive_{G,\{\mu\},\L}.
\]
The \emph{canonical local model $M^\can_{G,\{\mu\},\L}$} is defined to be the scheme-theoretic image in $M^\naive_{G,\{\mu\},\L}$ of the composite.  It is shown in \cite{P-R2} that the first displayed arrow is an isomorphism on generic fibers (the second is trivially an isomorphism after base change to $F^{\Gal}$, of course) and that $\M_{G,\{\mu\},\L}$ can be identified with a certain twisted product of local models for $GL_n$ over $\Spec \O_{F^{\Gal}}$, so that $\M_{G,\{\mu\},\L}$ is flat.  One then obtains the following.

\begin{thm}[{\cite[5.1, 5.3]{P-R2}}]
For any $\mu$ as above and any periodic $\O_F$-lattice chain $\L$, $M^\can_{G,\{\mu\},\L} = M^\loc_{G,\{\mu\},\L}$.\qed
\end{thm}

Note that although $M^\loc_{G,\{\mu\},\L}$ itself does not appear to admit a ready moduli-theoretic description, the theorem exhibits it as the image of a canonical morphism between schemes that do.

As pointed out by Haines, the splitting model can be used to give a second proof of flatness for $ M^\vert_{G,\{\mu\},\L}$ that bypasses part of the proof of topological flatness of G\"ortz.  See \cite[7.5]{P-R2} and \cite[\S5 Rem.]{Go4}.

\subsection{Weil restriction of $GSp_{2g}$}\label{ss:ResGSp}
% In this subsection we shall very briefly discuss a second instance of a Weil restriction whose local models have been studied in some detail, namely $G := \Res_{F/F_0}GSp_{2g}$.
% 
% A second example of a Weil restriction whose local models have been studied in some detail is $G := \Res_{F/F_0}GSp_{2g}$.  In this subsection we shall very briefly summarize what is known about local models for $G$.  We continue with the notation of the previous subsection.

In addition to Weil restrictions of $GL_n$, local models for Weil restrictions of $GSp_{2g}$ have also been studied in some detail.  In this subsection we shall very briefly survey their theory, outsourcing essentially all of the details to the papers \cite{P-R2} and \cite{Go4}.

Let $G:= \Res_{F/F_0} GSp_{2g}$, and otherwise continue with the assumptions and notation of the previous subsection.  For each $F_0$-embedding $\varphi\colon F \to F_0^\sep$, let $r_\varphi := g$. Let $\mu$ denote the resulting geometric cocharacter of $\Res_{F/F_0} GL_{2g}$, regard $\mu$ as a geometric cocharacter for $G$, and let $\{\mu\}$ denote the geometric conjugacy class of $\mu$ for $G$.  Then the reflex field of $\{\mu\}$ is $F_0$.  Let $\L$ be a periodic $\O_F$-lattice chain in $F^{2g}$ which is ``self-dual'' in the sense of \cite[\S8]{P-R4} or \cite[\S6]{Go4}.  The \emph{naive local model $M_{G,\{\mu\},\L}^\naive$} attached to $(G,\{\mu\},\L)$ is the closed subscheme of $M_{\Res_{F/F_0}GL_{2g}, \{\mu\},\L}^\naive$ whose points satisfy a perpendicularity condition relative to every pair of dual lattices in $\L$, in close analogy with the perpendicularity condition in \S\ref{s:examples}.\ref{ss:GSp_2g}; again see \cite[\S8]{P-R4} or \cite[\S6]{Go4}.

Essentially all the results in the previous subsection are known to carry over to the present setting.  For unramified extensions we have the following.

\begin{thm}[G\"ortz {\cite[\S2 Rem.\ (ii)]{Go2}}]
Suppose $F$ is unramified over $F_0$.  Then for any self-dual periodic $\O_F$-lattice chain $\L$, $M^\naive_{G,\{\mu\},\L}$ is flat over $\Spec \O_{F_0}$ with reduced special fiber.  The irreducible components of its special fiber are normal with rational singularities, so in particular are Cohen-Macaulay.\qed
\end{thm}

Let us suppose for the rest of the subsection that $F/F_0$ is totally ramified.  Then it is not known whether $M_{G,\{\mu\},\L}^\naive$ is flat (but see Conjecture \ref{st:ram_loc_mod_conj} below), and we define $M_{G,\{\mu\},\L}^\loc$ to be the scheme-theoretic closure in $M_{G,\{\mu\},\L}^\naive$ of its generic fiber.  In \cite[Display 12.2]{P-R2}, there is defined a natural ``vertexwise'' analog of $M_{\Res_{F/F_0}GL_{2g},\{\mu\},\L}^\vert$, which we denote by $M_{G,\{\mu\},\L}^\vert$ (this is denoted by $N_I^\loc$ in loc.\ cit.).  We then have the following.

\begin{thm}[G\"ortz {\cite[Prop.\ 3]{Go4}}; {\cite[Ths.\ 12.2, 12.4]{P-R2}}]\label{st:ResGSp}
For any self-dual periodic $\O_F$-lattice chain $\L$, $M^\vert_{G,\{\mu\},\L} = M^\loc_{G,\{\mu\},\L}$, that is, $M^\vert_{G,\{\mu\},\L}$ is flat over $\Spec \O_{F_0}$.  The special fiber of $M^\loc_{G,\{\mu\},\L}$ is reduced and its irreducible components are  normal with rational singularities, so in particular are Cohen-Macaulay.  When $\L$ is a minimal self-dual periodic lattice chain, the special fiber of $M^\loc_{G,\{\mu\},\L}$ is moreover irreducible.\qed
\end{thm}

G\"ortz's contribution to Theorem \ref{st:ResGSp} is to show that $M^\vert_{G,\{\mu\},\L}$ is topologically flat.  In fact he proves the following stronger result.

\begin{thm}[G\"ortz {\cite[Prop.\ 3]{Go4}}]
For any self-dual periodic $\O_F$-lattice chain $\L$, $M_{G,\{\mu\},\L}^\naive$ is topologically flat over $\Spec \O_{F_0}$.\qed
\end{thm}

Thus $M_{G,\{\mu\},\L}^\loc$, $M_{G,\{\mu\},\L}^\vert$, and $M_{G,\{\mu\},\L}^\naive$ all coincide at the level of topological spaces.  G\"ortz furthermore conjectures that they are equal on the nose.

\begin{conjecture}[G\"ortz {\cite[\S6 Conj.]{Go4}}]\label{st:ram_loc_mod_conj}
For any self-dual periodic $\O_F$-lattice chain $\L$, $M_{G,\{\mu\},\L}^\naive$ is flat over $\Spec \O_{F_0}$.
\end{conjecture}

Note that the conjecture stands in contrast to the case of $\Res_{F/F_0} GL_n$, where the naive local model may even fail to be topologically flat.

We finally mention that, in analogy with the previous subsection, the notions of {\it splitting model}  and {\it canonical local model} are also developed in \cite{P-R2} in the setting of local models for $G$.  We refer to loc.\ cit.\ for details, where, in particular, it is shown that the canonical local model equals $M_{G,\{\mu\},\L}^\loc$.

\subsection[Ramified, quasi-split unitary (types $(A_1,\varpi_1^\vee)$; $(A^{(2)}_{n-1}, \varpi^\vee_s)$, $1 \leq s \leq n-1$)]{Ramified, quasi-split unitary (types $(A_1,\varpi_1^\vee)$; $(A^{(2)}_{n-1}, \varpi^\vee_s)$, $1 \leq s \leq n-1$)%
\protect\footnote{Note that type $A_3^{(2)}$ does not actually appear in the table in \S\ref{s:motivationalsection}.\ref{ss:Shimuralocmod}.  Rather the adjoint group $PGU_4$ is of type $D_3^{(2)}$.}
}\label{ss:GU_n}
In this subsection we take up another typical example of a group that splits only after a ramified base extension, namely ramified, quasi-split $GU_n$.  We suppose $n \geq 2$ and $\charac k \neq 2$.
We continue with the notation of the previous subsection, but we now restrict to the special case that $F/F_0$ is ramified quadratic.  To simplify matters, assume that  $\pi \mapsto -\pi$  under the nontrivial automorphism of $F/F_0$, so that $\pi_0=\pi^2$ is a uniformizer of $F_0$. Let $\phi$ denote the $F/F_0$-Hermitian form on $F^n$ whose matrix with respect to the standard ordered basis is $H_n$ \eqref{disp:antidiag_1}.  We attach to $\phi$ the alternating $F_0$-bilinear form
\[
   \text{\aform}\colon
   \xymatrix@R=0ex{
      V \times V  \ar[r]  &  F_0\\
      (x,y)  \ar@{|->}[r]  &  \frac 1 2 \Tr_{F/F_0}\bigl(\pi^{-1}\phi(x,y)\bigr).
   }
\]
Given an $\O_F$-lattice in $F^n$, we denote by $\wh\Lambda$ its common \aform- and $\phi$-dual,
\[
   \wh\Lambda := \bigl\{\, x \in F^n \bigm| 
      \la \Lambda, x \ra \subset \O_{F_0} \,\bigr\}
      = \bigl\{\, x \in F^n \bigm| 
      \phi(\Lambda, x) \subset \O_F \,\bigr\}.
\]
As usual, \aform induces a perfect $\O_{F_0}$-bilinear pairing
\[
   \Lambda \times \wh\Lambda \to \O_{F_0}   
\]
for any $\O_F$-lattice $\Lambda$; and we say that an $\O_F$-lattice chain in $F^n$ is \emph{self-dual} if it is closed under taking duals.

Let $G := GU_n := GU(\phi)$ over $F_0$, and let $\L$ be a periodic self-dual $\O_F$-lattice chain in $F^n$.  Although we shall define local models for any such $\L$, when $n$ is even, facets in the building only correspond to $\L$ with the property that
\begin{altenumerate}
\item[$(*)$]\
   if $\L$ contains a lattice $\Lambda$ such that $\pi\Lambda \subset \wh\Lambda \subset \Lambda$ and $\dim_k \wh\Lambda/\pi\Lambda = 2$, then $\L$ also contains a lattice $\Lambda' \supset \Lambda$ with $\dim_k \Lambda'/\Lambda = 1$.
\end{altenumerate}
Such a $\Lambda'$ then satisfies $\wh \Lambda' = \pi \Lambda'$.  See \cite[\S4.a]{P-R3}, \cite[\S1.2.3]{P-R4}.

Over $F$ we have the standard splitting
\begin{equation}\label{disp:GU_splitting}
   G_F \xra[\sim]{(f,c)} GL_n \times \mG_m,
\end{equation}
where $c\colon G_F \to \mG_m$ is the similitude character and $f\colon G_F \to GL_n$ is given on $R$-points by the map on matrix entries
\[
   \xymatrix@R=0ex{
      R \otimes_{F_0} F \ar[r]  &  R\\
       x\otimes y \ar@{|->}[r]  & xy
   }
\]
for an $F$-algebra $R$.  Let $D$ denote the standard maximal torus of diagonal matrices in $GL_n$.  Choose a partition $n = r+s$; we refer to the pair $(r,s)$ as the \emph{signature}.  Let $\mu$ denote the cocharacter $\bigl(1^{(s)},0^{(r)};1\bigr)$ of $D \times \mG_m$.  Then we may regard $\mu$ as a geometric cocharacter of $G$ via \eqref{disp:GU_splitting}, and we denote by $\{\mu\}$ its geometric conjugacy class.  We denote by $E$ the reflex field of $\{\mu\}$; then $E = F_0$ if $r = s$ and $E = F$ otherwise.  Let $\O_E$ denote the ring of integers in $E$.

The \emph{naive local model $M^\naive_{G,\{\mu\},\L}$} is the functor on the category of $\O_E$-algebras that assigns to each $\O_E$-algebra $R$ the set of all families $(\F_\Lambda)_{\Lambda\in\L}$ such that
\begin{altenumerate}
\item
   for every $\Lambda \in \L$, $\F_\Lambda$ is an $(\O_F \otimes_{\O_{F_0}} R)$-submodule of $\Lambda \otimes_{\O_{F_0}} R$ which Zariski-locally on $\Spec R$ is a direct summand as an $R$-module of rank $n$;
\item
   for every inclusion of lattices $\Lambda \subset \Lambda'$ in $\L$, the induced map $\Lambda \otimes_{\O_{F_0}} R \to \Lambda' \otimes_{\O_{F_0}} R$ carries $\F_\Lambda$ into $\F_{\Lambda'}$;
\item
   for every $a \in F^\times$ and every $\Lambda\in\L$, the isomorphism $\smash{\Lambda \xra[\sim] a a\Lambda}$ identifies $\F_\Lambda \isoarrow \F_{a\Lambda}$;
\item
  for every $\Lambda \in \L$, the perfect $R$-bilinear pairing
   \[
      \bigl(\Lambda\otimes_{\O_{F_0}} R\bigr) \times 
         \bigl(\wh\Lambda \otimes_{\O_{F_0}} R\bigr) \to R
   \]
   induced by \aform identifies $\F_\Lambda^\perp \subset \wh\Lambda \otimes_{\O_{F_0}} R$ with $\F_{\wh\Lambda}$; 
\item
   (\emph{Kottwitz condition}) for every $a \in \O_F$ and every $\Lambda \in \L$, the element $a \otimes 1 \in \O_F \otimes_{\O_{F_0}} R$ acts on the quotient $(\Lambda \otimes_{\O_{F_0}} R)/\F_\Lambda$ as an $R$-linear endomorphism with characteristic polynomial
   \[
      {\rm char}_R \bigl(\, a \otimes 1 \bigm| (\Lambda\otimes_{\O_{F_0}} R)/\F_\Lambda\,\bigr) 
      = (X- a)^r(X-\ol a)^s,
   \]
   where we use a bar to denote the nontrivial automorphism of $F/F_0$.
   % 
   % for every $\Lambda\in\L$, we have an equality of polynomials with coefficients in $R$
   % \[
   %    {\textstyle \det_R} \bigl(X (1 \otimes 1) + Y  (\pi \otimes 1) 
   %    \bigm| 
   %    (\Lambda \otimes_{\O_{F_0}} R)/\F_\Lambda\bigr) = (X + Y \pi)^r(X - Y \pi)^s,
   % \]
   % where $1 \otimes 1$, $\pi \otimes 1 \in \O_F \otimes_{\O_{F_0}} R$.
\end{altenumerate}
   
   When $r = s$, the right-hand side of the last display is to be interpreted as $\bigl(X^2 -(a+\ol a) X + a\ol a\bigr)^s$.  The Kottwitz condition is equivalent to requiring the ``determinant'' condition that for every $\Lambda \in \L$, we have an equality of polynomials with coefficients in $R$
   \[
      {\textstyle \det_R} \bigl(\, X (1 \otimes 1) + Y  (\pi \otimes 1) 
      \bigm| 
      (\Lambda \otimes_{\O_{F_0}} R)/\F_\Lambda\,\bigr) = (X + Y \pi)^r(X - Y \pi)^s,
   \]
   where $1 \otimes 1$, $\pi \otimes 1 \in \O_F \otimes_{\O_{F_0}} R$; and these conditions are mutually equivalent to requiring that the single element $\pi \otimes 1$ acts on $(\Lambda\otimes_{\O_{F_0}} R)/ \F_\Lambda$ with characteristic polynomial $(X - \pi)^r(X + \pi)^s$.

% When $r = s$, the right-hand side of the last display is to be interpreted as $(X^2 - Y^2\pi_0)^s$.  The determinant condition is equivalent to requiring that for all $\Lambda \in \L$, the element $\pi \otimes 1$ acts on $\F_\Lambda$ as an $R$-linear endomorphism with characteristic polynomial $(X - \pi)^r(X + \pi)^s$.

As always, the naive local model is representable by a closed subscheme of a finite product of Grassmannians over $\Spec \O_E$.  If we denote by $V$ the $n$-dimensional $F$-vector space
\[
   V := \ker(\,\pi \otimes 1 - 1 \otimes \pi \mid F^n \otimes_{F_0} F\,),
\]
then the map
\[
   (\F_\Lambda)_\Lambda \mapsto \ker(\,\pi \otimes 1 - 1 \otimes \pi \mid \F_\Lambda\,)
\]
(independent of $\Lambda$) defines an isomorphism from the $F$-generic fiber $M^\naive_{G,\{\mu\},\L} \otimes_{\O_E} F$ onto the Grassmannian $\Gr(s,V)_F$.

It was observed in \cite{P1} that $M^\naive_{G,\{\mu\},\L}$ fails to be flat in general; historically,  this was the first time it was found that the Rapoport--Zink local model can fail to be flat.  The key point is that the Kottwitz condition fails to impose a condition on the reduced special fiber.  Indeed, if $R$ is a $k$-algebra, then $\pi \otimes 1$ is nilpotent in $\O_F \otimes_{\O_{F_0}} R$.  Hence, when $R$ is reduced, $\pi \otimes 1$ necessarily acts on $(\Lambda \otimes_{\O_{F_0}} R)/\F_\Lambda$ with characteristic polynomial $X^n$, in accordance with the Kottwitz condition.  Thus the reduced special fiber is \emph{independent} of the signature. %
%
% \footnote{Analogously, the proof of \cite[Prop.\ 3.8(b)]{P1} should be corrected to say that\[
%    \det(T_0 \cdot 1 + T_1 \cdot \pi \mid Q) = T_0^n
% \]
% when $S$ is \emph{reduced,} so that the determinant condition becomes redundant on the scheme $(\A_{C^p}^{r,s} \times_{\Spec R} \Spec \mF_p)_{\rm red}$, and the map $i^{r,s} \times_{\Spec \O_\P} \Spec \mF_p$ is an isomorphism on reduced underlying subschemes.  Of course the proof of loc.\ cit.\ still goes through after this correction.}
%
Hence by Chevalley's theorem (EGA IV.13.1.5), the special fiber has dimension
\[
   \geq \max\{\dim M^\naive_{G,\{\mu\},\L} \otimes_{\O_E} E \}_{0 \leq s \leq n} = \max\{\dim \Gr(s,V) \}_{0 \leq s \leq n} = \Bigl\lfloor \frac n 2 \Bigr\rfloor \Bigl\lceil \frac n 2 \Bigr\rceil.
\]
The max in the display is achieved for $|r-s| \leq 1$.  Thus $M^\naive_{G,\{\mu\},\L}$ is not flat for $|r-s| > 1$, as its generic and special fibers have different dimension.  We note that the analogous argument given in the proof of \cite[Prop.\ 3.8(b)]{P1} should be amended to use the reduced special fiber in place of the honest special fiber.
% 
%  Hence  the generic and special fibers of $M^\naive_{G,\{\mu\},\L}$ have different dimensions unless $|r-s| \leq 1$, contradicting  Chevalley's theorem 
% (EGA IV.13.1.5) if $|r-s| > 1$.

As always, one remedies for non-flatness of the naive local model by defining the honest local model $M^\loc_{G,\{\mu\},\L}$ to be the scheme-theoretic closure in $M^\naive_{G,\{\mu\},\L}$ of its generic fiber.  Although less is known about $M^\loc_{G,\{\mu\},\L}$ for ramified $GU_n$ than for ramified $\Res_{F/F_0}GL_n$ and $\Res_{F/F_0} GSp_{2g}$, there are by now a number of results that have been obtained in various special cases.  In low rank, the case $n = 3$ has been completely worked out.

\begin{thm}[{\cite[4.5, 4.15]{P1}}, {\cite[\S6]{P-R4}}]
\label{st:GU(2,1)}
Let $n=3$ and $(r,s) = (2,1)$.
\begin{altenumerate}
\item\label{it:GU(2,1)_0}
   Let $\L$ be the homothety class of the lattice $\Lambda_0 = \O_F^n \subset F^n$.  Then $M^\naive_{G,\{\mu\},\L} = M^\loc_{G,\{\mu\},\L}$, that is, $M^\naive_{G,\{\mu\},\L}$ is flat over $\Spec \O_F$.  Moreover, $M^\naive_{G,\{\mu\},\L}$ is normal and Cohen-Macaulay, it is smooth outside a single point $y$ in its special fiber, and its special fiber is integral and normal and has a rational singularity at $y$.  The blowup $\wt M^\loc_{G,\{\mu\},\L} \to M^\loc_{G,\{\mu\},\L}$ at $y$ is regular with special fiber a reduced union of two smooth surfaces meeting transversely along a smooth curve.
\item
   Let $\L = [\Lambda_1, \Lambda_{2}]$, the lattice chain consisting of the homothety classes of $\Lambda_1$ and $\Lambda_{2}$.  Then $M^\loc_{G,\{\mu\},\L}$ is smooth over $\Spec \O_F$ with geometric special fiber isomorphic to $\mP^2$.
\item
   Let $\L$ be the standard maximal lattice chain in $F^3$.  Then $M^\loc_{G,\{\mu\},\L}$ is normal and Cohen-Macaulay.  Its special fiber is reduced and consists of two irreducible components, each normal and with only rational singularities, which meet along two smooth curves which, in turn, intersect transversally at a point.\qed
\end{altenumerate}
\end{thm}

We shall discuss the case $n=2$ at the end of the subsection in Remark \ref{rk:GU2}.

\begin{Remark}
In each case in the theorem, the stabilizer of $\L$ in $G(F_0)$ is a parahoric subgroup.  In this way the three cases correspond to the three conjugacy classes of parahoric subgroups in $G(F_0)$.  See \cite[\S1.2.3(a)]{P-R4}.
\end{Remark}

\begin{Remark}\label{rk:switch}
Quite generally, for any fixed $n$ and $\L$, it is elementary to verify that $M^\naive_{G,\{\mu\},\L}$ (and hence $M^\loc_{G,\{\mu\},\L}$) is unchanged up to isomorphism if we replace the signature $(r,s)$ with $(s,r)$.  Moreover, it is easy to see from the \emph{wedge condition} discussed below that $M^\loc_{G,\{\mu\},\L}$ is just $\Spec \O_F$ itself in case $r$ or $s$ is $0$.  So the theorem covers all cases of interest when $n=3$.
\end{Remark}

For larger $n$, results on $M^\loc_{G,\{\mu\},\L}$ are known for cases of simple signature and for cases of simple lattice chains $\L$. One important tool for proving reducedness of the special fiber is Hironaka's Lemma (EGA IV.5.12.8). 

\begin{thm}\label{st:GU_simple_cases}
\begin{altenumerate}
\item
   \emph{(\cite[Th.~5.1]{P-R4}; Arzdorf \cite[Th.~2.1]{A}, Richarz \cite[Cor. 5.6]{Ri})} \label{it:spec_parahoric}  Let $n \geq 3$.  Suppose that $n$ is even and $\L = [\Lambda_{n/2}]$, or that $n = 2m+1$ is odd and $\L = [\Lambda_0]$ or $\L = [\Lambda_m, \Lambda_{m+1}]$.  Then for any signature $(r, s)$, the special fiber of $M^\loc_{G,\{\mu\},\L}$ is integral and normal and has only rational singularities.
\item\label{it:GU_Lambda_0}
   \emph{(\cite[4.5]{P1})}  Let $n \geq 2$, $(r,s) = (n-1,1)$, and $\L = [\Lambda_0]$.  Then $M^\loc_{G,\{\mu\},\L}$ is normal and Cohen-Macaulay, and it is smooth over $\Spec \O_E$ outside a single point $y$ in the special fiber.  For $n = 2$, $M^\loc_{G,\{\mu\},\L}$ is regular and its special fiber is a divisor with simple normal crossings.  For $n \geq 3$, the blowup $\wt M^\loc_{G,\{\mu\},\L} \to M^\loc_{G,\{\mu\},\L}$ at $y$ is regular with special fiber a divisor with simple normal crossings.
\item\label{it:GU_Lambda_m}
   \emph{(\cite[\S5.3]{P-R4}; Richarz \cite[Prop.~4.16]{A}% , \cite[Thm. 3.14 (ii)]{Ri}
)} \label{it:spec_parahoric_smooth}
Let $n \geq 3$ and $(r,s) = (n-1,1)$.  Suppose that $n$ is even and $\L = [\Lambda_{n/2}]$ or that $n = 2m+1$ is odd and $\L = [\Lambda_m, \Lambda_{m+1}]$.  Then $M^\loc_{G,\{\mu\},\L}$ is smooth.\qed
\end{altenumerate}
\end{thm}

\begin{Remark}
In \eqref{it:spec_parahoric}, the cases $n$ even, $\L = [\Lambda_{n/2}]$ and $n$ odd, $\L = [\Lambda_0]$ are in \cite{P-R4}, and the other is due to Arzdorf.  A different proof is due to Richarz \cite{Ri}. The significance of the assumptions on $n$ and $\L$ is that, up to $G(F_0)$-conjugacy, these are all the cases that correspond to \emph{special maximal parahoric level structure,} i.e.\ the parahoric stabilizer of $\L$ in $G(F_0)$ is the parahoric 
subgroup corresponding to a vertex in the building which is \emph{special} in the sense of Bruhat--Tits theory.  See \cite[\S1.2.3]{P-R4}. 
\end{Remark}

\begin{Remark}
The blowup $\wt M^\loc_{G,\{\mu\},\L}$ occurring in \eqref{it:GU_Lambda_0} is described explicitly by Kr\"amer in \cite{Kr} in terms of a moduli problem analogous to the Demazure resolution of a Schubert variety in the Grassmannian.  She shows that the special fiber of $\wt M^\loc_{G,\{\mu\},\L}$ consists of two smooth irreducible components of dimension $n-1$ --- one of which, the fiber over $y$, being isomorphic to $\mP_k^{n-1}$, the other one being a $\mathbb P_k^1$-bundle over a smooth quadric --- which intersect transversely in a smooth irreducible variety of dimension $n-2$.
\end{Remark}

\begin{Remark}
In \eqref{it:spec_parahoric_smooth}, the case of $n$ even is in \cite{P-R4}, and the case of $n$ odd is due to Richarz.  The result in the former case is not directly stated in \cite{P-R4}, but it follows from the cited reference, where it is shown that $M^\loc_{G,\{\mu\},[\Lambda_{n/2}]}$ has an open neighborhood around its ``worst point'' isomorphic to $\mA^{n-1}_{\O_F}$ (note that in the last sentence of \cite[\S5.3]{P-R4}, ${}_1 U_{r,s}$ should be replaced by ${}_1 U_{r,s}^\wedge$).

In the cases $n$ even, $\L = [\Lambda_{n/2}]$ and $n$ odd, $\L = [\Lambda_0]$, the local models are never smooth outside the cases enumerated in \eqref{it:GU_Lambda_m} (up to switching $(r, s)$ and $(s, r)$, cf. Remark \ref{rk:switch}), provided that the signature is nontrivial, i.e.,\ $r\neq 0$ or $s\neq 0$, see \cite[Th.\ 3.15]{Ri1}. Probably the same holds for the case $n = 2m+1$ is odd and $\L = [\Lambda_m, \Lambda_{m+1}]$.
\end{Remark}

In light of the failure of $M^\naive_{G,\{\mu\},\L}$ to be flat in general, it is an interesting problem to obtain a moduli-theoretic description of $M^\loc_{G,\{\mu\},\L}$.  
Motivated by the Kottwitz condition's failure to impose a condition on the reduced special fiber, 
in \cite{P1} the following additional  condition is introduced to the moduli problem defining $M^\naive_{G,\{\mu\},\L}$:
\begin{altenumerate}
\item [(vi)]\ 
   (\emph{wedge condition}) if $r \neq s$, then for every $\Lambda \in \L$, we have
   \[
      \sideset{}{_R^{s+1}}{\bigwedge} (\pi\otimes 1 + 1 \otimes \pi \mid \F_i) = 0
      \quad\text{and}\quad
      \sideset{}{_R^{r+1}}{\bigwedge} (\pi\otimes 1 - 1 \otimes \pi \mid \F_i) = 0.
   \]
   (There is no condition when $r=s$.)
\end{altenumerate}

The \emph{wedge local model $M^\wedge_{G,\{\mu\},\L}$} is the closed subscheme of $M^\naive_{G,\{\mu\},\L}$ that classifies points satisfying the wedge condition.  It is easy to see that the wedge and naive local models have common generic fiber, and under the special hypotheses of Theorem \ref{st:GU_simple_cases}\eqref{it:GU_Lambda_0} it has been shown that the wedge condition cuts out the flat closure $M^\loc_{G,\{\mu\},\L}$.

\begin{prop}[{\cite[4.5]{P1}}]\label{st:M^wedge=M^loc}
Let $n \geq 2$, $(r,s) = (n-1,1)$, and $\L = [\Lambda_0]$.  Then $M^\wedge_{G,\{\mu\},\L} = M^\loc_{G,\{\mu\},\L}$.\qed
\end{prop}

More generally, it is conjectured in \cite{P1} that $M^\wedge_{G,\{\mu\},[\Lambda_0]}$ is flat for any $n$ and any signature.%
\footnote{We mention that it is also conjectured that $M^\naive_{G,\{\mu\},[\Lambda_0]}$ coincides with $M^\wedge_{G,\{\mu\},[\Lambda_0]}$ and $M^\loc_{G,\{\mu\},[\Lambda_0]}$ for $|r-s| \leq 1$.  This is proved for $n$ equal to $2$ and $3$ in \cite[4.5, 4.15]{P1}.}
But for more general lattice chains, the wedge condition turns out to be insufficient \cite[Rems.~5.3, 7.4]{P-R4}.  For example, for $n = 3$ and $(r,s) = (2,1)$, the schemes $M^\wedge_{G,\{\mu\},\L}$ for $\L = [\Lambda_1,\Lambda_2]$ and $\L$ the standard lattice chain are topologically flat but not flat. And for $n$ even and $r \neq 0$ or $s \neq 0$, the scheme $M^\wedge_{G,\{\mu\},[\Lambda_{n/2}]}$ is not even topologically flat.

\begin{Remark}\label{locverswedge}
When $n$ is even,  \cite{P-R4} only shows that $M^\wedge_{G,\{\mu\},[\Lambda_{n/2}]}$ is not topologically flat for $r$ and $s$ odd. But the same holds for $r$ and $s$ even, provided neither is $0$:  for example, the point denoted $\F_1$ in \cite[\S5.3]{P-R4} is not in the closure of the generic fiber in this case.  Accordingly, for $n$ even, \cite[Rem.~5.3(a)]{P-R4} should be corrected to say that Conjecture 5.2 in loc.\ cit.\ implies that $M^\wedge_{G,\{\mu\},[\Lambda_{n/2}]}$ contains $M^\loc_{G,\{\mu\},[\Lambda_{n/2}]}$ as an open subscheme for any signature, not that $M^\wedge_{G,\{\mu\},[\Lambda_{n/2}]} = M^\loc_{G,\{\mu\},[\Lambda_{n/2}]}$ for $r$ and $s$ even.    (Here the corrected statement allows for $r$ and $s$ to be odd as well as even, since the odd case reduces to the even case, as follows from \cite[\S5.3]{P-R4}.) See Conjecture \ref{conjmatunit} below for a statement of Conjecture 5.2 in loc.\ cit. 

More precisely, one verifies at once that the perfect pairing
\[
   \Lambda_{n/2} \times \Lambda_{n/2} \xra[\sim]{\id \times \pi}
      \Lambda_{n/2} \times \Lambda_{-n/2} \xra{\text{\aform}} \O_{F_0}
\]
is \emph{split symmetric}.  Hence $M^\naive_{G,\{\mu\},[\Lambda_{n/2}]}$ naturally embeds as a closed subscheme of $\OGr(n,2n)_{\O_E}$.  Then \cite[Conj.~5.2]{P-R4},  together with the topological flatness result Theorem \ref{st:GU_top_flat} below,  implies that $M^\loc_{G,\{\mu\},[\Lambda_{n/2}]}$ is the intersection of $M^\wedge_{G,\{\mu\},[\Lambda_{n/2}]}$ with the connected component of $\OGr(n,2n)_{\O_E}$ marked by the common generic fiber of $M^\naive_{G,\{\mu\},[\Lambda_{n/2}]}$ and $M^\wedge_{G,\{\mu\},[\Lambda_{n/2}]}$.
\end{Remark}

Although the wedge condition is not sufficient in general to cut out the local model inside $M^\naive_{G,\{\mu\},\L}$, one can still hope to describe $M^\loc_{G,\{\mu\},\L}$ via a further refinement of the moduli problem.  In \cite{P-R4} it is shown that, in addition to the wedge condition, $M^\loc_{G,\{\mu\},\L}$ satisfies a close analog of the spin condition that arose in the setting of even orthogonal groups in \S\ref{s:examples}.\ref{ss:GO_2g}, which is again called the \emph{spin condition}.  In the setting of Remark \ref{locverswedge}, with $n$ even and $\L = [\Lambda_{n/2}]$, the spin condition amounts exactly to intersecting $M_{G,\{\mu\},[\Lambda_{n/2}]}^\naive$ with the connected component of $\OGr(n,2n)_{\O_E}$ marked by the generic fiber of $M_{G,\{\mu\},[\Lambda_{n/2}]}^\naive$.  In general the spin condition is more complicated, and we shall just refer to the source papers for its formulation: see \cite[\S7.2]{P-R4} or \cite[\S2.5]{Sm3} (the latter contains a correction to a minor sign error in the former).  As in the orthogonal case, we denote by $M^\spin_{G,\{\mu\},\L}$ the closed subscheme of $M^\wedge_{G,\{\mu\},\L}$ that classifies points satisfying the spin condition, and we have the following.

\begin{conjecture}[{\cite[Conj.~7.3]{P-R4}}]\label{conj:GU_top_flat}
Let $\L$ be a periodic self-dual $\O_F$-lattice chain, satisfying property $(*)$ from the beginning of the subsection if $n$ is even.  Then for any $n \geq 3$ and any signature, $M^\spin_{G,\{\mu\},\L} = M^\loc_{G,\{\mu\},\L}$, that is, $M^\spin_{G,\{\mu\},\L}$ is flat over $\Spec \O_E$.
\end{conjecture}

Although the conjecture remains open, there is the following result, in analogy with the orthogonal case.

\begin{thm}[{\cite[Main Th.]{Sm3}}, \cite{Sm4}]\label{st:GU_top_flat}
For any $n \geq 3$, any signature, and any $\L$ as in Conjecture \ref{conj:GU_top_flat}, $M^\spin_{G,\{\mu\},\L}$ is topologically flat over $\Spec \O_E$.  If $n$ is odd, then $M^\wedge_{G,\{\mu\},\L}$ is also topologically flat.\qed
\end{thm}

In the special case $n = 2m+1$ is odd, $(r,s) = (n-1,1)$, and $\L = [\Lambda_m,\Lambda_{m+1}]$ mentioned in Theorem  \ref{st:GU_simple_cases}\eqref{it:GU_Lambda_m}, topological flatness of $M^\wedge_{G,\{\mu\},\L}$ also follows from \cite[Prop.~4.16]{A}.

We emphasize that for odd $n$, although $M^\spin_{G,\{\mu\},\L}$ and $M^\wedge_{G,\{\mu\},\L}$ coincide as topological spaces, their scheme structures really do differ in general, and it is only the spin local model that is conjectured to be flat.  By contrast, for even $n$, $M^\spin_{G,\{\mu\},\L}$ and $M^\wedge_{G,\{\mu\},\L}$ typically do not even agree at the level of topological spaces; see \cite{Sm4}.

\begin{Remark}[$GU_2$]\label{rk:GU2}
To be able to treat the Bruhat-Tits-theoretic aspects of $GU_n$ in a uniform way, the paper \cite{P-R4} omits the case $n = 2$.  Let us briefly discuss it now.  The only nontrivial signature to worry about is $(r,s) = (1,1)$.  In this case the naive and wedge local models coincide and are defined over $\Spec \O_{F_0}$.  The derived group $SU_2$ is isomorphic to $SL_2$, which is split.  Each alcove in the building has two vertices, both of which are special and $GU_2(F_0)$-conjugate.  Thus there are essentially two cases to consider: the (special) maximal parahoric case and the Iwahori case.

First take $\L = [\Lambda_1]$.  Then the stabilizer in $GU_{2}(F_0)$ of $\Lambda_1$ is a maximal parahoric subgroup.  The naive local model is a closed subscheme of $\Gr(2,\Lambda_1)_{\O_{F_0}}$, and by restricting standard open affine charts of the Grassmannian, \cite[\S5.3]{P-R4} computes two affine charts on $M_{GU_2,\{\mu\},[\Lambda_1]}^\naive$.  (Although, strictly speaking, \cite{P-R4} makes the blanket assumption $n \geq 3$, the calculations in loc.\ cit.\ still go through for $n = 2$.)  The first chart $U_{1,1}$ identifies with the scheme of all $2 \times 2$ matrices $X$ such that
\[
   X^2 = \pi_0\cdot \mathrm{Id}, \quad
   X^t = -JXJ,
   \quad\text{and}\quad
   \charac_X(T) = T^2 - \pi_0,
\]
where
\[
   J =
   \begin{pmatrix}
      0 & -1\\
      1 & 0
   \end{pmatrix}.
\]
One easily solves these equations to find that $U_{1,1} = \Spec k$.  The second affine chart $_1U_{1,1}$ identifies with the scheme of all $2\times 2$ matrices $X$ such that
\[
   X^t = -JXJ,
\]
which, as noted in loc.\ cit., is the scheme of all scalar matrices $X$.  Hence $_1U_{1,1} \cong \mA_{\O_{F_0}}^1$.  By restricting the remaining standard affine charts on the Grassmannian to the local model, one finds that globally
\[
   M_{GU_2,\{\mu\},[\Lambda_1]}^\naive = M_{GU_2,\{\mu\},[\Lambda_1]}^\wedge \cong \mP_{\O_{F_0}}^1 \amalg \Spec k.
\]
Of course this scheme is not flat because of the copy of $\Spec k$.  As described in Remark \ref{locverswedge} and the paragraph preceding Theorem \ref{st:GU_top_flat}, $M_{GU_2,\{\mu\},[\Lambda_1]}^\naive$ is actually contained in $\OGr(2,\Lambda_1)_{\O_{F_0}}$ inside $\Gr(2,\Lambda_1)_{\O_{F_0}}$,
and imposing the spin condition amounts to intersecting $M_{GU_2,\{\mu\},[\Lambda_1]}^\naive$ with the connected component of $\OGr(2,\Lambda_1)_{\O_{F_0}}$ marked by $M_{GU_2,\{\mu\},[\Lambda_1]}^\naive \otimes_{\O_{F_0}} F_0$.  In this way the spin condition visibly eliminates the extraneous copy of $\Spec k$.

For the Iwahori case we take $\L$ to be the standard lattice chain.  To warm up, let us consider the naive local model associated just to the homothety class $[\Lambda_0]$, without worrying about the functoriality conditions attached to the inclusions $\Lambda_0 \subset \Lambda_1$ and $\Lambda_1 \subset \pi^{-1} \Lambda_0$.  Let $x$ denote the $k$-point on $M_{GU_2,\{\mu\},[\Lambda_0]}^\naive$ given by
\[
   (\pi \otimes 1) \cdot (\Lambda_0 \otimes_{\O_{F_0}}k) \subset \Lambda_0 \otimes_{\O_{F_0}}k.
\]
An affine chart for $M_{GU_2,\{\mu\},[\Lambda_0]}^\naive$ around $x$ is described in \cite[p.\ 596--7]{P1}: it is the scheme of all $2 \times 2$ matrices $X$ such that%
\footnote{Note that \cite{P1} has the condition $X^t = X$ instead of $X^t = H_2 X H_2$, owing to how the form $\psi_p$ is defined there.}
\[
   X^2 = \pi_0 \cdot \mathrm{Id},\quad
   X^t = H_2XH_2,
   \quad\text{and}\quad
   \charac_X(T) = T^2 - \pi_0,
\]
where as always $H_2$ is the antidiagonal unit matrix \eqref{disp:antidiag_1}. Writing
\[
   X = 
   \begin{pmatrix}
      x_{11} & x_{12}\\ 
      x_{21} & x_{22}
   \end{pmatrix},
\]
one finds that this chart is given by $\Spec\O_{F_0}[x_{12},x_{21}]/(x_{12}x_{21} - \pi_0)$.  Thus we find semistable reduction; in fact the global special fiber consists of two copies of $\mP^1_k$ meeting at the point $x$.

The full local model $M_{GU_2,\{\mu\},\L}^\loc$ can now be obtained from $M_{GU_2,\{\mu\},[\Lambda_0]}^\loc$ and $M_{GU_2,\{\mu\},[\Lambda_1]}^\loc$ by imposing the functoriality conditions attached to the inclusions $\Lambda_0 \subset \Lambda_1$ and $\Lambda_1 \subset \pi^{-1}\Lambda_0$.  We leave it to the reader to verify that in fact $M_{GU_2,\{\mu\},\L}^\loc \cong M_{GU_2,\{\mu\},[\Lambda_0]}^\loc$, i.e.\ the submodule $\F_{\Lambda_0}$ uniquely determines $\F_{\Lambda_1}$, without constraint.  This fact admits a building-theoretic interpretation: the stabilizer of $\Lambda_0$ in $GU_2(F_0)$ is not a maximal parahoric subgroup, but rather, after passing to its connected component, the Iwahori subgroup fixing the entire standard chain.  Finally, note that the local model obtained in each of our two cases is isomorphic to the local model for $GL_2$ in the analogous case, cf.\ Example \ref{exn=2}.
\end{Remark}

\subsection{Quasi-split but nonsplit orthogonal (types $(D_{g}^{(2)}, \varpi^\vee_{g-1})$, $(D_{g}^{(2)}, \varpi^\vee_{g})$)}

Again assume $\charac k \neq 2$ and let $n = 2g-2$, $n\geq 4$.  Let $V$ be the $2g$-dimensional $F$-vector space on the ordered basis $e$, $f$, $e_1,\dotsc,$ $e_{2g-2}$, and let \sform denote the symmetric $F$-bilinear form on $V$ whose matrix with respect to this basis is%
\footnote{It follows from Springer's Theorem \cite[Th.\ VI.1.4]{Lam} that, after passing to a sufficiently big unramified extension of $F$, any symmetric bilinear form on $F^{2g}$ becomes isomorphic to \sform or to the split form \sform considered in \S\ref{s:examples}.\ref{ss:GO_2g}.}
\[
   \left(
   \begin{array}{cc|c}
      \pi & &\\
        & 1 & \\
       \hline
       &  & H_{2g-2}
   \end{array}
   \right),
\]
with $H_{2g-2}$ the anti-diagonal matrix \eqref{disp:antidiag_1}.  As always, we denote by $\wh\Lambda$ the \sform-dual of any $\O_F$-lattice $\Lambda$ in $V$, and \sform induces a perfect pairing $\Lambda \times \wh\Lambda \to \O_F$. 

Let $G := GO\bigl(\text{\sform}\bigr)$ over $F$.  Then $G$ is quasi-split but not split. Consider the  cocharacter $\bigl(1^{(g)}, 0^{(g)}\bigr)$ of  $G_{\ol F}\simeq GO_{2g}$ given as in \S \ref{s:examples}.\ref{ss:GO_2g}, and let $\{\mu\}$ denote its $G(\ol F)$-conjugacy class over $\ol F$.  Let 
$\L$ be a periodic lattice chain in $F^{2g}$ which is self-dual for the form \sform.  The \emph{naive local model $M^\naive_{G,\{\mu\},\L}$} is the closed  $\O_F$-subscheme of $M^\loc_{GL_{2g},\{\mu\},\L}$ defined in the exactly the same way as for $GSp_{2g}$ and split $GO_{2g}$, that is, we impose the duality condition \eqref{it:lm_duality} with the understanding that all notation is taken with respect to \sform.
Once again, $M^\naive_{G,\{\mu\},\L}$ has generic fiber $\OGr\bigl(\sform \bigr)_F$, the orthogonal Grassmannian of totally isotropic $g$-planes in $F^{2g}$ for the form \sform. This Grassmannian is now connected 
although not geometrically connected. We can see \cite[8.2.1]{P-R4} that   $\OGr\bigl(\sform \bigr)_F$   supports a canonical morphism to $\Spec K$ where $K=F(\sqrt{D})$ is the ramified quadratic extension of $F$ obtained by extracting a square root of the discriminant $D=(-1)^g\pi$. (Let us remark here that the form \sform splits over $K$.) The base change $M_{G,\{\mu\},\L}^\naive\otimes_{\O_F}K$  is the split  orthogonal Grassmannian ${\rm OGr}(g, 2g)_K$ which has two (geometrically) connected components. 

There is enough computational evidence to suggest the following.

\begin{conjecture}
The scheme $M_{G,\{\mu\}, \L}^\naive$ is topologically flat over $\Spec \O_F$.
\end{conjecture}

This is in contrast to the case of split $GO_{2g}$.
However,  the naive local model $M_{G,\{\mu\}, \L}^\naive$ is still typically not flat and a  version of the spin condition is 
 needed.  This condition is explained in \cite[8.2]{P-R4} where  the reader can also find the conjecture that
the corresponding spin local model $M_{G, \{\mu\}, \L}^{\rm spin}$ is flat. In fact, in this case, 
$M_{G, \{\mu\}, \L}^{\rm spin}$ is naturally an $\O_K$-scheme. The local model $M_{G,\{\mu\}, \L}^{\rm loc}$ is 
by definition the flat closure of $M_{G,\{\mu\}, \L}^\naive\otimes_{\O_F}K$ in $M_{G,\{\mu\}, \L}^\naive\otimes_{\O_F}\O_K$
and is also naturally an $\O_K$-scheme.
Except for the results of a few calculations not much is known in this case. For more details, we refer to loc.\ cit. 

Let us remark here that, similarly to the example in the split case of \S  \ref{s:examples}.\ref{ss:GO_2g},
the group $G$ does not fit neatly into our framework of \S \ref{s:motivationalsection} since $G$ is not connected. To define corresponding 
local models for the connected quasi-split group $PGO^\circ\bigl(\sform\bigr)$ of type $D^{(2)}_g$ and the cocharacters given as above, we can argue as follows:  
First note that the reflex field in this case is the quadratic ramified extension $K$ of $F$ as above.
As above, the generic fiber $M_{G,\{\mu\}, \L}^\naive\otimes_{\O_F}F$ supports a canonical morphism to $\Spec K$. 
We can now consider the flat closures of the two components of the orthogonal Grassmannian ${\rm OGr}(g, 2g)_K=M_{G,\{\mu\}, \L}^\naive\otimes_{\O_F}K$ in $M_{G,\{\mu\}, \L}^\naive\otimes_{\O_F}\O_K$. These two schemes over $\Spec \O_K$   give by definition the local models for $PGO^\circ\bigl(\sform\bigr)$ and the two PEL minuscule cocharacters $\varpi^\vee_{g-1}$, $\varpi_{g}^\vee$.

\section{Local models and flag varieties for loop groups}\label{s:afv}

A basic technique in the theory of local models, introduced by G\"ortz \cite{Go1}, is to embed the special fiber of the local model into an appropriate affine flag variety.  In this section we discuss this and related matters, focusing on the representative examples of the linear and symplectic groups.  Throughout this section we denote by $k$ a field, by $K := k((t))$ the field of Laurent series in $t$ with coefficients in $k$, and by $\O_K := k[[t]]$ the subring of $K$ of power series.

\subsection{Affine flag varieties}\label{ss:afv}
For any contravariant functor $G$ on the category of $K$-algebras, we denote by $LG$ the functor on $k$-algebras
\[
   LG\colon R \mapsto G\bigl( R((t)) \bigr),
\]
where we regard $R((t))$ as a $K$-algebra in the obvious way.  Similarly, for any contravariant functor $P$ on the category of $\O_K$-algebras, we denote by $L^+P$ the functor on $k$-algebras
\[
   L^+P\colon R \mapsto P\bigl( R[[t]] \bigr).
\]
If $P$ is an affine $\O_K$-scheme, then $L^+P$ is an affine $k$-scheme.  If $G$ is an affine $K$-scheme, then $LG$ is an ind-scheme expressible as the colimit of a filtered diagram of closed immersions between affine $k$-schemes.

For our purposes, we shall be interested in the case that $G$ and $P$ are \emph{group-valued functors;} then we call $LG$ and $L^+P$ the \emph{loop group} and \emph{positive loop group} attached to $G$ and $P$, respectively.  Given such $P$, let $P_\eta$ denote its generic fiber, and consider the fpqc quotient
\[
   \F_P := LP_\eta/L^+P,
\]
or in other words, the fpqc sheaf on $k$-algebras associated to the presheaf
\[
   R \mapsto P\bigl( R((t)) \bigr) \big/ P\bigl( R[[t]] \bigr).
\]
When $P$ is smooth and affine, the following is a basic structure result on $\F_P$. It generalizes results of Faltings \cite{Fa4}.

\begin{thm}[{\cite[Th.\ 1.4]{P-R3}}]
Let $P$ be a smooth affine group scheme over $\O_K$.  Then $\F_P$ is representable by an ind-scheme of ind-finite type over $k$, and the quotient morphism $LP_\eta \to \F_P$ admits sections locally in the \'etale topology.\qed
\end{thm}

Recall that an ind-scheme over $k$ is \emph{of ind-finite type} if it is expressible as a filtered colimit of $k$-schemes of finite type.

For applications to local models, we are mainly interested in the sheaf $\F_P$ in the case that $G = P_\eta$ is a \emph{connected reductive group} over $K$ and $P$ is a \emph{parahoric group scheme}.  Let us elaborate.  Let $G$ be a connected reductive group over the $t$-adically valued field $K$, let $G_\ad$ denote its adjoint group, and consider the Bruhat--Tits building $\B := \B\bigl( G_\ad(K) \bigr)$.  Let $\mathbf f$ be a facet in $\B$.  Then Bruhat--Tits theory attaches to the pair $(G,\mathbf f)$ the parahoric group scheme $P_{\mathbf f}$; this is a smooth affine $\O_K$-scheme with generic fiber $G$, with connected special fiber, and whose $\O_K$-points are identified with the corresponding parahoric subgroup of $G(K)$.  We make the following definition.

\begin{Definition}\label{def:afv}
Given a facet $\mathbf f$ in the building of the adjoint group of the connected reductive $K$-group $G$, the \emph{affine flag variety relative to $\mathbf f$} (or \emph{to $P_{\mathbf f}$}, or \emph{to $P_{\mathbf f}(\O_K)$}) is the ind-scheme over $k$
\[
   \F_{\mathbf f} := \F_{P_{\mathbf f}} = LG/L^+P_{\mathbf f}.
\]
\end{Definition}

In some cases, this mirrors the closely related constructions
of (partial) affine flag varieties in the setting of the theory of Kac-Moody Lie algebras (\cite{Kac}, \cite{Ku2}).
See  Remark \ref{rem:KacMoody} for more details on this relation.

In concrete examples involving classical groups, one can often identify the affine flag variety with a space of lattice chains; this fact is crucial to the embedding of the special fibers of local models mentioned at the beginning of this section.  Let $R$ be a $k$-algebra, and consider the $R((t))$-module $R((t))^n$ for some $n \geq 1$.  Recall that a \emph{lattice in $R((t))^n$} is an $R[[t]]$-submodule $L \subset R((t))^n$ which is free as an $R[[t]]$-module Zariski-locally on $\Spec R$, and such that the natural arrow $L \otimes_{R[[t]]} R((t)) \to R((t))^{n}$ is an isomorphism.  We leave it as an exercise to check that it is equivalent to say that $L$ is an $R[[t]]$-submodule of $R((t))^n$ such that $t^NR((t))^n \subset L \subset t^{-N}R((t))^n$ for $N$ sufficiently big, and such that $t^{-N}R((t))^n/L$ is projective as an $R$-module for one, hence any, such $N$.

All of the terminology for lattices from \S\ref{s:examples} admits an obvious analog in the present setting. A collection of lattices in $R((t))^n$ is a \emph{chain} if it is totally ordered under inclusion and all successive quotients are projective $R$-modules (necessarily of finite rank).  A lattice chain is \emph{periodic} if $t^{\pm 1} L$ is in the chain for every lattice $L$ in the chain.  In analogy with the definition of $\Lambda_i$ \eqref{disp:Lambda_i}, for $i = na+j$ with $0 \leq j < n$, we define the $\O_{K}$-lattice
\begin{equation}\label{disp:lambda_i}
   \lambda_i := \sum_{l=1}^j t^{-a-1}\O_K e_l + \sum_{l=j+1}^{n} t^{-a}\O_K e_l \subset K^n,
\end{equation}
where now $e_1,\dotsc,e_n$ denotes the standard ordered basis in $K^n$.  The $\lambda_i$'s form a periodic lattice chain $\dotsb \subset \lambda_{-1} \subset \lambda_0 \subset \lambda_1 \subset \dotsb$, which we again call the \emph{standard chain}.
More generally, let $I \subset \mZ$ be any nonempty subset which is closed under addition by $n$; or in other words, $I$ is the inverse image under the canonical projection $\mZ \to \mZ/n \mZ$ of a nonempty subset of $\mZ/n \mZ$.  Then we denote by $\lambda_I$ the periodic subchain of the standard chain consisting of all lattices of the form $\lambda_i$ for $i \in I$.  %In this way, of course, $\lambda_\mZ$ denotes the chain \eqref{disp:std_lambda_chain} itself.

\begin{eg}[$GL_n$ and $SL_n$]
\label{eg:GL_afv}
Let $G = GL_n$ over $K$.  Then the facets in the Bruhat--Tits building $\B\bigl(PGL_n(K)\bigr)$ are in bijective correspondence with the periodic $\O_K$-lattice chains in $K^n$, and the parahoric group schemes for $G$ can be described as automorphism schemes of these lattice chains.  More precisely, let us consider the chain $\lambda_I$ for some nonempty $I$ closed under addition by $n$; of course, every periodic $\O_K$-lattice chain in $K^n$ is $G(K)$-conjugate to $\lambda_I$ for some such $I$.
Let $P_I$ denote the automorphism scheme of $\lambda_I$ as a periodic lattice chain over $\O_K$.  Then for any $\O_K$-algebra $A$, the $A$-points of $P_I$ consist of all families
\begin{equation}\label{disp:(g_i)}
   (g_i) \in \prod_{i \in I} GL_A(\lambda_i \otimes_{\O_K} A)
\end{equation}
such that the isomorphism $\lambda_i \otimes A \xrightarrow[\sim]{t^a \otimes \id_A} \lambda_{i - na} \otimes A$ identifies $g_i$ with $g_{i - na}$ for all $i \in I$ and all $a \in \mZ$, and such that the diagram
\[
   \xymatrix{
      \lambda_i \otimes_{\O_K} A \ar[r] \ar[d]_-{g_i}^-\sim 
         &   \lambda_j \otimes_{\O_K} A \ar[d]^-{g_j}_-\sim\\
      \lambda_i \otimes_{\O_K} A \ar[r]   
         &   \lambda_j \otimes_{\O_K} A
      }
\]
commutes for all $i < j$ in $I$.  The scheme $P_I$ is a smooth $\O_K$-scheme with evident generic fiber $G = GL_n$ and whose $\O_K$-points identify with the full fixer in $G(K)$ of the facet $\mathbf f$ corresponding to $\lambda_I$.  Moreover, it is not hard to see that $P_I$ has connected special fiber.  Hence $P_I$ is the parahoric group scheme $P_{\mathbf f}$ attached to $\mathbf f$; see \cite[1.7, 4.6, 5.1.9, 5.2.6]{BTII}.  

For $R$ a $k$-algebra, let $\Lat_n(R)$ %{\bf WHY MAKE IT SO COMPLICATED?? ALSO THE NOTATION $\L$ IS ALREADY OVERLOADED WITH MEANING MR}
denote the category whose objects are the $R[[t]]$-lattices in $R((t))^n$ and whose morphisms are the natural inclusions of lattices.  Of course, any $R[[t]]$-lattice chain may be regarded as a full subcategory of $\Lat_n(R)$.  We define $\F_I$ to be the functor on $k$-algebras that assigns to each $R$ the set of all functors $L\colon\lambda_I \to \Lat_n(R)$ such that
\begin{altenumerate}
\renewcommand{\theenumi}{C}
\item\label{it:afv_chain_cond}
   (\emph{chain}) the image $L(\lambda_I)$ is a lattice chain in $R((t))^n$;
   %need to leave this in to ensure that succ quots are vector bundles
\renewcommand{\theenumi}{P}
\item\label{it:afv_periodicity_cond}
   (\emph{periodicity}) $L(t\lambda_i) = tL(\lambda_i)$ for all $i \in I$, so that the chain $L(\lambda_I)$ is periodic; and
\renewcommand{\theenumi}{R}
\item\label{it:afv_rank_cond}
   (\emph{rank}) $\dim_k \lambda_j/\lambda_i = \rank_R L(\lambda_j) / L(\lambda_i)$ for all $i  < j$.
\end{altenumerate}
In more down-to-earth terms, an $R$-point of $\F_I$ is just a periodic lattice chain in $R((t))^n$ indexed by the elements of $I$, such that the successive quotients have the same rank as the corresponding quotients in $\lambda_I$.

The loop group $LG$ acts on $\F_I$ via the natural representation of $G\bigl(R((t))\bigr)$ on $R((t))^n$, and it follows that the $LG$-equivariant map $LG \to \F_I$ specified by taking the tautological inclusion $\bigl(\lambda_I \inj \L(K^n)\bigr) \in \F_I(k)$ as basepoint defines an $LG$-equivariant morphism
\[
   \varphi\colon \F_{\mathbf f} \to \F_I.
\]
In fact $\varphi$ is an isomorphism: it is plainly a monomorphism, and it is an epimorphism because every periodic lattice chain in $R((t))^n$ admits a so-called ``normal form'' Zariski-locally on $\Spec R$, as is proved in \cite[Ch.~3 App.]{R-Z}.

Similar remarks apply to $SL_n$ over $K$.  Up to conjugacy, the parahoric group schemes for $SL_n$ are again given by certain automorphism schemes $P_I'$ of the chains $\lambda_I$ for nonempty $I$ closed under addition by $n$, where this time we consider families $(g_i)$ as in \eqref{disp:(g_i)} satisfying the same conditions as above and such that $\det (g_i) = 1$ for all $i$.  For given such $I$, let $\mathbf f'$ denote the facet 
associated to $P_I'$ and $\F'_{\mathbf f'} = LSL_n/L^+P_I'$ the associated affine flag variety.  The inclusion $SL_n \subset GL_n$ induces a monomorphism $\F'_{\mathbf f'} \inj \F_{\mathbf f}$, where $\mathbf f$ again denotes the associated facet for $GL_n$.

To describe $\F'_{\mathbf f'}$ as a space of lattice chains, we call  a functor $L\colon\lambda_I \to \Lat_n(R)$  \emph{special} if
\begin{altenumerate}
\renewcommand{\theenumi}{S}
\item\label{it:afv_special_cond}
   $\bigwedge_{R[[t]]}^n L(\lambda_i) = t^{-i} R[[t]]$ as a submodule of $\bigwedge_{R((t))}^n R((t))^n = R((t))$ for all $i \in I$. 
\end{altenumerate}
Then the isomorphism $\varphi$ above identifies $\F'_{\mathbf f'}$ with the subfunctor of $\F_I$ of special points $L$; this is easy to check directly, or see \cite[3.5]{Go1}.  As a consequence, note that a point $L \in \F_{\mathbf f}(R)$ is special as soon as $\bigwedge_{R[[t]]}^n L(\lambda_i) = t^{-i} R[[t]]$ for a single $i \in I$.

For applications to local models, it is convenient to consider not just the canonical embedding $\F'_{\mathbf f'} \inj \F_{\mathbf f}$, but the following variant,  involving a simple generalization of the notion of special.  For $r \in \mZ$, we say that a point $L \in \F_{\mathbf f}(R)$ is \emph{$r$-special} if $\bigwedge_{R[[t]]}^n L(\lambda_i) = t^{r-i} R[[t]]$ as a submodule of $R((t))$ for one, hence every, $i \in I$.  Let
\[
   I - r := \{\, i-r \mid i \in I \,\}.
\]
Then the functor $\lambda_i \mapsto \lambda_{i-r}$ is an $r$-special point in $\F_{\mathbf f}(k)$, and, taking it as basepoint, it determines an $LSL_n$-equivariant isomorphism from $\F'_{\mathbf f''}$ onto the subfunctor in $\F_{\mathbf f}$ of $r$-special points, where $\mathbf f''$ is the facet for $SL_n$ corresponding to $\lambda_{I-r}$.
\end{eg}

\begin{eg}[$GSp_{2g}$ and $Sp_{2g}$]\label{eg:GSp_afv}
Let $\phi$ denote the alternating $K$-bilinear form on $K^{2g}$ whose matrix with respect to the standard basis is
\[J_{2g}=
   \begin{pmatrix}
            &  H_g\\
      -H_g  &
   \end{pmatrix},
\]
as in \eqref{disp:sympl}.  We denote by $G$ the $K$-group $GSp_{2g} := GSp(\phi)$.

To describe the parahoric group schemes for $G$, let $I$ be a nonempty subset of \mZ closed under addition by $2g$ and multiplication by $-1$.  For any $i \in I$ and any $\O_K$-algebra $R$, the pairing $\phi$ induces a perfect $R$-bilinear pairing
\[
   (\lambda_i \otimes_{\O_K} R) \times (\lambda_{-i} \otimes_{\O_K} R)
      \xrightarrow{\phi_R} R,
\]
where we use a subscript $R$ to denote base change from $\O_K$ to $R$.  Let $P_I$ denote the $\O_K$-group scheme whose $R$-points consist of all families 
$$(g_i) \in \prod\nolimits_{i \in I} GL_R(\lambda_i \otimes_{\O_K} R)$$
 satisfying the same conditions as in the $GL_n$ case and such that, in addition, there exists $c \in R^\times$ such that
\begin{equation}\label{disp:c}
   \phi_R(g_i x, g_{-i}y) = c \cdot \phi_R(x,y)
\end{equation}
for all $i \in I$ and all $x \in \lambda_i \otimes_{\O_K} R$, $y \in \lambda_{-i} \otimes_{\O_K} R$.  Then, analogously to the $GL_n$ case, $P_I$ is a parahoric group scheme for $G$, and up to conjugacy all parahoric group schemes arise in this way.

Given nonempty $I$ closed under addition by $2g$ and multiplication by $-1$, let $\mathbf f$ denote the associated facet in the Bruhat--Tits building for $G_\ad$.  To describe the affine flag variety attached to $\mathbf f$, let $R$ be a $k$-algebra and recall the lattice category $\Lat_{2g}(R)$ from Example \ref{eg:GL_afv}.  For an $R[[t]]$-lattice $\Lambda$ in $R((t))^{2g}$, let $\wh\Lambda$ denote the $\phi$-dual of $\Lambda$, that is, the $R[[t]]$-module
\[
   \wh\Lambda := \bigl\{\, x \in R((t))^{2g} \bigm|
                     \phi_{R((t))}(\Lambda,x) \subset R[[t]] \,\bigr\}.
\]
%Then $\wh\Lambda$ is an $R[[t]]$-lattice, and $\phi \otimes R((t))$ restricts to a perfect pairing
%\[
%   \Lambda \times \wh\Lambda \to R[[t]].
%\]
We define $\F_I$ to be the functor on $k$-algebras that assigns to each $R$ the set of all functors $L\colon\lambda_I \to \Lat_{2g}(R)$ satisfying conditions \eqref{it:afv_chain_cond}, \eqref{it:afv_periodicity_cond}, and \eqref{it:afv_rank_cond} from Example \ref{eg:GL_afv} and such that, in addition,
\begin{altenumerate}
\renewcommand{\theenumi}{D}
\item\label{it:afv_dlty_cond}
   (duality) Zariski-locally on $\Spec R$, there exists $c \in R((t))^\times$ such that $\wh{L(\lambda_i)} = c \cdot L\bigl(\wh\lambda_i\bigr)$ for all $i \in I$.
\end{altenumerate}
Analogously to the $GL_n$ case, the loop group $LG$ acts naturally on $\F_I$, and taking the tautological inclusion $\bigl(\lambda_I \inj \L(K^n)\bigr) \in \F_I(k)$ as basepoint specifies an $LG$-equivariant isomorphism
\[
   \F_{\mathbf f} \isoarrow \F_I.
\]
This description of $\F_I$ is plainly equivalent to the lattice-theoretic description of the affine flag variety for $GSp_{2g}$ given in \cite[\S10]{P-R2} (except that the scalar denoted $a$ there should only be required to exist Zariski-locally on $\Spec R$).

For the group $Sp_{2g} := Sp(\phi)$ over $K$, up to conjugacy, the parahoric group schemes are again given by certain automorphism schemes $P'_I$ of the chains $\lambda_I$ for nonempty $I$ closed under addition by $2g$ and multiplication by $-1$, namely, we now take the closed subscheme of $P_I$ of points for which $c = 1$ in \eqref{disp:c}.  For given such $I$, let $\mathbf f'$ denote the associated facet for $Sp_{2g}$ and $\F'_{\mathbf f'}$ the associated affine flag variety.  The inclusion $Sp_{2g} \subset GSp_{2g}$ induces a monomorphism $\F'_{\mathbf f'} \inj \F_{\mathbf f}$, where $\mathbf f$ again denotes the associated facet for $GSp_{2g}$.  In this way $\F'_{\mathbf f'}$ identifies with the subfunctor of $\F_I$ of points $L$ such that $\wh{L(\lambda_i)} = L(\wh\lambda_i)$ for all $i \in I$, i.e.\ such that $c$ can be taken to equal $1$ in \eqref{it:afv_dlty_cond}.  This subfunctor can also be described as the subfunctor of all special $L$ such that the lattice chain $L(\lambda_I)$ is self-dual.

% {\bf MR: IS THIS TRUE? AFTER ALL, THE DETERMINANT IS ONLY A POWER OF THE MULTIPLIER \emph{BS:  By ``THIS'' are you referring to the sentence before (``The inclusion\dots'')\ or the sentence after (``Given a\dots'')?  Either way, I don't think I understand the thrust of the objection.}} Given a special point $L \in \F_I(R)$, it follows (or is also easy to check directly from the definitions) that the scalar $c$ appearing in \eqref{it:afv_dlty_cond} can be taken locally, hence globally, to equal $1$, so that the lattice chain $L(\lambda_I)$ is self-dual.

As in the linear case, it is convenient to consider other embeddings besides the standard one $\F'_{\mathbf f'} \inj \F_{\mathbf f}$.  This time we consider $r$-special $L$ only for $r \in g\mZ$.  Then the functor $\lambda_i \mapsto \lambda_{i-r}$ is an $r$-special point in $\F_{\mathbf f}(k)$ whose image lattice chain is self-dual, and just as in the linear case, it specifies an $LSp_{2g}$-equivariant isomorphism from $\F'_{\mathbf f''}$ onto the subfunctor in $\F_{\mathbf f}$ of $r$-special points $L$ such that $L(\lambda_I)$ is self-dual, where $\mathbf f''$ is the facet for $Sp_{2g}$ corresponding to $\lambda_{I-r}$.  Note that if $L$ is $r$-special, then the scalar $c$ appearing in \eqref{it:afv_dlty_cond} can be taken to equal $t^{-r/g}$.
\end{eg}

The affine flag varieties for other groups discussed in \S\ref{s:examples} can all be described similarly.  For example, see \cite[\S6.2]{Sm1} for $GO_{2g}$ (at least in the Iwahori case) and \cite[\S3.2]{P-R4} and \cite[\S4.2]{Sm3} for ramified $GU_n$.

Returning to the general discussion, we conclude this subsection with a couple of further structure results from \cite{P-R3}.  The first describes the connected components of loop groups and affine flag varieties in the case $k$ is algebraically closed.  Let $G$ be a connected reductive group over $K$ with $k = \ol k$.  Let $\pi_1(G)$ denote the fundamental group of $G$ in the sense of Borovoi \cite{Bor}; this can be described as the group $X_*(T)/Q^\vee$ of geometric cocharacters of $T$ modulo coroots, where $T$ is any maximal torus in $G$ defined over $K$. %This definition is incorrectly stated in PR3!
Let $K^\sep$ denote a separable closure of $K$.  Then the inertia group $I := \Gal(K^\sep/K)$ acts naturally on $\pi_1(G)$, and we may consider the coinvariants $\pi_1(G)_I$.  In \cite{Ko2}, Kottwitz constructs a functorial surjective homomorphism 
\begin{equation}\label{disp:Kottwitz_hom}
   G(K) \surj \pi_1(G)_I
\end{equation}
which turns out to parametrize the connected components of $LG$ and $\F_{\mathbf f}$ as follows.

\begin{thm}[{\cite[Th.\ 5.1]{P-R3}}]
Assume that $k$ is algebraically closed.  Then for any facet $\mathbf f$, the Kottwitz homomorphism induces isomorphisms
\begin{flalign*}
   \phantom{\qed} & &
   \pi_0(LG) \isoarrow \pi_0(\F_{\mathbf f}) \isoarrow \pi_1(G)_I. & &
   \qed
\end{flalign*}
\end{thm}
In the special case that $G$ is \emph{split} we have $\pi_1(G)_I = \pi_1(G)$.  Then the theorem may be regarded as an avatar of the familiar statement in topology, where $LG$ plays the role of the loop space of $G$.

The final result of the subsection (a generalization of a result of Faltings \cite{Fa4}) concerns the (ind-)scheme structure on $LG$ and $\F_{\mathbf f}$.  Recall that an ind-scheme is \emph{reduced} if it is expressible as a filtered colimit of reduced schemes.

\begin{thm}[{\cite[Th.\ 6.1]{P-R3}}]\label{reduced}
Assume that $k$ is perfect and let $G$ be a connected semi-simple $K$-group.  Suppose that $G$ splits over a tamely ramified extension of $K$ and that the order of the fundamental group $\pi_1(G_\der)$ of the derived group $G_\der$ is prime to the characteristic of $k$. Then the ind-schemes $LG$ and $\F_{\mathbf f}$, for any facet $\mathbf f$, are reduced.\qed
\end{thm}

We note that Theorem \ref{reduced} is only an existence theorem. In {\cite[Prop. 6.6]{P-R1}}, in the case $G=SL_n$, a candidate is proposed for writing the affine Grassmannian $LG/L^+G$ as an increasing union of reduced projective subschemes. This candidate indeed works if 
${\rm char }\ k=0$, or if $n\leq 2$, cf.\ loc.\ cit.  This is related to Remark \ref{rem:weyman} below. 

By contrast, if $G$ is reductive but not semi-simple, then $LG$ and $\F_{\mathbf f}$ are necessarily \emph{non-reduced} \cite[Prop.\ 6.5]{P-R3}.  We do not know if the assumption in the theorem that $G$ splits over a tamely ramified extension of $K$ is necessary.  On the other hand, the assumption on the order of $\pi_1(G_\der)$ appears to be:  for example, $\pi_1(PGL_2) = \mZ/2\mZ$ and $LPGL_2$ is non-reduced in characteristic $2$ \cite[Rem.\ 6.4]{P-R3}.

\subsection{Schubert varieties}\label{ss:Schubert_vars}
% \textbf{Note:  this subsection needs to be generalized to allow for Schubert cells and varieties attached to different parahoric subgroup schemes than the one we divide out by on the right.}

In this subsection we discuss Schubert cells and varieties in affine flag varieties; these are  the  analogs in the context of loop groups of the usual notions for ordinary flag varieties.  Let $G$ be connected reductive over $K$, and let $\mathbf f$ and $\mathbf f'$ be facets in $\B(G_\ad)$ contained in a common alcove.

\begin{Definition}
For $g \in G(K) = LG(k)$, the associated \emph{$\mathbf f'$-Schubert cell in $\F_{\mathbf f}$,} denoted $C_g$, is the reduced, locally closed subscheme of $\F_{\mathbf f}$ whose underlying topological space is the image 
of $L^+P_{\mathbf f'}$ in $\F_{\mathbf f}$ under the $L^+P_{\mathbf f'}$-equivariant map sending $1$ to the class of  $g$.%
\footnote{Note that $C_g$ is not a topological cell when $\mathbf f$ is not an alcove, i.e.,\ it is not isomorphic to an affine space.}
The associated \emph{$\mathbf f'$-Schubert variety in $\F_{\mathbf f}$,} denoted $S_g$, is the Zariski closure of $C_g$ in $\F_{\mathbf f}$ endowed with its reduced scheme structure.
\end{Definition}

We also refer to $\mathbf f'$-Schubert cells as $P_{\mathbf f'}$-Schubert cells or $P_{\mathbf f'}(\O_K)$-Schubert cells, and analogously for $\mathbf f'$-Schubert varieties.

The $\mathbf f'$-Schubert cell $C_g$ and the $\mathbf f'$-Schubert variety $S_g$ in $\F_{\mathbf f}$ only depend on the image of $g$ in the double coset space $P_{\mathbf f'}(\O_K) \bs G(K) / P_{\mathbf f}(\O_K)$.  For $k$ algebraically closed, we shall see later in  Proposition \ref{st:wt_W_double_cosets} that  this double coset space can be identified with the \emph{Iwahori-Weyl group} of $G$ when $\mathbf f$ and $\mathbf f'$ are a common alcove, and with a certain double coset space of the Iwahori-Weyl group in general.  Note that, since $P_{\mathbf f'}$ is smooth over $\O_K$ with connected special fiber, it follows from \cite[p.\ 264 Cor.\ 2]{Gr} that $L^+P_{\mathbf f'}$ is reduced and irreducible.  Hence each Schubert cell is irreducible.  Hence each Schubert variety is reduced and irreducible (reducedness being imposed by definition).  In general, the Schubert cells and Schubert varieties are subschemes of $\F_{\mathbf f}$ of finite type over $k$; moreover the Schubert varieties are proper over $k$.

The following theorem gives important information on the  structure of  Schubert varieties.

\begin{thm}[{\cite[Th.\ 8.4]{P-R3}}]\label{st:schub_vties}
Suppose that $G$ splits over a tamely ramified extension of $K$ and that the order of the fundamental group $\pi_1(G_\der)$ of the derived group $G_\der$ is prime to the characteristic of $k$.  
Then all $\mathbf f'$-Schubert varieties in $\F_{\mathbf f}$ are normal and have only rational singularities.  If $k$ has positive characteristic, then all $\mathbf f'$-Schubert varieties contained in a given $\mathbf f'$-Schubert variety are compatibly Frobenius split.\qed
\end{thm}

We refer to \cite{B-K} for the notion of a scheme $X$ in characteristic $p$ being Frobenius split, and for a family of closed subschemes of $X$ being compatibly Frobenius split. This property has important consequences for  the local structure: if $X$ is Frobenius split, then $X$ is reduced and weakly normal, cf.\ \cite[\S1.2]{B-K}. Also, if $\{X_1,\dotsc,X_n\}$ is a family of compatibly split closed subschemes of $X$, then their (reduced) union $X_1\cup \dotsb \cup X_n$ and their intersection $X_1\cap \dotsb \cap X_n$ are also compatibly split; in particular, $X_1\cap \dotsb \cap X_n$ is reduced. Frobenius splitness also has interesting global consequences, such as strong forms of the Kodaira vanishing property, cf.\ loc.\ cit.

\begin{comment}
To recall the notion of Frobenius split, let $X$ be a scheme of characteristic $p$ and let $\Fr_X \colon X \to X$ denote its absolute Frobenius morphism.  Then $X$ is \emph{Frobenius split} if the homomorphism $\O_X \xrightarrow{\Fr_X^\sharp} (\Fr_X)_*\O_X$ admits a retraction in the category of quasi-coherent $\O_X$-modules.  If $X$ is Frobenius split and $Y$ is a closed subscheme of $X$, then $Y$ is \emph{compatibly Frobenius split} if there exists a retraction of $\Fr_X^\sharp$ which induces a Frobenius splitting of $Y$, or in other words, if there exists a retraction $\rho\colon (\Fr_X)_*\O_X \to \O_X$ such that $\rho\bigl( (\Fr_X)_*\I_Y\bigr) \subset \I_Y$, where $\I_Y$ denotes the sheaf of ideals of $Y$ in $\O_X$.
\end{comment}

After introducing the Iwahori-Weyl group in \S\ref{s:combinatorics},  we will give in Propositions \ref{st:schub_dim} and \ref{dimandclos}, respectively, the dimension of Schubert varieties and their inclusion relations in terms of the combinatorics of the Iwahori-Weyl group.

\begin{Remark}
In \cite{P-R3} $\mathbf f'$-Schubert cells and varieties are only defined, and Theorem \ref{st:schub_vties} is only formulated and proved, in the case that $\mathbf f = \mathbf f'$.  But the method of proof involves a reduction to the case that $\mathbf f = \mathbf f'$ is an alcove, and this reduction step works just as well for any $\mathbf f'$-Schubert variety in $\F_{\mathbf f}$ in the sense defined here.  See \cite[Rem.\ 8.6, \S8.e.1]{P-R3}.  We shall see how $\mathbf f'$-Schubert varieties in $\F_{\mathbf f}$ with $\mathbf f' \neq \mathbf f$ arise naturally in the context of local models in the next subsection.
\end{Remark}

\begin{Remark}\label{rem:KacMoody}
In the case that the group $G$ is split semi-simple and simply connected, Theorem \ref{st:schub_vties} is due to Faltings \cite{Fa4}. 
Let us mention here that there are also corresponding results in the theory of affine flag varieties for Kac-Moody Lie algebras.
To explain this,  assume that $G$ is quasi-split, absolutely simple
and simply connected and splits over a tamely ramified extension. The  local Dynkin diagram of $G$ (as in the table of \S 1) is  also the Dynkin diagram 
of  a uniquely determined affine (or twisted affine) Kac-Moody Lie algebra ${\mathfrak g}={\mathfrak g}_{\rm KM}(G)$
(see \cite{Kac}). In the Kac-Moody setting, there is an affine flag variety ${\mathcal F}_{{\mathfrak g}}$ and Schubert varieties $S^{{\mathfrak g}}_w$
(see \cite{Ku2}, \cite{Ma}; here $w$ is an element of the affine Weyl group). Their definition is given by using an embedding into the infinite dimensional projective space associated to a highest
weight representation of the Kac-Moody algebra $\mathfrak g$; it is a priori different from our approach. The
normality of Schubert varieties $S^{\mathfrak g}_w$ in the Kac-Moody setting is a well-known 
cornerstone of the theory; it was shown by Kumar \cite{Ku1} in characteristic $0$ and 
by Mathieu \cite{Ma} and Littelmann \cite{Litt} in all characteristics.
It is not hard to  show that when ${\mathfrak g}={\mathfrak g}_{\rm KM}(G)$,
the Kac-Moody Schubert varieties $S^{{\mathfrak g}}_w$ are stratawise isomorphic to the  Schubert varieties 
$S_w$ in $\F_B=LG/L^+B$ that we consider here.  As a result, we can see a posteriori, 
as a  consequence of Theorem \ref{st:schub_vties} and the results 
of Mathieu and Littelmann, that the
Schubert varieties $S^{\mathfrak g}_{w}$ and $S_w$ are isomorphic. This also implies that the affine flag variety $\F_B$
 is isomorphic to the affine flag variety ${\mathcal F}_{{\mathfrak g}}$  for the corresponding Kac-Moody Lie algebra.
See \cite[9.h]{P-R3} for more details.
\end{Remark}

\subsection{Embedding the special fiber of local models}\label{ss:embedding}
We now come to the key application of affine flag varieties to the theory of local models, namely the embedding of the special fiber of the local model into an appropriate affine flag variety.  Since we do not know how to define the local model in general (cf.\ \S \ref{s:motivationalsection}), we can only describe the embedding in particular examples. Here we do so for $GL_n$ and $GSp_{2g}$. Note that for the Beilinson--Gaitsgory local model in \S
\ref{s:motivationalsection}.\ref{ss:BGm} such
an embedding is tautological. 

We resume the notation of \S\ref{s:examples}.  In particular, we take $k$ to be the residue field of $F$ and we recall the $\O_F$-lattices $\Lambda_i$ from \eqref{disp:Lambda_i}.  In analogy with our notation for the $\lambda_i$'s, for nonempty $I \subset \mZ$ closed under addition by $n$, we denote by $\Lambda_I$ the periodic lattice chain in $F^n$ consisting of the lattices $\Lambda_i$ for $i \in I$.

For any $\O_F$-scheme $X$, we write $\ol X$ for its special fiber $X \otimes_{\O_F} k$.

\begin{eg}[$GL_n$]\label{eg:GL_n_emb}
Let $I \subset \mZ$ be nonempty and closed under addition by $n$, let $\mu$ denote the cocharacter $\bigl( 1^{(r)}, 0^{(n-r)} \bigr)$ of the standard diagonal maximal torus in $GL_n$ and $\{\mu\}$ its geometric conjugacy class, and recall the local model $M^\loc_{GL_n,\{\mu\},\Lambda_I}$ over $\Spec \O_F$ from \S\ref{s:examples}.\ref{ss:GL_n}.  We embed the special fiber $\ol M^\loc_{GL_n,\{\mu\},\Lambda_I}$ in the affine flag variety $\F_I$ for $GL_n$ (see Example \ref{eg:GL_afv}) as follows.

Let $R$ be a $k$-algebra and $(\F_{\Lambda_i})_{i \in I}$ an $R$-point of $M^\loc_{GL_n,\{\mu\},\Lambda_I}$.  For $i \in I$, we identify
\[
   \Lambda_i \otimes_{\O_F} k \simeq \lambda_i \otimes_{\O_K} k
\]
by identifying the standard ordered bases on the two sides.  In this way we get an isomorphism of lattice chains $\Lambda_I \otimes_{\O_F} k \simeq \lambda_I \otimes_{\O_K} k$.  Via this isomorphism, we regard $\F_{\Lambda_i} \subset \Lambda_i \otimes_{\O_F} R$ as a submodule of $\lambda_i \otimes_{\O_K} R$, and we define $L_i$ to be the inverse image of $\F_{\Lambda_i}$ under the reduction-mod-$t$-map
\[
   \lambda_i \otimes_{\O_K} R[[t]] \surj \lambda_i \otimes_{\O_K} R.
\]
Then $L_i$ is an $R[[t]]$-lattice in $R((t))^n$.  Denoting the elements of $I$ by
\[
   \dotsb < i_{-1} < i_0 < i_1 < \dotsb,
\]
we get a diagram of lattices in $R((t))^n$
% {\bf IN THE MIDDLE ROW THE LEFT AND RIGHT INCLUSION SIGNS SHOULD LIGN UP}  
\begin{equation}\label{disp:L_diag}
   % begin by entering the entries that will appear in the bottom row of the diagram, since these are the widest
   \newsavebox{\entryA}
   \sbox{\entryA}{$t\lambda_{i_{-1}} \otimes_{\O_K} R[[t]]$}
   \newsavebox{\entryB}
   \sbox{\entryB}{$t\lambda_{i_0} \otimes_{\O_K} R[[t]]$}
   \newsavebox{\entryC}
   \sbox{\entryC}{$t\lambda_{i_1} \otimes_{\O_K} R[[t]]$}
   % next compute lengths of the above entries
   \newlength{\entryAlen}
   \settowidth{\entryAlen}{\usebox{\entryA}}
   \newlength{\entryBlen}
   \settowidth{\entryBlen}{\usebox{\entryB}}
   \newlength{\entryClen}
   \settowidth{\entryClen}{\usebox{\entryC}}
   % now put together the diagram
   \vcenter{
   \xymatrix@C-3ex@R-3ex{
      \dotsb \ar@{}[r]|-*{\subset} &
         \makebox[\the\entryAlen]{$\lambda_{i_{-1}} \otimes_{\O_K} R[[t]]$}
               \ar@{}[r]|-*{\subset} \ar@{}[d]|-*{\rotatebox{90}{$\subset$}} &
         \makebox[\the\entryBlen]{$\lambda_{i_0} \otimes_{\O_K} R[[t]]$} 
               \ar@{}[r]|-*{\subset} \ar@{}[d]|-*{\rotatebox{90}{$\subset$}} &
         \makebox[\the\entryClen]{$\lambda_{i_1} \otimes_{\O_K} R[[t]]$}
               \ar@{}[r]|-*{\subset} \ar@{}[d]|-*{\rotatebox{90}{$\subset$}} &
         \dotsb\\
      \dotsb \ar@{}[r]|-*{\subset} &
         \makebox[\the\entryAlen]{$L_{i_{-1}}$} \ar@{}[r]|-*{\subset}
               \ar@{}[d]|-*{\rotatebox{90}{$\subset$}} &
         \makebox[\the\entryBlen]{$L_{i_0}$} \ar@{}[r]|-*{\subset}
               \ar@{}[d]|-*{\rotatebox{90}{$\subset$}} &
         \makebox[\the\entryClen]{$L_{i_1}$} \ar@{}[r]|-*{\subset}
               \ar@{}[d]|-*{\rotatebox{90}{$\subset$}} &
         \dotsb\\
      \dotsb \ar@{}[r]|-*{\subset} &
         % finally, the entries created at the beginning now get entered
         \usebox\entryA \ar@{}[r]|-*{\subset} &
         \usebox\entryB \ar@{}[r]|-*{\subset} &
         \usebox\entryC \ar@{}[r]|-*{\subset} &
         \dotsb
   }
   }.
\end{equation}
It is easy to verify that the collection $(L_i)_{i\in I}$ specifies a point in $\F_I(R)$, and we define the morphism
\[
   \iota\colon \ol M^\loc_{GL_n,\{\mu\},\Lambda_I} \to \F_I
\]
by the rule $(\F_{\Lambda_i})_i \mapsto (L_i)_i$.  Plainly $\iota$ is a monomorphism, and it is therefore a closed immersion of ind-schemes since $\ol M^\loc_{GL_n,\{\mu\},\Lambda_I}$ is proper.

Moreover, it is easy to see that $(L_i)_i$ is \emph{$r$-special} as defined in  Example \ref{eg:GL_afv}.  Hence we get an embedding
\[
   \vcenter{
   \xymatrix@R-2ex@C-5ex{
      \ol M^\loc_{GL_n,\{\mu\},\Lambda_I} 
            \ar@{^{(}->}[rr]^-{\iota} \ar@{_{(}-->}[rd]
         & &  \F_I\\
      & \F'_{I-r} \ar@{^{(}->}[ru]
   }
   },
\]
% I need to improve the above diagram!  BS
where $\F'_{I-r}$ is the affine flag variety for $SL_n$ corresponding to the set $I - r$ and the unlabeled solid arrow is the embedding discussed in Example \ref{eg:GL_afv}.

The embeddings of $\ol M^\loc_{GL_n,\{\mu\},\Lambda_I}$ into $\F_I$ and into $\F'_{I-r}$ enjoy an important equivariance property which we now describe.  Let $\A$ denote the $\O_F$-group scheme of automorphisms of the lattice chain $\Lambda_I$, defined in the obviously analogous way to the $\O_K$-group scheme $P_I$ in Example \ref{eg:GL_afv}.  Then $\A$ acts naturally on $M^\loc_{GL_n,\{\mu\},\Lambda_I}$.  Now consider the positive loop group $L^+P_I$ over $\Spec k$.  The tautological action of $P_I$ on $\lambda_I$ furnishes a natural action of $L^+P_I$ on the chain $\lambda_I \otimes_{\O_K} k$.  The isomorphism $\lambda_I \otimes_{\O_K} k \simeq \Lambda_I \otimes_{\O_F} k$ then yields a homomorphism $L^+P_I \to \ol \A$.  Hence $L^+P_I$ acts on $\ol M^\loc_{GL_n,\{\mu\},\Lambda_I}$.  It is now easy to see that the embedding $\ol M^\loc_{GL_n,\{\mu\},\Lambda_I} \inj \F_I$ is $L^+P_I$-equivariant with respect to the natural $L^+P_I$-action on $\F_I$.  As a consequence, we see that $\ol M^\loc_{GL_n,\{\mu\},\Lambda_I}$ decomposes into a union of $P_I$-Schubert cells inside $\F_I$.

Entirely similar remarks apply to the embedding $\ol M^\loc_{GL_n,\{\mu\},\Lambda_I} \inj \F'_{I-r}$: the positive loop group $L^+P'_I$ acts naturally on $\ol M^\loc_{GL_n,\{\mu\},\Lambda_I}$ in an analogous way, the embedding into $\F'_{I-r}$ is then $L^+P'_I$-equivariant, and we conclude that $\ol M^\loc_{GL_n,\{\mu\},\Lambda_I}$ decomposes into a union of $P'_I$-Schubert cells inside $\F'_{I-r}$.
\end{eg}

\begin{eg}[$GSp_{2g}$]
Now let $n = 2g$, $\mu = \bigl(1^{(g)},0^{(g)}\bigr)$, and $\{\mu\}$ its geometric conjugacy class in $GSp_{2g}$, and suppose that $I$ is, in addition, closed under multiplication by $-1$.  Then we may consider the local model $M^\loc_{GSp_{2g},\{\mu\},\Lambda_I}$ for $GSp_{2g}$ as in \S\ref{s:examples}.\ref{ss:GSp_2g}.  The embedding of the special fiber $\ol M^\loc_{GSp_{2g},\{\mu\},\Lambda_I}$ into the affine flag variety $\F_I$ for $GSp_{2g}$ is completely analogous to the situation just considered for $GL_n$.

More precisely, under the embedding of $\ol M^\loc_{GL_{2g},\{\mu\},\Lambda_I}$ into the affine flag variety for $GL_{2g}$ from the previous example, it is easy to see that the closed subscheme
\[
   \ol M^\loc_{GSp_{2g},\{\mu\},\Lambda_I} \subset \ol M^\loc_{GL_{2g},\{\mu\},\Lambda_I}
\]
is carried into the locus of points satisfying condition \eqref{it:afv_dlty_cond} in Example \ref{eg:GSp_afv} (where for any $R$-valued point, the scalar $c$ in \eqref{it:afv_dlty_cond} can be taken to equal $t^{-1}$ globally).  Hence we get the desired embedding of $\ol M^\loc_{GSp_{2g},\{\mu\},\Lambda_I}$ into the affine flag variety for $GSp_{2g}$.

Continuing the analogy, the parahoric group scheme for $GSp_{2g}$ denoted $P_I$ in Example \ref{eg:GSp_afv} again acts naturally on $\ol M^\loc_{GL_{2g},\{\mu\},\Lambda_I}$, and we again conclude that $\ol M^\loc_{GSp_{2g},\{\mu\},\Lambda_I}$ decomposes into a union of $P_I$-Schubert cells inside $\F_I$.  Moreover, since the image of $\ol M^\loc_{GL_{2g},\{\mu\},\Lambda_I}$ consists of $g$-special points in $\F_I$, there is an induced embedding $\ol M^\loc_{GL_{2g},\{\mu\},\Lambda_I} \inj \F'_{I-g}$, where $\F'_{I-g}$ denotes the affine flag variety for $Sp_{2g}$ corresponding to the set $I-g$; and $\ol M^\loc_{GSp_{2g},\{\mu\},\Lambda_I}$ decomposes into a union of $P'_I$-Schubert cells inside $\F'_{I-g}$.
\end{eg}

The embeddings of the special fibers of the local models for other groups discussed in \S\ref{s:examples} can all be described similarly.  See \cite[\S7.1]{Sm1} for $GO_{2g}$ (at least in the Iwahori case), \cite[\S4]{P-R2} for totally ramified $\Res_{F/F_0}GL_n$, \cite[\S11]{P-R2} for totally ramified $\Res_{F/F_0} GSp_{2g}$, and \cite[\S3.3]{P-R4} and \cite[\S4.4]{Sm3} for ramified $GU_n$.  %In all cases, once the special fiber of the local model is embedded into an affine flag variety, its image decomposes into a union of Schubert cells, and it becomes a fundamental problem to determine which Schubert cells occur.  We shall return to this point in \S\ref{s:combinatorics}.
In all cases, the image of the special fiber of the local model decomposes into a union of Schubert cells inside the affine flag variety.  In fact, since the local model is proper, the image is a union of Schubert varieties. It then becomes an interesting  problem to determine which Schubert varieties occur in the union.  This is a problem of an essentially combinatorial nature to which we turn in \S\ref{s:combinatorics}.

\section{Combinatorics}\label{s:combinatorics}

In all known examples --- and as we saw explicitly for $GL_n$ and $GSp_{2g}$ in \S\ref{s:afv}.\ref{ss:embedding} --- the special fiber of the local model admits an embedding into an affine flag variety, with regard to which it decomposes into a union of Schubert varieties.  It is then a basic problem to determine which Schubert varieties occur in the union.  
Arising from this are a number of considerations of an essentially combinatorial nature to which we turn in this section.  Much of our discussion is borrowed from \cite[\S\S2--3]{R} and \cite[\S\S2.1--2.2]{P-R4}.
% This leads to a number of considerations which are essentially combinatorial in nature%
% , which we now take up in this section.
% , to which we shall devote this section.

We shall work over a complete, discretely valued field $L$, which we suppose in addition is \emph{strictly Henselian}.  For applications to local models, we are especially interested in the setting $L = \ol k((t))$, where $\overline k$ is an algebraic closure of the residue field $k$ as denoted in \S\ref{s:examples}; this setting implicitly corresponds to working with the \emph{geometric} special fiber of the local model.  We write $\O_L$ for the ring of integers in $L$.

Given a connected reductive group $G$ over $L$, we denote by $\kappa_G$ its Kottwitz homomorphism, as encountered earlier in \eqref{disp:Kottwitz_hom}; recall that this is a functorial surjective map $G(L) \to \pi_1(G)_I$, where $I := \Gal(L^\sep/L)$.  To be clear about signs, we take $\kappa_G$ to be exactly the map defined by Kottwitz in \cite[\S7]{Ko2} (which makes sense over any complete, discretely valued, strictly Henselian field), without the intervention of signs.  This is opposite to the sign convention taken in Richarz's article \cite{Ri}, to which we shall refer in several places.  The only practical effect of this difference is that we shall be led to make use of dominant coweights wherever Richarz makes use of antidominant coweights.

\subsection{Iwahori-Weyl group}\label{ss:I-W_gp}
Let $G$ be a connected reductive group over $L$, let $S$ be a maximal split torus in $G$, let $G_\ad$ denote the adjoint group of $G$, and let $S_\ad$ denote the image of $S$ in $G_\ad$.  Then $S_\ad$ is a maximal split torus in $G_\ad$, and we let $\A := X_*(S_\ad) \otimes_\mZ \mR$ denote the apartment in the building of $G_\ad$ attached to $S_\ad$. Let $T$ be the centralizer of $S$ in $G$.  Then $T$ is a maximal torus, since by Steinberg's theorem $G$ is quasi-split.  Let $N$ be the normalizer of $T$ in $G$, and let $T(L)_1$ denote the kernel of the Kottwitz homomorphism $\kappa_T \colon T(L) \surj \pi_1(T)_I = X_*(T)_I$ for $T$.

\begin{Definition}\label{def:I-W_gp}
The \emph{Iwahori-Weyl group of $G$ associated to $S$} is the group
\[
   \wt W_{G,S} := \wt W_G := \wt W := N(L)/T(L)_1.
\]
\end{Definition}

Observe that the evident exact sequence
\[
   0 \to T(L)/T(L)_1 \to \wt W \to N(L)/T(L) \to 1
\]
exhibits $\wt W$ as an extension of the relative Weyl group
\[
   W_0 := N(L)/T(L)
\]
by (via the Kottwitz homomorphism)
\[
   X_*(T)_I \cong T(L)/T(L)_1.
\]
In fact this sequence splits, by splittings which depend on choices.  More precisely, for any parahoric subgroup $K \subset G(L)$ attached to a facet contained in the apartment for $S$, let%
\footnote{Here we follow the convention of \cite{H-R,P-R3,P-R4} by using a superscript $K$ in \eqref{disp:W^K}.  Some authors would instead denote the group \eqref{disp:W^K} by $W_K$, and then use $W^K$ to denote the set of elements $w$ in the affine Weyl group such that $w$ has minimal length in the coset $wW_K$.}
\begin{equation}\label{disp:W^K}
   W^K := \bigl(N(L) \cap K\bigr)\big/ T(L)_1.
\end{equation}

\begin{prop}[{\cite[Prop.~13]{H-R}}]\label{st:semidirect_prod}
Let $K$ be the maximal parahoric subgroup of $G(L)$ attached to a special vertex in $\A$. Then the subgroup $W^K$ of $\wt W$ projects isomorphically to the factor group $W_0$, so that $\wt W$ admits a semidirect product decomposition
\begin{flalign*}
   \phantom{\qed} & &
   \wt W = X_*(T)_I \rtimes W^K \cong X_*(T)_I \rtimes W_0. & &
   \qed
\end{flalign*}
\end{prop}

We typically write $t_\mu$ when we wish to regard an element $\mu \in X_*(T)_I$ as an element in $\wt W$, and we refer to $X_*(T)_I$ as the \emph{translation subgroup} of $\wt W$.

\begin{Remark}\label{rk:wtW_split}
Let $\R = (X^*,X_*,\Phi,\Phi^\vee)$ be a root datum.  We define the \emph{extended affine Weyl group $\wt W(\R)$ of $\R$} to be the semidirect product $X_* \rtimes W(\R)$, where $W(\R)$ denotes the (finite) Weyl group of $\R$.
   
In the case that $G$ is split, $\wt W$ canonically identifies with the extended affine Weyl group $\wt W(\R)$ of the root datum $\R := \bigl(X^*(S), X_*(S), \Phi, \Phi^\vee \bigr)$ of $G$.  Indeed, in this case $S=T$, the action of $I$ on $X_*(T)$ is trivial, and $T(L)_1=S(\O_L)$. Taking $G(\O_L)$ as the special maximal parahoric subgroup in Proposition \ref{st:semidirect_prod}, we have $\wt W= X_*(S) \rtimes W_0$, where $W_0=W$ is the absolute Weyl group, which identifies with $W(\R)$. Then $\wt W$ contains the affine Weyl group $W_\aff(\R) := Q^\vee\rtimes W_0$ as a normal subgroup with abelian factor group $\pi_1(G) = X_*(S)/Q^\vee$. Here $Q^\vee\subset X_*(S)$ denotes the subgroup generated by the coroots $\Phi^\vee$. 

Even if $G$ is nonsplit, $\wt W$ can be identified with a  generalized extended affine Weyl group of a reduced root system, as is explained in Remark \ref{rk:generalized_awg} below.
\end{Remark}

\begin{eg}[$GL_n$]\label{eg:wtW_GL_n}
Let $G = GL_n$ in Remark \ref{rk:wtW_split} and take for $S = T $ the standard split maximal torus of diagonal matrices in $G$.  Then
\[
   \wt W \cong \mZ^n \rtimes S_n,
\]
where $\mZ^n \isoarrow X_*(S)$ by sending the $i$th standard basis element to the cocharacter $x \mapsto \diag\bigl(1^{(i-1)},x,1^{(n-i)}\bigr)$, and where the symmetric group $S_n$ of permutation matrices maps isomorphically to the Weyl group.
\end{eg}

\begin{eg}[$GSp_{2g}$]\label{eg:wtW_GSp}
Let $G = GSp_{2g}$ in Remark \ref{rk:wtW_split} and take $S = T$ to be the standard split maximal torus of diagonal matrices in $G$.  Let
\[
   S_{2g}^*  := \bigl\{\, \sigma\in S_{2g} \bigm| \sigma(i^*) = \sigma(i)^* 
               \text{ for all $i$} \,\bigr\},
\]
where $i^* := 2g+1-i$ for any $i \in \{1,\dotsc,2g\}$.  Then $S_{2g}^*$ identifies with the subgroup of permutation matrices in $G$ and maps isomorphically to the Weyl group.  We obtain
\[
   \wt W \cong X_* \rtimes S_{2g}^*,
\]
where, in terms of the natural embedding of $S$ into the maximal torus for $GL_{2g}$, and in terms of the identification of the previous example,  we have
\[
   X_* := \bigl\{\, (x_1,\dotsc,x_{2g}) \in \mZ^{2g} \bigm|
               x_1 + x_{2g} = x_2 + x_{2g -1} = \dotsb = x_g + x_{g+1} \,\bigr\}
   \isoarrow X_*(S).
\]
% and where the subgroup $S_{2g}^*$ of permutation matrices in $G$ maps isomorphically to the Weyl group.  Here $S_{2g}^*$ is the subgroup of $S_{2g}$
% \[
%    S_{2g}^* := \bigl\{\, \sigma\in S_{2g} \bigm| \sigma(i^*) = \sigma(i)^* 
%                \text{ for all $i$} \,\bigr\},
% \]
% where $i^* := 2g+1-i$ for any $i \in \{1,\dotsc,2g\}$.
\end{eg}

\begin{Remark}[$GO_{2g}$]\label{rk:wtW_GO}
Although the orthogonal similitude group $GO_{2g}$ is not connected, we can give an ad hoc definition of its Iwahori-Weyl group by following the recipe in Definition \ref{def:I-W_gp} in the most literal way, where we take the normalizer in the full group.  We find
\[
   \wt W_{GO_{2g}} \cong X_* \rtimes S_{2g}^*,
\]
just as in the previous example.  The Iwahori-Weyl group $\wt W_{GO_{2g}^\circ}$ of the identity component $GO_{2g}^\circ$ is naturally a subgroup of $\wt W_{GO_{2g}}$ of index $2$. Explicitly,
\[
   \wt W_{GO_{2g}^\circ} \cong X_* \rtimes S_{2g}^\circ,
\]
where
\[
   S_{2g}^\circ := \bigl\{ \sigma \in S_{2g}^* \bigm|
                      \text{$\sigma$ is even in $S_{2g}$} \,\bigr\}.
\]

It turns out that, just as we shall see for connected groups, in the function field case $\wt W_{GO_{2g}}$ continues to parametrize the Schubert cells in the Iwahori affine flag variety for $GO_{2g}$.  See \cite{Sm1}.
\end{Remark}

The following group-theoretic result provides the key link between the Iwahori-Weyl group and local models.

\begin{prop}[{\cite[Prop.~8]{H-R}}]\label{st:wt_W_double_cosets}
Let $B$ be an Iwahori subgroup of $G(L)$ attached to an alcove in $\A$.  Then the inclusion $N(L) \subset G(L)$ induces a bijection
\[
   \wt W \isoarrow B\bs G(L) / B.
\]
More generally, let $K$ and $K'$ be parahoric subgroups of $G(L)$
attached to facets in $\A$.
Then the inclusion $N(L) \subset G(L)$ induces a bijection
\begin{flalign*}
   \phantom{\qed} & &
   W^{K'} \bs \wt W / W^{K} \isoarrow K'\bs G(L) / K. & &
   \qed
\end{flalign*}
\end{prop}

\begin{Remark}\label{rk:identWK}
 Assume in Proposition \ref{st:wt_W_double_cosets} that $K$ is  a parahoric subgroup attached to a \emph{special} vertex.    Then $\wt W \cong X_*(T)_I \rtimes W^K$ and, since $W^K \cong W_0$,
\[
   W^K \bs \wt W / W^K \cong X_*(T)_I/W_0.
\]
This last set may in turn be identified with the set of dominant elements in $X_*(T)_I$
 % relative to the choice of a positive chamber 
 (any element in $X_*(T)_I$  is conjugate under $W_0$ to a unique dominant element, cf.\ {\cite[Remark before Cor.\ 1.8]{Ri}}; recall that we use dominant elements where Richarz uses antidominant ones). The notion of dominant elements in $X_*(T)_I$ arises after identifying $\wt W$ with a generalized extended affine Weyl group of a certain root system $\Sigma$ in the sense of Remark \ref{rk:generalized_awg} below, and then choosing a basis for $\Sigma$; comp.\ Remark \ref{rk:bruhatKK}. 

If we assume in addition  that $G$ is split, then $K$ is hyperspecial, $X_*(T)_I=X_*(S)$, the notion of a dominant coweight in $X_*(S)$ is more standard. 

\end{Remark}

\begin{Remark}\label{rk:W^K_Weyl_gp}
Let $K$ be a parahoric subgroup of $G(L)$ and $P$ the corresponding parahoric group scheme over $\Spec \O_L$. Then $W^K$ can be identified with
the Weyl group of the special fiber $\ol{P}$ of $P$ {\cite[Prop.~12]{H-R}}.
\end{Remark}

\subsection{Bruhat order}\label{ss:bruhat_order}
The force of Proposition \ref{st:wt_W_double_cosets} in the context of local models is that, when $L$ is a field of Laurent series and $K$ and $K'$ are as in the proposition, the set of double classes $W^{K'} \bs \wt W / W^{K}$ (which is $\wt W$ itself in the Iwahori case) parametrizes the $K'$-Schubert cells in the affine flag variety for $K$.  To better exploit this fact, in this subsection we shall introduce the \emph{Bruhat order} on $\wt W$.
% , for which we shall need to appeal to some more advanced aspects of Bruhat-Tits theory than we have yet encountered.  Readers who are less well-versed in Bruhat-Tits theory should be able to treat many of the results from Bruhat-Tits theory in this subsection as black boxes, with little impairment to overall understanding.
This will lead us to some much heavier usage of Bruhat-Tits theory than we have yet encountered, but non-experts should be able to safely treat many of the external results we appeal to as black boxes, with little impairment to overall understanding. 
% for which we shall need to intensify our appeals to Bruhat-Tits theory.  Non-experts
% To better exploit this fact, we shall now introduce the \emph{Bruhat order} on $\wt W$, which will require some further development of the structure theory of $\wt W$.  we shall Bruhat-Tits theory in this subsection will be much more intense than in previous 
We continue with the notation of the previous subsection.

Consider the apartment $\A = X_*(S_\ad) \otimes_\mZ \mR \cong X_*(T_\ad)_I \otimes_\mZ \mR$.  We obtain from the Kottwitz homomorphism a map%
\footnote{Note that the displayed composite differs from the analogous map defined in Tits's article \cite[\S1.2]{T} by a sign of $-1$.  However this discrepancy will make no difference in any of our subsequent appeals to \cite{T}.  It matters only in that, as we have mentioned before, we systematically work with dominant elements where Richarz \cite{Ri} uses antidominant elements; see especially Proposition \ref{dimspec}.}
\[
   T(L)/T(L)_1 \xrightarrow[\sim]{\kappa_T} X_*(T)_I
        \to X_*(T_\ad)_I
        \to \A.
\]
Let $\Aff(\A)$ denote the group of affine transformations on $\A$.  The relative Weyl group $W_0$ acts naturally on $\A$ by linear transformations, and regarding $\A \rtimes W_0$ as a subgroup of $\Aff(\A)$, the displayed map extends to a map of exact sequences
\begin{equation}\label{disp:nu_diag}
   \vcenter{
   \xymatrix{
      1 \ar[r]
         & T(L)/T(L)_1 \ar[r] \ar[d]
         & \wt W \ar[r] \ar@{-->}[d]^-{\nu}
         & W_0 \ar[r] \ar@{=}[d]
         & 1 \\
      1 \ar[r]
         & \A \ar[r]
         & \A \rtimes W_0 \ar[r]
         & W_0 \ar[r]
         & 1,
   }
   }
\end{equation}
in which $\nu$ is unique up to conjugation by a unique translation element in $\Aff(\A)$; see \cite[\S1.2]{T} or \cite[Prop.\ 1.6 and 1.8]{Lan}.

Having chosen $\nu$, the set of \emph{affine roots} $\Phi_\aff$, which consists of certain affine functions on $\A$, is defined in Tits's article \cite[\S1.6]{T}.  There then exists a unique reduced root system $\Sigma$\label{ref:Sigma} on $\A$ with the properties that
\begin{altitemize}
\item
   every root $\alpha \in \Sigma$ is proportional to the linear part of some affine root; and
\item
   for any special vertex $v \in \A$ \cite[\S1.9]{T}, translation by $-v$
\[
   \A \xrightarrow{t_{-v}} \A
\]
carries the vanishing hyperplanes of the affine roots to precisely the vanishing hyperplanes of the functions on $\A$
\begin{equation}\label{disp:aff_root_hyperplanes}
   u \mapsto \alpha(u) + d
   \quad\text{for}\quad
   \alpha \in \Sigma,\ d \in \mZ;
\end{equation}
\end{altitemize}
see \cite[1.3.8]{BTI}, \cite[\S1.7]{T}, and \cite[VI \S2.5 Prop.\ 8]{B}.  For $G$ absolutely simple and simply connected, see Remark \ref{rk:generalized_awg} below for a description of $\Phi_\aff$ and $\Sigma$.

The \emph{affine Weyl group $W_\aff$} is the group of affine transformations on $\A$ generated by the reflections through the affine root hyperplanes.  Analogously, the \emph{affine Weyl group $W_\aff(\Sigma)$ of $\Sigma$} \cite[VI \S2.1 Def.\ 1]{B} is the group of affine transformations on $\A$ generated by the reflections through the vanishing hyperplanes of the functions \eqref{disp:aff_root_hyperplanes}.  Thus for any special vertex $v$, we have
\[
   W_\aff = t_vW_\aff(\Sigma)t_v^{-1}.
\]
Hence $W_\aff$ is a Coxeter group generated (as a Coxeter group) by the reflections through the walls of any fixed alcove.  The affine Weyl group for $\Sigma$ admits the semidirect product decomposition $W_\aff(\Sigma) = Q^\vee(\Sigma) \rtimes W(\Sigma)$, where $Q^\vee(\Sigma)$ is the coroot lattice for $\Sigma$ and $W(\Sigma)$ is the Weyl group of $\Sigma$.

To apply the preceding discussion to the Iwahori-Weyl group, consider the diagram
\[
   \xymatrix@C=3ex{
      & \wt W \ar[d]^-{\nu}\\
      W_\aff \ar@{}[r]|-*{\subset}  &  \Aff(\A).
   }
\]
We shall show that $W_\aff$ lifts canonically to $\wt W$.  Indeed, let
\[
   G(L)_1 := \ker \kappa_G
\]
and
\begin{equation}\label{disp:N(L)_1}
   N(L)_1 := N(L) \cap G(L)_1.
\end{equation}
Let $B \subset G(L)$ be the Iwahori subgroup attached to an alcove in $\A$, and let $\Pi$ be the set of reflections through the walls of this alcove. Then, taking into account that $N(L) \cap B = T(L)_1$ \cite[Lem.\ 6]{H-R} and that $G(L)_1$ is the subgroup of $G(L)$ generated by the parahoric subgroups of $G(L)$ \cite[Lem.\ 17]{H-R}, the quadruple
\begin{equation}\label{disp:G(L)_1_double_Tits_system}
   \bigl(G(L)_1, B, N(L)_1, \Pi\bigr)
\end{equation}
is a double Tits system \cite[5.1.1]{BTI} whose Weyl group $N(L)_1/T(L)_1 \subset \wt W$ identifies via $\nu$ with $W_\aff$ by \cite[5.2.12]{BTII}.

The affine Weyl group can also be realized as a subgroup of $\wt W$ via the
simply connected cover $G_\sc$ of the derived group $G_\der$ of $G$. To explain
this, let $S_\sc$, $T_\sc$, and $N_\sc$ denote the respective inverse images of
$S \cap G_\der$, $T\cap G_\der$, and $N \cap G_\der$ in $G_\sc$. Then $S_\sc$ is
a maximal split torus in $G_\sc$ with centralizer $T_\sc$ and normalizer
$N_\sc$. Let $\wt W_\sc := N_\sc(L)/T_\sc(L)_1$ denote the Iwahori-Weyl group of
$G_\sc$, let $B_\sc \subset G_\sc(L)$ be the Iwahori subgroup attached to an
alcove in $\A$, and let $\Pi$ again be the set of reflections through the walls
of this alcove. Then by \cite[Prop.\ 5.2.10]{BTII}, $\bigl(G_\sc(L), B_\sc,
N_\sc(L), \Pi\bigr)$ is a double Tits system whose Weyl group $\wt W_\sc$ (we
again use \cite[Lem.\ 6]{H-R}) identifies with $W_\aff$ via the composite
\[\
   \wt W_\sc \to \wt W \xra\nu \A \rtimes W_0.
\]
In this way $T_\sc(L)/T_\sc(L)_1 \cong X_*(T_\sc)_I$ \label{ref:X_*(T_sc)_I}
 identifies with the translation elements in $W_\aff$, or in other words, with the coroot lattice in $W_\aff(\Sigma)$.  Moreover, for any parahoric subgroup $K_\sc \subset G_\sc(L)$ attached to a special vertex $v$, the composite
\[
   \wt W_\sc \to \wt W \xra\nu \A \rtimes W_0 \xra{\substack{\text{conjugation}\\ \text{by }t_{-v}}} \A \rtimes W_0
\]
carries $W^{K_\sc} \subset \wt W_\sc$ isomorphically to the Weyl group of $\Sigma$, which identifies with $W_0$.  In other words, the composite isomorphism
\[
   \wt W_\sc \xra[\sim]{\nu|_{\wt W_\sc}}
   W_\aff \xra[\sim]{\substack{\text{conjugation}\\ \text{by }t_{-v}}}
   W_\aff(\Sigma)
\]
is compatible with the semidirect product decompositions $\wt W_\sc \cong X_*(T_\sc)_I \rtimes W^{K_{\sc}}$ and $W_\aff(\Sigma) = Q^\vee(\Sigma) \rtimes W(\Sigma)$.

\begin{Remark}\label{rk:generalized_awg}
Just as $\wt W_\sc$ can in this way be identified with the affine Weyl group for the root system $\Sigma$, so can $\wt W$ be identified
with a {\it generalized extended affine Weyl group} for $\Sigma$ via push-out by the canonical injection $X_*(T_{\rm sc})_I \inj X_*(T)_I$.  Here ``generalized'' means that the abelian group $X_*(T)_I$ may have torsion.

For any absolutely simple, simply connected group $G$ over a discretely valued field with algebraically closed residue field, the affine root system $\Phi_\aff$ and root system $\Sigma$ admit the following descriptions (up to a choice of normalization of the valuation, and of a special vertex as origin), which are given by Prasad--Raghunathan \cite[\S2.8]{PrRa}; we thank \mbox{J.-K.}\ Yu and X. Zhu for pointing this out to us.  Let $\Phi$ denote the relative roots of $S$ in $G$.

If $G$ is split over $L$, then $\Phi$ is necessarily reduced,
\[
   \Phi_\aff = \{\, a + \mZ \mid a\in \Phi\,\},
   \quad\text{and}\quad 
   \Sigma = \Phi.
\]
If $G$ is nonsplit over $L$ and $\Phi$ is reduced, then
\[
   \Phi_\aff = \biggl\{\, a + \frac{(a,a)} 2 \mZ \biggm| a\in \Phi \,\biggr\}
   \quad\text{and}\quad
   \Sigma = \Phi^\vee \cong \biggl\{\, \frac{2a}{(a,a)} \biggm| a \in \Phi\,\biggr\},
\] where \sform is a nondegenerate $W_0$-invariant inner product.  If $G$ is nonsplit and $\Phi$ is nonreduced, then $G$ is an outer form of type $A_{2n}$,
\[
   \Phi_\aff = 
      \biggl\{\, a + \mZ \biggm| a\in\Phi,\ \frac a 2 \notin \Phi\,\biggr\}
      \cup
      \biggl\{\, a + 1 + 2\mZ \biggm| a\in \Phi,\ \frac a 2 \in \Phi \,\biggr\},
\]
and $\Sigma$ is the subset of $\Phi$ of roots $a$ for which $2a \notin \Phi$.
\end{Remark}

We return to the main discussion.  Since the target of the Kottwitz homomorphism is abelian, it is immediate from the definition \eqref{disp:N(L)_1} of $N_1(L)$ that it is a normal subgroup of $N(L)$.  Hence $W_\aff \simeq N(L)_1/T(L)_1$ is a normal subgroup of $\wt W$, and we get an exact sequence
\[
   1 \to W_\aff \to \wt W\to X_*(T)_I/X_*(T_\sc)_I\to 1.
\]
This sequence splits canonically after choosing a base alcove: since $W_\aff$ acts simply transitively on the alcoves
in $\A$ \cite[1.7]{T},
$\wt W$ is the semidirect product of $W_\aff$ with the normalizer $\Omega \subset \wt W$ of the base alcove,
% i.e.\ the subgroup of $\wt W$ which preserves the alcove as a set, 
\begin{equation}\label{disp:semidirect_prod_2}
   \wt W=W_\aff \rtimes \Omega,
\end{equation}
with $\Omega \isoarrow \wt W/W_\aff \cong X_*(T)_I/X_*(T_\sc)_I \cong \pi_1(G)_I$.

The semidirect product decomposition \eqref{disp:semidirect_prod_2} for $\wt W$ has the important consequence of endowing $\wt W$ with a \emph{Bruhat order} and \emph{length function}.  Again let $\Pi$ denote the subset of $W_\aff$ of reflections across the walls of the base alcove.  As we have already recalled, $\Pi$ is a set of Coxeter generators for $W_\aff$.  We then get a Bruhat order $\leq$ and a length function $\ell$ on $W_\aff$ as for any Coxeter group:  for $s$, $s' \in W_\aff$, $\ell(s)$ is the smallest nonnegative integer $r$ such that $s$ is expressible as a product $s_1s_2\dotsm s_r$ with $s_1$, $s_2,\dotsc$, $s_r \in \Pi$; and $s' \leq s$ if there exists an expression $s = s_1 \dotsm s_r$ with $\ell(s) = r$ and the $s_i$'s in $\Pi$ such that $s'$ can be obtained by deleting some of the $s_i$'s from the product.  These definitions naturally extend to $\wt W$ via \eqref{disp:semidirect_prod_2}:  for $s$, $s' \in W_\aff$ and $\omega$, $\omega' \in \Omega$, we have $\ell(s\omega) := \ell(s)$, and $s'\omega' \leq s\omega$ exactly when $\omega' = \omega$ and $s' \leq s$ in $W_\aff$.

For parahoric subgroups $K$, $K' \subset G(L)$ attached to respective subfacets $\mathbf f$ and $\mathbf f'$ of the base alcove, the Bruhat order
% and length function 
on $\wt W$ induces
% a respective Bruhat order and length function
one on $W^{K'} \bs \wt W / W^{K}$. Indeed, let $X$ and $X'$ denote the
respective subsets of $\Pi$ of reflections fixing $\mathbf f$ and $\mathbf f'$,
let $W_X$ and $W_{X'}$\label{ref:W_X} denote the respective subgroups of
$W_\aff$ generated by $X$ and $X'$, and recall from \cite[5.2.12]{BTII} that the
parahoric subgroups of $G(L)$ are precisely the parahoric subgroups of the Tits
system \eqref{disp:G(L)_1_double_Tits_system}. Then
\[
   K = BW_XB
   \quad\text{and}\quad
   K' = BW_{X'}B,
\]
and
\[
   W^K = W_X
   \quad\text{and}\quad
   W^{K'} = W_{X'}.
\]
Hence by \cite[IV \S1 Ex.~3]{B} each double coset $w \in W^{K'} \bs \wt W / W^{K}$ contains a unique element in $\wt W$, which we denote $\wt w$, with the property that $\wt w \leq x$ for all $x \in w$.  For $w$, $w' \in W^{K'} \bs \wt W / W^{K}$, we then define
% $\ell(w) := \ell(\wt w)$ and 
$w' \leq w$ if $\wt w' \leq \wt w$ in $\wt W$.  The Bruhat order on $W^{K'} \bs \wt W / W^{K}$ has the property that if $x \leq y$ in $\wt W$, then $W^{K'} x W^{K} \leq W^{K'} y W^{K}$ in $W^{K'} \bs \wt W / W^{K}$.

\begin{Remark}\label{rk:bruhatKK}
The Bruhat order on $W^{K'} \bs \wt W / W^{K}$ can be expressed in a particularly simple way when $K=K'$ is a maximal parahoric subgroup attached to a special vertex $v$ of the base alcove. Indeed, the choice of $v$ allows us to identify $\wt W_\sc$ with the affine Weyl group of the reduced root system $\Sigma$, and $\wt W$ with a generalized extended affine Weyl group of $\Sigma$, as in Remark \ref{rk:generalized_awg} above.  The affine root hyperplanes passing through $v$ determine chambers in $\A$, and we take as positive chamber the chamber \emph{opposite} the unique chamber containing our base alcove.  This determines a notion of positive roots in $X^*(T_{\rm sc})^I$, of positive coroots in $X_*(T_{\rm sc})_I$, and of dominant elements in $X_*(T)_I$ (those elements that pair non-negatively with positive roots).  The dominance order on $X_*(T)_I$ is defined by $\lambda\leq \lambda'$ if $\lambda'-\lambda$ is a nonnegative $\mZ$-linear combination of positive  coroots.  Then, after identifying $W^{K} \bs \wt W / W^{K}$ with the set of dominant elements in $X_*(T)_I$ as in Remark \ref{rk:identWK}, the Bruhat order on $W^{K} \bs \wt W / W^{K}$ is identified with the restriction of the dominance order to the set of dominant elements,  cf.\ \cite[Cor.\ 1.8]{Ri}. (In contrast to \cite{Ri}, we use dominant elements instead of antidominant elements because we have taken the base alcove to be in the negative chamber.)
\end{Remark}
 
In the function field setting, and in analogy with the case of ordinary flag varieties, the Bruhat order
% and length function 
carries important geometric content about Schubert varieties.

\begin{prop}[Richarz {\cite[Prop.\ 2.8]{Ri}}; {\cite[Prop.\ 9.6]{P-R3}}]\label{dimandclos}
Suppose that $L = k((t))$ with $k$ algebraically closed.  Let $\mathbf f$ and $\mathbf f'$ be subfacets of the base alcove in $\A$, let $K$ and $K'$ denote the respective associated parahoric subgroups of $G(L)$, and consider the associated affine flag variety $\F_{\mathbf f}$ over $\Spec k$, cf.\ Definition  \ref{def:afv}.  For $w \in W^{K'} \bs \wt W / W^K$, let $S_w$ denote the associated $\mathbf f'$-Schubert variety in $\F_{\mathbf f}$.  Then for all $w$, $w' \in W^{K'} \bs \wt W / W^K$,
\begin{flalign*}
   \phantom{\qed} & &
   S_{w'} \subset S_w\ \text{in}\ \F_{\mathbf f} \iff w' \leq w\ \text{in}\ W^{K'} \bs \wt W / W^K. & &
   \qed
\end{flalign*}
\end{prop}

By choosing good representatives in $\wt W$ for double cosets, the inclusion relations between Schubert varieties can be phrased in a somewhat more precise way,  which is sometimes useful. We first state the following lemma.

\begin{lemma}[{Richarz, Waldspurger \cite[Lem.\ 1.9]{Ri}}] 
Let $\mathbf f$ and $\mathbf f'$ be subfacets of the base alcove in $\A$, and let $K$ and $K'$ denote the respective associated parahoric subgroups of $G(L)$. Let $w\in \widetilde W$.
\begin{altenumerate}
\item There exists a unique element $w^K$ of minimal length in $w W^K$. 
\item There exists a unique element $_{K'}w^K$ of maximal length in $\{(vw)^K\mid v\in W^{K'}\}$.
\end{altenumerate}\qed
\end{lemma}

We introduce the following subset of $\widetilde W$, 
\[
   _{K'}\widetilde W^K:= \bigl\{\, _{K'}w^K \bigm| w\in \widetilde W \,\bigr\}.
\]
Then $_{K'}\widetilde W^K$ maps bijectively to the set of double classes $W^{K'} \bs \wt W / W^{K}$, and we may phrase the inclusion relations between Schubert varieties in terms of these special representatives of double classes as follows.

\begin{prop}[{Richarz \cite[Prop.\ 2.8]{Ri}}]\label{st:schub_dim}
Let $w\in  {}_{K'}\widetilde W^K$. Then the ${\mathbf f'}$-Schubert variety $S_w$ in $\F_{\mathbf f} $ satisfies
\begin{altenumerate}
\item 
   \quad \phantom{i}$S_w=\bigcup_{\{w'\in {}_{K'}\widetilde W^K \mid w'\leq w\}}S_{w'}$; and
\item 
   \quad $\dim S_w=\ell(w) .$ \qed
\end{altenumerate}
\end{prop}

Now let us specialize to the case of special maximal parahorics. 

\begin{prop}[{Richarz \cite[Cor.\ 1.8]{Ri}}]\label{dimspec}
Let $K=K'$ be a special maximal parahoric subgroup attached to a vertex of the base alcove, and consider the dominant elements of $X_*(T)_I$ as in Remark \ref{rk:bruhatKK}.
\begin{altenumerate}
\item
   \quad \phantom{i}$_{K}\widetilde W^K=\{\,t_\lambda\mid \lambda\in X_*(T)_I \ \text{is dominant}\,\} .$ 
\item 
   \quad  $\dim S_w = \langle\lambda, 2\rho\rangle$ for $w$ corresponding to $t_\lambda\in {}_{K}\widetilde W^K$. 
\end{altenumerate}
Here $\rho$ denotes the halfsum of positive roots for $\Sigma$. \qed
\end{prop}

\subsection{The $\{\mu\}$-admissible set}
In this subsection we come to the key notion of \emph{$\{\mu\}$-admissibility}.
We continue with the notation of the previous two subsections.  Let $\{\mu\} \subset X_*(T)$ be a $W$-conjugacy class of geometric cocharacters of $T$.  Let $\wt \Lambda_{\{\mu\}} \subset \{\mu\}$ be the subset of cocharacters whose images in $X_*(T) \otimes_{\mZ} \mR$ are contained in some (absolute) Weyl chamber corresponding to a Borel subgroup of $G$ containing $T$ and defined over $L$.%
\footnote{Note that such Borel subgroups always exists since $G$ is quasi-split.}
Then $\wt \Lambda_{\{\mu\}}$ forms a single $W_0$-conjugacy class, since all such Borels are $W_0$-conjugate.  Let $\Lambda_{\{\mu\}}$ denote the image of $\wt\Lambda_{\{\mu\}}$ in $X_*(T)_I$.  Let $\mathbf a$ be an alcove in the apartment $\A$, and consider the associated Bruhat order $\leq$ on $\wt W$.  We first state a conjecture.

\begin{conjecture}\label{conj:max_elts}
Let $\{\ol\mu\}$ denote the image of the $W$-conjugacy class $\{\mu\}$ in $X_*(T)_I$.  Then the set of maximal elements in $\{\ol\mu\}$ with respect to the Bruhat order is precisely the set $\Lambda_{\{\mu\}}$.
\end{conjecture}

Of course the conjecture only has content for nonsplit $G$.  We have verified it for Weil restrictions of split groups and for unitary groups.

The validity of the conjecture not being known to us in general, we define the $\{\mu\}$-admissible set as follows.

\begin{Definition}\label{def:mu-admissible}
An element $w \in \wt W$ is \emph{$\{\mu\}$-admissible} if $w \leq t_\lambda$ for some $\lambda \in \Lambda_{\{\mu\}}$.   We denote the set of $\mu$-admissible elements in $\wt W$ by $\Adm(\{\mu\})$.  
% When $G$ is split and $T$ is a split maximal torus, we also express these by $\mu$-admissible and $\Adm(\mu)$, where $\mu$ is any element in $\{\mu\}$.
\end{Definition}

In other words, $w \in \wt W$ is $\{\mu\}$-admissible if and only if $w \leq \sigma t_{\ol\mu}\sigma^{-1} = t_{\sigma \cdot \ol\mu}$ for some $\sigma \in W_0$, where $\ol\mu$ is the image in $X_*(T)_I$ of a cocharacter $\mu \in \wt\Lambda_{\{\mu\}}$.  Since $W_0$ can be lifted to the affine Weyl group inside $\wt W$, all elements in $\Adm(\{\mu\})$ are congruent modulo $W_\aff$.  

More generally, let $K$ and $K'$ be parahoric subgroups of $G(L)$ attached to subfacets of $\mathbf a$, and consider the set of double cosets $W^{K'} \bs \wt W / W^{K}$.

\begin{Definition}\label{def:mu-adm_parahoric}
An element $w \in W^{K'} \bs \wt W / W^{K}$ is \emph{$\{\mu\}$-admissible} if
\[
   w \leq W^{K'} t_\lambda W^{K}
   \quad\text{for some}\quad
   \lambda \in \Lambda_{\{\mu\}}.
\]
We denote the set of $\{\mu\}$-admissible elements in $W^{K'} \bs \wt W / W^{K}$ by $\Adm_{K',K}(\{\mu\})$, or just by $\Adm_K(\{\mu\})$ when $K = K'$.  
\end{Definition}

Note that if Conjecture \ref{conj:max_elts} holds true, then the notion of $\{\mu\}$-admissibility is just that $w \leq W^{K'} t_{\ol\mu} W^K$ for some $\ol \mu$ in the image of $\{\mu\}$ in $X_*(T)_I$.

\begin{eg}\label{eg:adm_maximal_special}
Suppose that $K$ is a special maximal parahoric subgroup.  Then the Bruhat order on $W^K \bs \wt W / W^K$ identifies with the dominance order on the set of dominant elements in $X_*(T)_I$, as in Remark \ref{rk:bruhatKK}.  In this way the $\{\mu\}$-admissible set in $W^K \bs \wt W / W^K$ identifies with the dominant elements in $X_*(T)_I$ that are $\leq \ol \mu^{\mathrm{dom}}$ in the dominance order, where $\ol\mu^{\mathrm{dom}}$ denotes the unique dominant element in $\Lambda_{\{\mu\}}$.
\end{eg}

It is also worth making explicit the notion of $\{\mu\}$-admissibility in the setting of root data, which amounts to working in the special case that $G$ is split, cf.\ Remark \ref{rk:wtW_split}.  Let $\R = (X^*,X_*,\Phi,\Phi^\vee)$ be a root datum and $\{\mu\} \subset X_*$ a $W(\R)$-conjugacy class of cocharacters.  Again choose a base alcove and consider the induced Bruhat order on $\wt W(\R)$.  Then we define the $\{\mu\}$-admissible set
\[
   \Adm(\{\mu\}) := \bigl\{\, w \in \wt W(\R) \bigm| w \leq t_{\mu}\ \text{for some}\ \mu \in \{\mu\} \,\bigr\}.
\]
More generally, let $\mathbf f$ and $\mathbf{f'}$ be subfacets of the base alcove, and let $X$ (resp.\ $X'$) be the set of reflections across the walls of the base alcove containing $\mathbf f$ (resp.\ $\mathbf{f'}$).  As on p.\ \pageref{ref:W_X}, let $W_X$ (resp.\ $W_{X'}$) be the subgroup of $W_\aff(\R)$ generated by $X$ (resp.\ $X'$).  Then we define
\[
   \Adm_{\mathbf f, \mathbf{f'}}(\{\mu\}) := \bigl\{\, w \in W_{X'} \bs \wt W(\R) / W_X \bigm| w \leq W_{X'}t_{\mu}W_X\ \text{for some}\ \mu \in \{\mu\} \,\bigr\}.
\]
When $\mathbf f = \mathbf{f'}$, we write $\Adm_{\mathbf f}(\{\mu\}) := \Adm_{\mathbf f, \mathbf f}(\{\mu\})$.

\begin{Remark}
Let $\R$ be a root datum, choose a positive chamber, and take the base alcove to be the unique alcove contained in the negative chamber and whose closure contains the origin.  Let $\mu$ be a dominant cocharacter.  He and Lam \cite[Th.\ 2.2]{HL} have recently given a description of the partially ordered set $\Adm\bigl(W(\R)\cdot\mu\bigr) \cap W(\R)t_\mu W(\R)$ in terms of the combinatorics of \emph{projected Richardson varieties}.  Note that in the special case when $\mu$ is \emph{minuscule,} $\mu$ is minimal among dominant cocharacters in the dominance order, and it follows from Example \ref{eg:adm_maximal_special} that $\Adm\bigl(W(\R)\cdot\mu\bigr) \subset W(\R)t_\mu W(\R)$.  Thus He and Lam's result describes the full admissible set in the minuscule case.
\end{Remark}

\subsection{Relation to local models}

We continue with the notation of the previous three subsections.  Let us now return to the problem we posed at the beginning of \S\ref{s:combinatorics}, namely that of identifying the Schubert cells that occur in the geometric special fiber of a local model upon a suitable embedding into an affine flag variety.  More precisely, let $F$ be a discretely valued field with residue field $k$, let $\wt G$ be a connected reductive group over $F$, let $\{\mu\}$ be a conjugacy class of geometric cocharacters of $\wt G$, let $E$ denote the reflex field of $\{\mu\}$, and let $\O_E$ and $k_E$ denote the respective ring of integers in, and residue field of, $E$.  Suppose that for some choice of parahoric level structure we have attached a (flat) local model $M^\loc_{\wt G,\{\mu\}}$ to $\wt G$ and $\{\mu\}$ over $\Spec\O_E$; in each example we encountered in \S\ref{s:examples}, this was taken to be the scheme-theoretic closure of the generic fiber of the naive local model.  Let $F^\un$ denote the completion of a maximal unramified algebraic extension of $F$, let $\wt S$ be a maximal split torus in $\wt G_{F^\un} := \wt G \otimes_F F^\un$, let $\wt T$ be the centralizer of $\wt S$ in $\wt G_{F^\un}$, and regard $\{\mu\}$ as an absolute conjugacy class of geometric cocharacters of $\wt T$.   Let $L = \ol k_E ((t))$.  \emph{Then in every example we know,}\footnote{This will be addressed more systematically in \cite{PZ}.}
\begin{altitemize}
\item
   \emph{there exists a connected reductive group $G$ over $L$ (``a function field analog of $\wt G_{F^\un}$'') such that $G$ and $\wt G$ are forms of the same split Chevalley group defined over $\mZ$, and whose Iwahori-Weyl group $\wt W_G$ naturally identifies with $\wt W_{\wt G_{F^{\un}}}$;}
\item
   \emph{the geometric special fiber $\ol M^\loc_{\wt G,\{\mu\}} \otimes_{k_E} \ol k_E$ embeds $L^+P$-equivariantly in the affine flag variety $\F_P$ for $G$, where $P$ is a parahoric group scheme  for $G$ corresponding to the original choice of parahoric level structure;}
\item
   \emph{and with regard to this embedding, the Schubert cells occurring in $\ol M^\loc_{\wt G,\{\mu\}} \otimes_{k_E} \ol k_E$ are parametrized by precisely the $\{\mu\}$-admissible set, regarded as a subset of $W^{K} \bs \wt W_G / W^{K}$ via the above bijection $\wt W_G \cong \wt W_{\wt G_{F^\un}}$, where $K$ denotes the parahoric subgroup $P(\O_L) \subset G(L)$.}
\end{altitemize}   
Note that this says that the irreducible components of $\ol M^\loc_{\wt G,\{\mu\}} \otimes_{k_E} \ol k_E$, which correspond to the Schubert cells in $\ol M^\loc_{\wt G,\{\mu\}} \otimes_{k_E} \ol k_E$ that are  maximal for the inclusion relation of their closures, are exactly parametrized by the elements of $\Lambda_{\{\mu\}}$.

\begin{eg}[$GL_n$]\label{eg:GL_n_combinatorics}
Let $\wt G = GL_n$ over $F$, and let $\{\mu\}$ be the conjugacy class of $\mu = \bigl(1^{(r)},0^{(n-r)}\bigr)$, as in \S\ref{s:examples}.\ref{ss:GL_n}.  Then $E=F$.  Let $\L$ be the standard lattice chain $\Lambda_\mZ$ in $F^n$.  We take $G := GL_n$ over $L$ with split maximal diagonal torus $S$ and Iwahori-Weyl group $\wt W = \wt W_{G,S}$ as in Example \ref{eg:wtW_GL_n}, and we embed
\[
   \ol M^\loc_{GL_n,\{\mu\},\Lambda_\mZ} \otimes_k \ol k \to \F_\mZ
\]
as in Example \ref{eg:GL_n_emb}.  Then an element $w \in \wt W = N(L)/S(\O_L)$ specifies a Schubert cell contained in the image of $\ol M^\loc_{GL_n,\{\mu\},\Lambda_\mZ} \otimes_k \ol k$ exactly when $w \cdot \lambda_\mZ$ is contained in the image of $\ol M^\loc_{GL_n,\{\mu\},\Lambda_\mZ} \otimes_k \ol k$, that is, exactly when the lattice chain $w \cdot \lambda_\mZ$ satisfies
\begin{enumerate}
\item\label{it:ctn_cond}
   $\lambda_i \supset w \cdot \lambda_i \supset t\lambda_i$ for all $i$; and
\item\label{it:size_cond}
   $\dim_{\ol k} (w \cdot \lambda_i/t\lambda_i) = n - r$ for all $i$.
\end{enumerate}

To translate these conditions into more explicit combinatorics, let us identify each $\O_L$-lattice of the form $t^{i_1}\O_L \oplus \dotsb \oplus t^{i_n}\O_L$ with the vector $(i_1,\dotsc,i_n) \in \mZ^n$.  Then with regard to our identifications, the natural action of $N(L)/S(\O_L)$ on lattices translates to the natural action of $\mZ^n \rtimes S_n$ on $\mZ^n$ by affine transformations, with $\mZ^n$ acting by translations and $S_n$ acting by permuting coordinates.  For $i = nd + j$ with $0 \leq j < n$, the lattice $\lambda_i$ translates to the vector
\[
   \omega_i := \bigl( (-1)^{(j)},0^{(n-j)} \bigr) - \mathbf d,
\]
where for any $d$ we write boldface $\mathbf d$ to denote the vector $(d,\dotsc,d)$.
Conditions \eqref{it:ctn_cond} and \eqref{it:size_cond} become equivalent to
\begin{enumerate}
\renewcommand{\theenumi}{\arabic{enumi}$'$}
\item\label{it:ctn_cond'}
   $\mathbf 0 \leq w\cdot \omega_i - \omega_i \leq \mathbf 1$ for all $i$; and
\item\label{it:size_cond'}
   for all $i$, the sum of the entries of the vector $w\cdot \omega_i - \omega_i$ is $r$.
\end{enumerate}
Note that $\mu$ and all its Weyl conjugates, regarded as translation elements in $\wt W$, trivially satisfy \eqref{it:ctn_cond'} and \eqref{it:size_cond'}.  The main result for $GL_n$ in \cite{KR} is that the set of all $w \in \wt W$ satisfying \eqref{it:ctn_cond'} and \eqref{it:size_cond'} is precisely the set $\Adm(\{\mu\})$, where the Bruhat order is taken with respect to the alcove determined by the $\omega_i$'s.%
\footnote{Note that this alcove is the alcove contained in the \emph{negative} Weyl chamber (relative to the standard choice of positive roots) and whose closure contains the origin.  This is the motivation for our convention in defining the positive chamber in Remark \ref{rk:bruhatKK}.}
Entirely analogous remarks hold for any subchain $\Lambda_I$ of $\Lambda_\mZ$.
\end{eg}

Let us return to the general discussion, with $L$ again an arbitrary complete, discretely valued, strictly Henselian field.  Taking note that, in the previous example, the images of the $\omega_i$'s in the standard apartment for $PGL_n$ are the vertices of the base alcove, the papers \cite{KR,R} abstract conditions \eqref{it:ctn_cond'} and \eqref{it:size_cond'} to any Iwahori-Weyl group as follows.  Let $T_\ad$ denote the image of $T$ in $G_\ad$.  Consider the composition
\[
   X_*(T)_I \to X_*(T_\ad)_I \to X_*(T_\ad)_I \otimes_\mZ \mR \cong X_*(S_\ad) \otimes_\mZ \mR = \A,
\]
and let $\P_{\{\mu\}}$ denote the convex hull in $\A$ of the image of the set $\Lambda_{\{\mu\}}$.

\begin{Definition}\label{def:mu-permissible}
An element $w \in \wt W$ is \emph{$\{\mu\}$-permissible} if
\begin{altitemize}
\item
   $w \equiv t_{\ol\mu} \bmod W_\aff$ for one, hence any, $\ol\mu \in \Lambda_{\{\mu\}}$; and
\item
   $w\cdot x - x \in \P_{\{\mu\}}$ for all $x \in \mathbf a$.
\end{altitemize}
More generally, for any subfacet $\mathbf f$ of $\mathbf a$ with associated parahoric subgroup $K$, an element $w \in W^K \bs \wt W / W^K$ is \emph{$\{\mu\}$-permissible} if $w \equiv t_{\ol\mu} \bmod W_\aff$ for any $\ol\mu \in \Lambda_{\{\mu\}}$ and $w\cdot x -x \in \P_{\{\mu\}}$ for all $x \in \mathbf f$.  We write $\Perm(\{\mu\})$ for the set of $\{\mu\}$-permissible elements in $\wt W$ and $\Perm_K(\{\mu\})$ for the set of $\{\mu\}$-permissible elements in $W^K \bs \wt W / W^K$.
\end{Definition}

Note that for $w \in  W^K \bs \wt W / W^K$, the condition $w \equiv t_{\ol\mu} \bmod W_\aff$ is well-defined because $W^K \subset W_\aff$, and the condition 
\begin{equation}\label{disp:conv_cond}
   w\cdot x -x \in \P_{\{\mu\}} \quad\text{for all}\quad x \in \mathbf f
\end{equation}
is well-defined by \cite[\S3, p.\ 282]{R}.  By convexity, \eqref{disp:conv_cond} is equivalent to requiring that $w\cdot x -x \in \P_{\{\mu\}}$ for all \emph{vertices} $x$ of $\mathbf f$.

In the case of $GL_n$ and $\{\mu\}$ the conjugacy class of $\mu = \bigl(1^{(r)},0^{(n-r)}\bigr)$ from Example \ref{eg:GL_n_combinatorics}, one sees almost immediately that the set of elements in $\wt W$ satisfying \eqref{it:ctn_cond'} and \eqref{it:size_cond'} is precisely $\Perm(\{\mu\})$.  Thus the main result for $GL_n$ in \cite{KR} is to establish the equality $\Adm(\{\mu\}) = \Perm(\{\mu\})$ for such $\mu$.

In many (but not all!) cases known to us, the Schubert cells in the special fiber of the local model turn out to be parametrized by the $\{\mu\}$-permissible set, i.e.\ one has an equality between $\{\mu\}$-admissible and $\{\mu\}$-permissible sets.  And in explicit computations it is easier to determine the $\{\mu\}$-permissible set than the $\{\mu\}$-admissible set.  Thus it is of interest to understand the relationship between $\{\mu\}$-admissibility and $\{\mu\}$-permissibility.  The first results in this direction are the following.

\begin{prop}
\begin{altenumerate}
\renewcommand{\theenumi}{\roman{enumi}}
\item 
   \emph{(\cite[\S11]{KR})} For any $G$ and any $\{\mu\}$, $\Perm(\{\mu\})$ is closed in the Bruhat order and $\Adm(\{\mu\}) \subset \Perm(\{\mu\})$.
\item\label{it:Adm_neq_Perm}
   \emph{(Haines--Ng\^o \cite[7.2]{HN2})} The reverse inclusion can fail.  More precisely, suppose that $G$ is split over $L$ with irreducible root datum of rank $\geq 4$ and not of type $A$.  Then $\Adm(\{\mu\}) \neq \Perm(\{\mu\})$ for $\{\mu\}$ the conjugacy class of any sufficiently regular cocharacter $\mu$.\qed
\end{altenumerate}
\end{prop}

In \eqref{it:Adm_neq_Perm}, we refer to the proof of the cited result for the precise meaning of ``sufficiently regular.''  We also note that in \cite{Sm4} it is shown that $\Adm(\{\mu\}) \neq \Perm(\{\mu\})$ for $\{\mu\}$ the Weyl orbit of the coweight $\bigl(1^{(r)},0^{(m-r)}\bigr)$ for $B_m$ (using the standard coordinates, as in \cite[Pl.~II]{B}) for $m$, $r \geq 3$.

While $\{\mu\}$-admissibility and $\{\mu\}$-permissibility are not equivalent in general, the following result gives a summary of most situations in which they are known to coincide.  We shall formulate the results for extended affine Weyl groups attached to root data; in the most literal sense one may regard this as an assumption that $G$ is split over $L$, as in Remark \ref{rk:wtW_split}, but see Remark \ref{rk:nonsplit_case} below for the relevance of this to the nonsplit case.  Given a root datum $\R = (X^*,X_*,\Phi,\Phi^\vee)$ and a $W(\R)$-conjugacy class $\{\mu\} \subset X_*$ of cocharacters, we define $\Perm(\{\mu\})$ in obvious analogy with Definition \ref{def:mu-permissible},
\[
   \Perm(\{\mu\}) := \biggl\{\, w \in \wt W(\R) \biggm| 
   \twolinestight{$w \equiv t_{\mu} \bmod W_\aff(\R)$ for any $\mu \in \{\mu\}$ and}{$w\cdot x-x \in \wt\P_{\{\mu\}}$ for all $x$ in the base alcove}
   \,\biggr\},
\]
where $\wt \P_{\{\mu\}}$ denotes the convex hull of $\{\mu\}$ in $X_* \otimes_\mZ \mR$.
% {\bf I FIND THIS CONFUSING, AND WOULD PREFER IT IF WE COULD INDEED SAY THAT $G$ IS SPLIT, MR}

\begin{prop}\label{st:adm_vs_perm}
Let $\wt W$ be the extended affine Weyl group attached to a root datum \R, as in Remark \ref{rk:wtW_split}, and take the Bruhat order on $\wt W$ corresponding to a base alcove $\mathbf a$.
\begin{altenumerate}
\renewcommand{\theenumi}{\roman{enumi}}
\item\label{it:A}
   \emph{(Haines--Ng\^o \cite[3.3]{HN2}; \cite[3.5]{KR})}\, If \R involves only type $A$, then $\Adm(\{\mu\}) = \Perm(\{\mu\})$ for any $W(\R)$-conjugacy class $\{\mu\}$.
\item\label{it:GSp_Adm=Perm}
   \emph{(Haines--Ng\^o \cite[10.1]{HN2}; \cite[4.5, 12.4]{KR})}\, Suppose that $\wt W$ is the Iwahori-Weyl group of $GSp_{2g}$ and that $\{\mu\} = W(\R)\cdot \mu$ for $\mu$ a \emph{sum of dominant minuscule} cocharacters.  Then $\Adm(\{\mu\}) = \Perm(\{\mu\})$.
\item\label{it:general_classical}
   \emph{(\cite[3.5, 4.5]{KR}, \cite[7.6.1]{Sm1}, \cite[Main Theorem]{Sm2})} Suppose that \R involves only types $A$, $B$, $C$, and $D$ and that $\{\mu\}$ is a $W(\R)$-conjugacy class of \emph{minuscule} cocharacters.  Then $\Adm(\{\mu\}) = \Perm(\{\mu\})$.\qed
\end{altenumerate}
\end{prop}

In \eqref{it:GSp_Adm=Perm}, a cocharacter $\mu$ is a sum of dominant minuscule cocharacters (with respect to the standard choice of positive Weyl chamber) exactly when it is of the form $\bigl(n^{(g)},0^{(g)}\bigr) + \mathbf d$ for some $n \in \mZ_{\geq 0}$ and $d \in \mZ$, in the notation of Example \ref{eg:wtW_GSp}.

\begin{Remark}\label{rk:adm=perm_parahoric}
As stated, the proposition covers only the Iwahori case, but it is known to generalize to the general parahoric case.  To be precise, let $\mathbf f$ be a subfacet of $\mathbf a$, let $X$ be the set of reflections across the walls of the base alcove containing $\mathbf f$, and let $W_X$ be the subgroup of $W_\aff(\R)$ generated by $X$.  Then, in analogy with Definition \ref{def:mu-permissible}, we define $\Perm_{\mathbf f}(\{\mu\})$ to be the set of all $w \in W_X \bs \wt W(\R) / W_X$ such that $w \equiv t_{\mu} \bmod W_\aff(\R)$ for any $\mu \in \{\mu\}$, and such that $\wt w\cdot x-x \in \wt\P_{\{\mu\}}$ for all $x \in \mathbf f$, where $\wt w$ is any representative of $w$ in $\wt W(\R)$ (this is again independent of the choice of $\wt w$ by \cite[\S3, p.\ 282]{R}).

Then in \eqref{it:A}, we have $\Adm_{\mathbf f}(\{\mu\}) = \Perm_{\mathbf f}(\{\mu\})$ for any $\{\mu\}$ and any $\mathbf f$ when \R involves only type $A$; this was proved in the case of minuscule $\{\mu\}$ in \cite[9.6]{KR}, and the general case is an immediate consequence of G\"ortz's result \cite[Cor.\ 9]{Go4} (which itself makes crucial use of the cited result in the Iwahori case of Haines--Ng\^o).

In \eqref{it:GSp_Adm=Perm}, we have $\Adm_{\mathbf f}(\{\mu\}) = \Perm_{\mathbf f}(\{\mu\})$ inside $W_X \bs \wt W_{GSp_{2g}} / W_X$ for any $\mathbf f$ and any $\{\mu\}$ which is the conjugacy class of a sum of dominant minuscule coweights; this was proved in the case of minuscule $\{\mu\}$ in \cite[10.7]{KR}, and the general case is an immediate consequence of G\"ortz's result \cite[Cor.\ 13]{Go4} (which again relies on the Iwahori case established in \cite{HN2}).

It follows that the parahoric version of \eqref{it:general_classical} holds for any $\mathbf f$ and any minuscule $\{\mu\}$ provided \R involves only types $A$ and $C$.  On the other hand, the general parahoric version of \eqref{it:general_classical} for types $B$ and $D$ will be deduced in \cite{Sm5} from the Iwahori case for these types, via arguments along the lines of those in \cite{KR} or \cite{Go4}.
    % {\bf MR: HERE WE SHOULD  SAY WHETHER THERE IS A PARAHORIC VERSION OF PROPOSITION \ref{st:adm_vs_perm}, WITH REFERENCES, e.g. to \cite{KR, Go4} } 
\end{Remark}

\begin{Remark}\label{rk:nonsplit_case}
Proposition \ref{st:adm_vs_perm} is useful for more than just the case that $G$ is split.  Indeed, for any group $G$, questions of admissibility and permissibility in $\wt W$ can \emph{always} be reduced to the case of an extended affine Weyl group attached to a root datum.  The link is made via the reduced root system $\Sigma$ on $\A$ attached to the affine root system $\Phi_\aff$ for $G$, as discussed in \S\ref{s:combinatorics}.\ref{ss:bruhat_order}, p.\ \pageref{ref:Sigma}.

Consider the group $X_*(T_\ad)_I$.  By \cite[4.4.16]{BTII}, $X_*(T_\ad)$ is an induced Galois module.  Hence $X_*(T_\ad)_I$ is torsion-free.  And by \cite[Lem.\ 15]{H-R}, we have
\[
   Q^\vee(\Sigma) \subset X_*(T_\ad)_I \subset P^\vee(\Sigma),
\]
where $Q^\vee(\Sigma)$ and $P^\vee(\Sigma)$ denote the respective coroot and coweight lattices for $\Sigma$.  Hence
\[
   \R := \bigl( X^*(T_\ad)^I, X_*(T_\ad)_I, \Sigma, \Sigma^\vee \bigr)
\]
is a root datum.  For any $v \in \A$ which is a special vertex relative to the affine root system for $G$, the image of the composition
\[
   \wt W \xra\nu \A \rtimes W_0 
      \xra{\substack{\text{conjugation}\\ \text{by }t_{-v}}} \A \rtimes W_0
\]
is contained in $\wt W(\R) = X_*(T_\ad)_I \rtimes W_0$; let us write
\[
   f\colon \wt W \to \wt W(\R).
\]
Then,  on translation elements $f$ restricts to the natural map $X_*(T)_I \to X_*(T_\ad)_I$, and $f$ carries $W^K \subset \wt W$ isomorphically to $W_0$, where $K \subset G(L)$ is the parahoric subgroup attached to $v$ and $W^K$ is the subgroup \eqref{disp:W^K}.

For the present discussion it is necessary to understand the map $f$ in terms of the semidirect product decomposition $\wt W = W_\aff \rtimes \Omega$ \eqref{disp:semidirect_prod_2}, where $\Omega$ is the stabilizer of the base alcove $\mathbf a$ inside $\wt W$.  Inside $\wt W(\R)$ is the affine Weyl group $W_\aff(\R) = Q^\vee(\Sigma) \rtimes W_0$, and we denote by $\Omega(\R)$ the stabilizer in $\wt W(\R)$ of the translate $\mathbf a -v$, which is an alcove in $\A$ for $\Sigma$.  Then $\wt W(\R) = W_\aff(\R) \rtimes \Omega(\R)$, and $f$ restricts to an isomorphism $W_\aff \isoarrow W_\aff(\R)$ and a map $\Omega \to \Omega(\R)$.  Endow $\wt W$ with the Bruhat order corresponding to $\mathbf a -v$.  Then it is clear that
\begin{altitemize}
\item
   $w' \leq w$ in $\wt W$ $\implies$ $f(w') \leq f(w)$ in $\wt W(\R)$, with the converse holding exactly when $w' \equiv w \bmod W_\aff$;
\end{altitemize}
and that for any $W$-conjugacy class $\{\mu\} \subset X_*(T)$,
\begin{altitemize}
\item
   $f$ carries the subset $\Adm(\{\mu\}) \subset \wt W$ bijectively onto $\Adm(\{\ol\mu_\ad\}) \subset \wt W(\R)$, where $\{\ol\mu_\ad\}$ denotes the image of $\Lambda_{\{\mu\}}$ in $X_*(T_\ad)_I$; and
\item
   $f$ carries the subset $\Perm(\{\mu\}) \subset \wt W$ bijectively onto $\Perm(\{\ol\mu_\ad\}) \subset \wt W(\R)$.
\end{altitemize}
Moreover, we have
\begin{altitemize}
\item
   $\Adm(\{\mu\}) = \Perm(\{\mu\})$ in $\wt W$ $\iff$ $\Adm(\{\ol\mu_\ad\}) = \Perm(\{\ol\mu_\ad\})$ in $\wt W(\R)$.
\end{altitemize}
\end{Remark}

\begin{Remark}
The following variant of the preceding remark, which uses the building for $G$ in place of the building for $G_\ad$, is sometimes more convenient in practice.

Let $\wt \A := X_*(S) \otimes \mR \cong X_*(T)_I \otimes \mR$, and consider the natural map
\[
   T(L)/T(L)_1 \isoarrow X_*(T)_I \to \wt \A.
\]
Then there exists an extension of the displayed composite to a map $\wt W \xra{\wt\nu} \wt \A \rtimes W_0$.  More precisely, replacing $\A$ with $\wt \A$ everywhere in the diagram \eqref{disp:nu_diag}, there exists a map $\wt W \xra{\wt\nu} \wt \A \rtimes W_0$ making the diagram commute, and any two extensions differ by conjugation by a translation element, but this translation element is no longer uniquely determined.

Let $\Phi_\aff$ denote the affine root system for $G$ relative to the composite
\[
   \wt W \xra{\wt\nu} \wt\A \rtimes W_0 \to \A \rtimes W_0.
\]
Let $\Sigma$ denote the associated reduced root system on $\A$, as in the preceding remark.  Then we can regard the elements of $\Sigma$ as linear functions on $\wt A$, and the $W_0$-action on $\wt \A$ allows us to canonically lift the coroots to $\wt\A$:  for each root $\alpha \in \Sigma$ we have the associated reflection $s_\alpha \in W_0$, and this determines the associated coroot $\alpha^\vee$ via the formula
\[
   s_\alpha(x) = x - \la \alpha, x \ra \alpha^\vee
   \quad\text{for}\quad
   x \in \wt \A.%
\footnote{Of course, we can also canonically lift the coroots to $\wt\A$ via the embedding $X_*(T_\sc)_I \inj X_*(T)_I$ discussed in \S\ref{s:combinatorics}.\ref{ss:bruhat_order}, p.\ \pageref{ref:X_*(T_sc)_I}.}
\]

Finally let $N$ denote the torsion subgroup of $X_*(T)_I$.  Then
\[
   \wt\R := \bigl( X^*(T)^I, X_*(T)_I/N, \Sigma, \Sigma^\vee \bigr)
\]
is a root datum, and everything carries over from the previous remark with $\wt\R$ in place of $\R$.
\end{Remark}

\begin{Remark}
Although we are interested in minuscule conjugacy classes of cocharacters for applications to Shimura varieties, we caution that, in the context of the previous two remarks, the image of a minuscule $\wt\Lambda_{\{\mu\}}$ in $X_*(T)_I$ or $X_*(T_\ad)_I$  \emph{need not be minuscule for $\Sigma$.}  In this way the study of admissibility for non-minuscule cocharacters in root data is relevant to the study of admissibility for minuscule cocharacters in nonsplit groups.
\end{Remark}

\begin{Remark}\label{rk:R_conj}
It is conjectured in \cite[\S3, p.~283]{R} that $\Adm(\{\mu\}) = \Perm(\{\mu\})$ for any Weyl orbit $\{\mu\}$ of minuscule cocharacters in any extended affine Weyl group attached to a based root datum.  Thus part \eqref{it:general_classical} of  Proposition \ref{st:adm_vs_perm} is a partial confirmation of this conjecture.  In fact, \cite{R} formulates the more optimistic conjecture that $\Adm(\{\mu\}) = \Perm(\{\mu\})$ whenever $\{\mu\}$ is the conjugacy class attached to a sum $\mu$ of dominant minuscule cocharacters. However,  this more optimistic version of the conjecture can fail, cf.\ \cite{Sm4}.  In particular, $\Adm(\{\mu\}) \neq \Perm(\{\mu\})$ for $\mu$ the sum of dominant minuscule coweights
\[
   (1,1,1,0) = \bigl(\tfrac 1 2, \tfrac 1 2, \tfrac 1 2, \tfrac 1 2\bigr) + \bigl(\tfrac 1 2, \tfrac 1 2, \tfrac 1 2, - \tfrac 1 2\bigr)
\]
in $D_4$ (using the standard coordinates, as in \cite[Pl.~IV]{B}).
\end{Remark}

Let us conclude this subsection by making more explicit the relation of the $\{\mu\}$-admissible and $\{\mu\}$-permissible sets to the local models discussed in this article and elsewhere in the literature.  For all of the local models attached to $GL_n$ in \S\ref{s:examples}.\ref{ss:GL_n} \cite{Go1}, $GSp_{2g}$ in \S\ref{s:examples}.\ref{ss:GSp_2g} \cite{Go2}, $\Res_{F/F_0} GL_n$ in \S\ref{s:examples}.\ref{ss:ResGL_n} \cite{Go4,P-R2}, $\Res_{F/F_0}GSp_{2g}$ in \S\ref{s:examples}.\ref{ss:ResGSp} \cite{Go4,P-R2}, and ramified, quasi-split $GU_n$ in \S\ref{s:examples}.\ref{ss:GU_n} \cite{P-R4,Sm3,Sm4}, the geometric special fiber of the local model $M^\loc_{G,\mu,\L}$ admits an embedding into an affine flag variety --- constructed very much in the spirit of \S\ref{s:afv}.\ref{ss:embedding} --- with regard to which it decomposes into a union of Schubert cells indexed by exactly the $\{\mu\}$-admissible set.  In \S\ref{s:examples}.\ref{ss:GO_2g}, the orthogonal group $GO_{2g}$ is disconnected, so that as in Remark \ref{rk:wtW_GO} the present discussion does not literally apply.  Nevertheless, the special fiber of the local model, which has two connected components, can still be embedded into an affine flag variety for $GO_{2g}$, where it is found to contain the Schubert cells indexed by \emph{two} admissible sets for $GO^\circ_{2g}$:  one for the conjugacy class of the cocharacter $\bigl(1^{(g)},0^{(g)}\bigr)$, and the other for the conjugacy class of the cocharacter $\bigl(1^{(g-1)},0,1,0^{(g-1)}\bigr)$.  Note that these cocharacters are $GO_{2g}$-conjugate but not $GO_{2g}^\circ$-conjugate.  See \cite{Sm1}.

In all of the examples mentioned in the previous paragraph, one also has an equality between $\{\mu\}$-admissible and $\{\mu\}$-permissible sets with the single exception of ramified, quasi-split $GU_n$ for $n$ even and $\geq 4$, in which case $\{\mu\}$-admissibility and $\{\mu\}$-permissibility are typically not equivalent.  See \cite{Sm4}.

\subsection{Vertexwise admissibility}\label{ss:vwise_adm}

We continue with the notation of the previous four subsections.  Let $K$ and $K'$ be parahoric subgroups of $G(L)$ attached to subfacets of the base alcove $\mathbf a$, and let $\{\mu\} \subset X_*(T)$ be a $W$-conjugacy class. It is an immediate consequence of the properties of the Bruhat order that the canonical projection $\wt W \to W^{K'} \bs \wt W / W^{K}$ induces a surjective map
\begin{equation}\label{disp:Adm->Adm_K,K'}
   \Adm(\{\mu\}) \surj \Adm_{K',K}(\{\mu\}).
\end{equation}
If $\mathbf f$ is a subfacet of $\mathbf a$ with associated parahoric subgroup $K$, then for each vertex $x$ of $\mathbf f$, let $K_x$ denote the associated parahoric subgroup and
\[
   \rho_{x}\colon W^K \bs \wt W / W^K \surj W^{K_x} \bs \wt W / W^{K_x}
\]
the canonical projection.  We make the following definition.

\begin{Definition}
The \emph{$\{\mu\}$-vertexwise admissible set} in $W^K \bs \wt W / W^K$ is the subset
\[
   \Adm_K^\vert(\{\mu\}) := \bigcap_{\substack{\text{vertices}\\ x\text{ of } \mathbf f}} \rho_x^{-1} \bigl(\Adm_{K_x}(\{\mu\})\bigr).
\]
\end{Definition}

In other words, an element $w \in W^K \bs \wt W / W^K$ is $\{\mu\}$-vertexwise admissible if $W^{K_x} w W^{K_x} \in \Adm_{K_x}(\{\mu\})$ for all vertices $x$ of $\mathbf f$.  It is an obvious consequence of the map \eqref{disp:Adm->Adm_K,K'} that $\Adm_K(\mu) \subset \Adm^\vert_K(\mu)$, and we conjecture the following.

\begin{conjecture}\label{st:vert_adm}
Let $\{\mu\} \subset X_*(T)$ be a $W$-conjugacy class of minuscule cocharacters, and let $\mathbf f$ be a subfacet of $\mathbf a$  with associated parahoric subgroup $K$.  Then the inclusion $\Adm_K(\{\mu\}) \subset \Adm^\vert_K(\{\mu\})$ is an equality.
\end{conjecture}

% Here by $\mu$ \emph{minuscule} we mean that $\la \mu,\alpha \ra \in \{-1,0,1\}$ for all absolute roots $\alpha$ of $G$.  {\bf WHY AGAIN??MR} 
We do not know if the assumption that $\{\mu\}$ be minuscule is necessary, but the examples that we have studied all arise from local models, where the assumption holds by definition.

We note that in cases where $\{\mu\}$-admissibility and $\{\mu\}$-permissibility are \emph{equivalent}, the conjecture is automatic.  Indeed, for any $\{\mu\}$ we have $\Adm_K^\vert(\{\mu\}) \subset \Perm_K(\{\mu\})$ because $\Adm_{K'}(\{\mu\}) \subset \Perm_{K'}(\{\mu\})$ for any $K'$, and in particular for $K'$ of the form $K_x$, and because $\{\mu\}$-permissibility is manifestly a vertexwise condition.  Hence the equality $\Adm_K(\{\mu\}) = \Perm_K(\{\mu\})$ implies the equality $\Adm_K(\{\mu\}) = \Adm_K^\vert(\{\mu\})$.  Because of this, the conjecture may in some sense be regarded as a version for arbitrary groups of the conjecture in \cite{R} that $\Adm(\{\mu\}) = \Perm(\{\mu\})$ for minuscule cocharacters in \emph{split} groups; see Remark \ref{rk:R_conj}.
% {\bf I WOULD PREFER IF YOU SAID THAT THE CONJECTURE ABOVE IS THE NEXT BEST THING IN THE GENERAL CASE (WHEN $\mu$ IS MINUSCULE) TO THE CONJECTURE $Adm(\mu)=Perm(\mu)$ WHICH IS SUPPOSED TO HOLD WHEN $G$ SPLITS OVER $L$ MR}

We also note that the conjecture holds in all examples that we know of arising from local models.  
% {\bf MR: GIVE REFERENCES FOR THE NEXT CLAIMS} 
More precisely, for all of the local models attached to $GL_n$ in \S\ref{s:examples}.\ref{ss:GL_n}; $GSp_{2g}$ in \S\ref{s:examples}.\ref{ss:GSp_2g}; $GO_{2g}$ in  \S\ref{s:examples}.\ref{ss:GO_2g}; $\Res_{F/F_0} GL_n$ in \S\ref{s:examples}.\ref{ss:ResGL_n};  and $\Res_{F/F_0}GSp_{2g}$ in \S\ref{s:examples}.\ref{ss:ResGSp}, the conjecture holds because $\{\mu\}$-admissibility and $\{\mu\}$-permissibility are equivalent by Proposition \ref{st:adm_vs_perm} and Remark \ref{rk:adm=perm_parahoric}. 
% For the local models attached to $GO_{2g}$ in \S\ref{s:examples}.\ref{ss:GO_2g}, the conjecture is known in the Iwahori case via the equivalence of $\{\mu\}$-admissibility and $\{\mu\}$-permissibility in Proposition \ref{st:adm_vs_perm}\eqref{it:general_classical}; and we expect that the general parahoric case should not be difficult to deduce from the Iwahori case.  
For the local models attached to ramified, quasi-split $GU_n$ for $n$ odd in \S\ref{s:examples}.\ref{ss:GU_n}, the conjecture is known via the equivalence of $\{\mu\}$-admissibility and $\{\mu\}$-permissibility \cite{Sm3}, but these cases are not covered by Proposition \ref{st:adm_vs_perm}.  Finally, for the local models attached to ramified, quasi-split $GU_n$ for $n$ even in \S\ref{s:examples}.\ref{ss:GU_n}, $\{\mu\}$-admissibility and $\{\mu\}$-permissibility are typically not equivalent, but the conjecture still holds in these cases \cite{Sm4}.

\section{Local models and nilpotent orbits}\label{s:no}

In a few cases, the special fibers of local models can be described via nilpotent orbits
and their closures. As was first observed in \cite{P-R1}, this connection is especially tight in the case
of the (ramified) group ${\rm Res}_{F/F_0}GL_n$. This also gives a connection between 
affine Schubert varieties for $SL_n$ and nilpotent orbit closures. In this section we discuss this relation in a somewhat informal manner. 

\subsection{Nilpotent orbits}\label{ss:no} 

Let $G$ be a reductive group over a field $k$ and denote by ${\mathfrak g}$ its Lie algebra, 
which we think of as an affine space.
Recall that an element $x$ of  $\mathfrak g$ is called nilpotent if its adjoint endomorphism ${\rm ad}(x)\colon {\mathfrak g} \to {\mathfrak g}$ 
is nilpotent.  
The property of being nilpotent is invariant under the adjoint action of $G$ on $\mathfrak g$;
a nilpotent orbit $N_x=\{{\rm ad}(g)\cdot x\ |\ g\in G\}$ is the orbit of a nilpotent element $x$ under the adjoint action. Here we consider $N_x$ as the reduced subscheme with underlying topological space the orbit of $x$.  We will denote by
$\ov{N_x}$ the Zariski closure of $N_x$ in the affine space ${\mathfrak g}$.
The  varieties $\ov{N_x}$  have been the subject of intense study (\cite{KP,BC,dCP2}, etc.)  
The most classical example of course is when $G=GL_r$ and ${\mathfrak g}={\rm Mat}_{r\times r}$.
Then $N_A$ is the conjugation orbit of the nilpotent matrix $A$. These orbits
are parametrized by partitions $\bold r=(r_1\geq r_2\geq \cdots \geq r_s)$
of $r$; the numbers $r_i$ are the sizes of the blocks in the Jordan decomposition of $A$.

\subsection{Relations to local models}\label{ss:lm&no}
We consider the situation of \S\ref{s:examples}.\ref{ss:ResGL_n}, i.e take $G={\rm Res}_{F/F_0}GL_n$, where $F/F_0$ is a totally ramified separable extension of degree $e$. Let $\pi$ be a uniformizer of $\O_F$, and let $Q(T)\in \O_{F_0}[T]$ be the Eisenstein polynomial satisfied by $\pi$. 

Recall from loc.\ cit.\  that the minuscule cocharacter
$\mu$ is determined by choosing $r_\varphi$ with $0\leq r_\varphi\leq n$, for  each embedding $\varphi$ 
of $F$ in a fixed  algebraic closure $\ol F_0$ of $F_0$.  
We choose the lattice chain ${\mathcal L}=\{\pi^k \Lambda_0\}_{k\in {\mathbb Z}}$ 
to be given by the multiples of the standard $\O_F$-lattice in $F^n$. 
(Then the corresponding parahoric group
is maximal and special). We denote by $M$ the naive local model $M^{\rm naive}_{G, \{\mu\}, \L}$
for these choices (defined in \S\ref{s:examples}.\ref{ss:ResGL_n}) and write $M^{\rm loc}$ for the corresponding 
local model $M^{\rm loc}_{G, \{\mu\},\L}$ given as the flat closure of $M\otimes_{\O_E}E$ over the ring of 
integers $\O_E$ of the reflex field $E$. Denote by $k_E$ the residue field of $\O_E$.

Set $r=\sum_{\varphi}r_\varphi$. Then the subspace $\F:=\F_{\La_0}\subset \La_{0, S}:=\La_0\otimes_{\Z_p}\O_S$ occurring 
in the definition of $M$ is locally on $S$ free of rank $n-r$. This allows us to
 consider the $GL_r$-torsor $\wt M$ over $M$ defined
by 
$$
\wt M(S)=\bigl\{\, (\F, \alpha) \bigm|  \F\in M(S),\ \alpha\colon \Lambda_{0, S}/\F\xrightarrow{\sim} \O_S^r \,\bigr\}\ ,
$$
and construct a $GL_r$-equivariant morphism
$$
\wt q\colon \wt M\to N\ ,
$$
with
\begin{equation}\label{nileq}
N:=\biggl\{\,A\in {\rm Mat}_{r\times r}\biggm| \det(X\cdot I-A)\equiv \prod_\varphi \bigl(X-\varphi(\pi)\bigr)^{r_\varphi},\  Q(A)=0\,\biggr\}\ , 
\end{equation}
where the $GL_r$-action on the target is via conjugation.
The morphism $\wt q$ is smooth \cite[Th.\ 4.1]{P-R1}, and 
hence we obtain a smooth morphism of algebraic stacks
\begin{equation}
q\colon M\to [GL_r\backslash N].
\end{equation}
Note that the special fiber $N\otimes_{\O_E}k_E$ 
is the $GL_r$-invariant subscheme of the nilpotent matrices 
${\rm Nilp}_{r\times r}$ over $k_E$, given as 
\begin{equation}
N\otimes_{\O_E}k_E = \bigl\{\, A\in {\rm Mat}_{r\times r} \bigm| \det(X\cdot I-A)\equiv X^r,\  A^e=0 \,\bigr\}\ .
\end{equation} 

Recall  the dual partition $\bold{ t}$ of the decomposition $\{r_\varphi\}_\varphi$ of $r$ 
 defined by 
$$
t_1=\#\{\,\varphi\mid  r_\varphi\geq 1\,\},\ \ t_2=\#\{\,\varphi\mid  r_\varphi\geq 2\,\}, \ \ \hbox{\rm etc.}
$$
We have $t_1\geq t_2\geq\cdots \geq t_n$. Consider the (reduced by definition)
closed nilpotent orbit $\ov{N}_{\bold{t}}$ that
corresponds to the partition $\bold t$. All matrices in this closure $\ov{N}_{\bold t}$
have  Jordan blocks of size at most $e$. Hence we have a $GL_r$-equivariant 
closed immersion
$$
i\colon \ov{N}_{\bold t}\hookrightarrow N\otimes_{\O_E}k_E.
$$
From \cite{P-R1}, Theorem 5.4 and the above, we now deduce
that the special fiber $M^{\rm loc}\otimes_{\O_E}k_E$ of the local model $M^{\rm loc}$
is isomorphic to the pull-back of $i$ along $q$. This gives the following:

\begin{thm}\label{thm:nilp}
There is a smooth morphism of algebraic stacks
\begin{flalign*}
   \phantom{\qed} & &
   q^{\rm loc}\colon M^{\rm loc}\otimes_{\O_E}k_E \to [GL_r\backslash \ov{N}_{\bold t}].
  & &
   \qed
\end{flalign*}
\end{thm} 

\begin{cor}
The special fiber $M^{\rm loc}\otimes_{\O_E}k_E$ of the local model $M^{\rm loc}$
for the choice of $\mu$ determined by $\{r_\varphi\}_\varphi$
is smoothly equivalent to the closed nilpotent orbit 
$\ov{N}_{\bold t}$.\qed
\end{cor}

In particular, $M^{\rm loc}\otimes_{\O_E}k_E$ is reduced.
By \cite{MvdK} the closed orbits $\ov{N}_{\bold t}$ are normal 
and Frobenius split (when $k_E$ has positive characteristic),  and so we conclude that the same properties are true
for $M^{\rm loc}\otimes_{\O_E}k_E$. 

\begin{Remark}\label{rem:weyman}
Note that if all $r_\varphi$ differ amongst themselves by at most $1$, then $\ov{N}_{\bold t}=(N\otimes_{\O_E}k_E)_{\rm red}$. In \cite{P-R1}, it is conjectured that $N\otimes_{\O_E}k_E$ is in fact reduced. This holds by a classical result of Kostant when $r\leq e$,  and this is proved by Weyman in \cite{W} in the cases where either ${\rm char }\  k_E=0$, or where $e=2$, comp. Theorem \ref{st:gl_nilp_flat} below. 
\end{Remark}

\begin{Remark}\label{rem:kisin}
{\rm The fact that $M^{\rm loc}\otimes_{\O_E}k_E$ 
is reduced and normal has found an interesting application
in the theory of deformations of Galois representations by Kisin \cite{Ki}.
This application is based on the following lemma, comp.\ \cite[Cor.\ 2.4.10]{Ki}: {\it Let $X$ be a scheme which is proper and flat over the spectrum $S$ of a complete discrete valuation ring. We denote by $X_\eta$, resp.\ $X_s$ the generic, resp.\ the special fiber. If $X_s$ is reduced, then there are  bijections between  the sets of connected components 
$$
\pi_0(X_s)=\pi_0(X)=\pi_0(X_\eta) .
$$ 
}}\end{Remark}

\medskip

Consider the
dominant cocharacter $\lambda$ of $GL_n$ that corresponds to  
$\bold t$,  and denote by $\ov \O_\lambda$
the corresponding Schubert variety in the affine Grassmannian for $GL_n$.
Now we can see that the embedding of the special fiber of the local model in the affine Grassmannian
(cf.\ \S\ref{s:afv}.\ref{ss:embedding}, and for this example \cite{P-R1})
gives an isomorphism
\begin{equation}
M^{\rm loc}\otimes_{\O_E}k_E\xrightarrow{\sim} \ov \O_\lambda.
\end{equation}
Since by varying the data $\bold t$  we can obtain all dominant
cocharacters $\lambda$, this observation together with Theorem \ref{thm:nilp} also shows

\begin{thm}[{\cite[Th.~C]{P-R1}}]\label{thm:sch&no}
 Any Schubert variety in the affine Grassmannian of
$GL_n$ is smoothly equivalent to a nilpotent orbit closure for $GL_r$, for
suitable $r$. \qed 
\end{thm}

This has also been shown independently by Mirkovi\'c and Vybornov \cite{MVy}.
Recall that earlier Lusztig \cite{Lu} interpreted certain Schubert varieties in the
affine Grassmannian of $GL_r$ as compactifications of the nilpotent
variety of $GL_r$ (namely the Schubert variety corresponding to the
coweight $(r,0,\ldots, 0)$), compatible with the orbit stratifications of
both varieties. In particular, as used by Lusztig in his paper, all
singularities of nilpotent orbit closures occur in certain Schubert
varieties in the affine Grassmannians. The above goes in 
the opposite direction.

\begin{Remark}
This tight connection between local models (or affine Schubert varieties)
and nilpotent orbits does not persist for other groups. 
There are,  however, some  isolated instances of such a correspondence
in other cases. 
  For example, the reduced special fibers of 
the local models
for the ramified unitary groups and special parahoric subgroups
are smoothly equivalent to nilpotent orbit closures in the classical 
symmetric pairs $({\mathfrak sl}_n, {\mathfrak so}_n)$, resp.\ 
 $({\mathfrak sl}_{2n}, {\mathfrak sp}_{2n})$ which have been studied
 by Kostant-Rallis \cite{KoRall}, and Ohta \cite{Oh}. See \cite[\S 5]{P-R4}, and especially Theorem 5.4
 and its proof in \cite{P-R4}, for more details. However, not all such nilpotent orbits
 appear in this correspondence.
\end{Remark}

\section{Local models and matrix equations}\label{s:matrix}

In some cases, local charts around points of local models can be described via  the spectra of affine rings given by generators and relations, in shorthand matrix form ({\it matrix equations}).
We have already seen some instances of this in Example \ref{exn=2}, Remark \ref{rk:GU2}, and \S \ref{s:no}. Rather than giving a formal definition of what we mean by matrix equations, we list in this section a few examples.  Obviously,  structure results on matrix equations have consequences for local models. What  is more surprising is that   sometimes results on local models imply structure results on matrix equations. 

\subsection{Matrix equations related to naive local models}\label{ss:redmat}
 Our first example is as follows. Let $\O$ be a discrete valuation ring with uniformizer $\pi$. We fix positive integers $r$ and $n$, and consider the following closed subscheme of affine space of dimension $nr^2$ over $\Spec \O$, 
\begin{equation}\label{gencirceq}
Z_{r, n}=\biggl\{\,(A_1, \dotsc, A_n)\in {\rm Mat}_{r\times r}^n
\biggm| 
\twolinestight{$A_1 A_2\dotsm A_n= A_2 A_3\dotsm A_1=\dotsb =$}
   {$=A_n A_1\dotsm A_{n-1}=\pi\cdot I$} \,\biggr\} . 
\end{equation}
In the special case $r=1$ there is only one equation $X_1X_2\dotsm X_n=\pi$ in the $n$ unknowns 
$X_1, X_2, \ldots, X_n$,  which describes the semistable reduction case. The special fiber $Z_{r, n}\otimes_\O k$ is called  the {\it generalized circular variety} over the residue field $k$. The scheme $Z_{r, 2} \otimes_\O k$ is called the {\it variety of circular complexes}, and has been considered long before local models were defined, cf.\ \cite{MT, St}. 
\begin{thm}[G\"ortz {\cite[4.4.5]{Go1}}]\label{st:gencirc_flat}
The scheme $Z_{r, n}$ is flat over $\O$, with reduced special fiber. The irreducible components of its special fiber are normal with rational singularities, so in particular are Cohen-Macaulay. \qed
\end{thm}
The matrix equation (\ref{gencirceq}) arises  in the analysis of local charts for local models for the triple consisting of $GL_n$, the Iwahori subgroup, and the minuscule cocharacter $\mu=(1^{(r)}, 0^{(n-r})$. Recall from Theorem \ref{st:GL_n_loc_mod_flat} that, in this case, the local model coincides with the naive local model. More precisely, and similarly to what happened in \S\ref{s:no}.\ref{ss:lm&no},  we define a scheme $\widetilde M^{\rm loc}$ over $M^{\rm loc}$ which parametrizes, in addition to a point $(\F_i\mid i\in \mZ/n\mZ)$ of $M^{\rm loc}(S)$, a basis of $\Lambda_{i, S}/\F_i$. Then associating to the transition morphisms $\Lambda_{i, S}/\F_i\to \Lambda_{i+1,S}/\F_{i+1}$ their matrices in terms of these bases, we obtain a morphism $q\colon \widetilde M^{\rm loc}\to Z_{r, n}$, which turns out to be smooth, cf.\ \cite{Fa1, P-R1}. Hence the properties claimed in the theorem follow from Theorem \ref{st:GL_n_loc_mod_flat}, locally at each point of $Z_{r, n}$ in the image of $q$. Something similar holds for any parahoric subgroup corresponding to a partial periodic lattice chain $\L$. Now apply this result to the local model for the triple consisting of $GL_{rn}$, the parahoric subgroup corresponding to the periodic lattice chain $\L=\{\Lambda_i\mid i\in r\mZ\}$, and the minuscule coweight $( 1^{(r)}, 0^{(rn-r)})$. It is easy to see that in this case the morphism $q\colon \widetilde M^{\rm loc}\to Z_{r, n}$ is surjective, and this proves the claim, cf.\ \cite{Go5}.

 The next example arises in the analysis of the naive local model for the triple consisting of a symplectic group,  a {\it non-special} maximal parahoric subgroup,  and the unique conjugacy class of nontrivial minuscule coweights $\mu$.  Let $n$ be even, and define 
\begin{equation}\label{sympnonsp}
Z=\{\,A\in {\rm Mat}_{n\times n}\mid AJ^t\!A=~^t\!AJA=\pi\cdot I\,\} . 
\end{equation}
Here, as in the beginning of \S \ref{s:examples}.\ref{ss:GSp_2g}, $J=J_n$ denotes the matrix describing the standard symplectic form. 
\begin{thm}[G\"ortz {\cite[\S5]{Go2}}] %\label{decon}
The   scheme  $Z$ is flat over $\O$, with reduced irreducible normal special fiber,  which has only rational singularities.  \qed
\end{thm}
The proof of G\"ortz of this theorem uses local model techniques, combined with the theory of deConcini \cite{deC1} of doubly symplectic tableaux which gives a good basis of  the coordinate ring of $Z\otimes_\O k$ as a  $k$-vector space.
 
Similarly, in the analysis of the naive local model for the triple consisting of a symplectic group,  the parahoric subgroup which stabilizes a pair
of lattices $\Lambda$, $\Lambda'$ where $\Lambda$ is self-dual and $\Lambda'$  is self-dual up to scalar $\pi$,
 and the unique (nontrivial dominant) minuscule coweight $\mu$, the following matrix equations arise, 
\begin{equation}\label{sympsp}
Z=\{\,A, B\in {\rm Mat}_{n\times n}\mid AB=BA=\pi\cdot I,\ ^t\!A=A,\  ^t\!B=B\,\} . 
\end{equation}
 More precisely, $Z$ is locally around the origin isomorphic to an open neighborhood of the `worst point' of the local model in question.  
\begin{thm}[Chai--Norman {\cite{CN}}, {\cite{DP}}, G\"ortz {\cite[2.1]{Go2}}] \label{Thm6.1.6}
The   scheme  $Z$ is flat, normal and Cohen-Macaulay over $\O$, with reduced special fiber. The irreducible components of its special fiber are normal with rational singularities.  \qed
\end{thm}
Whereas G\"ortz' proof  of this theorem uses local models (in particular, the embedding of the special fiber in the affine Grassmannian) and Frobenius splitting methods, the proof of Chai and Norman uses techniques from the theory of algebras with straightening laws (and the proof in \cite{DP} is a simplification of this proof). The Cohen-Macaulay property of $Z$ is shown directly in \cite{CN}, but it can also be derived by the methods of 
G\"ortz (see \cite[\S4.5.1]{Go1}). We refer to \cite{Go2} and \cite{DP} for further discussion of other methods in the literature.

Another example of a matrix equation we have seen already in the previous section,  cf.\ (\ref{nileq}). For better comparison with the  matrix equations appearing right after it, let us recall it. As in the beginning of \S\ref{s:no}.\ref{ss:lm&no},
let  $F/F_0$ be a totally ramified separable extension of degree $e$. Let $\pi$ be a uniformizer of $\O_F$, and let $Q(T)\in \O_{F_0}[T]$ be the Eisenstein polynomial satisfied by $\pi$.  As in loc.\ cit.,\ we fix a tuple ${\bold r}=(r_\varphi)$. Then 
\begin{equation}
N=N_{{\bold r}}=
   \biggl\{\,A\in {\rm Mat}_{r\times r}
   \biggm| 
   \twolinestight{$Q(A)=0$ and}
      {$\det(X\cdot I-A)\equiv \prod\nolimits_\varphi (X-\varphi(\pi))^{r_\varphi}$}\,\biggr\}\ , 
\end{equation}
which is a scheme over $\Spec \O_E$, where $E$ is the reflex field corresponding to  ${\bold r}$. 
\begin{thm}[Weyman {\cite{W}}]\label{st:gl_nilp_flat}
Assume that all $e$ integers $r_\varphi$ differ amongst each other by at most $1$. Assume further that either the characteristic of the residue field $k_E$ is zero, or that $e\leq 2$, or that $\sum\nolimits_\varphi r_\varphi\leq e$. Then  $N$ is flat over $\O$, with reduced special fiber, which is normal with rational singularities.\qed
\end{thm}
As explained in the previous section, the scheme $N$ occurs in relation to the naive local model for the triple consisting of the group $G := \Res_{F/F_0} GL_n$, the natural special maximal parahoric subgroup,  and  the minuscule  cocharacter $\mu$ determined by 
${\bold r}$, cf.\ \S\ref{s:examples}.\ref{ss:GL_n}. If the conclusion of Theorem \ref{st:gl_nilp_flat} were true without the ``further'' restrictions listed (as is conjectured in \cite{P-R1}), then the local model and the naive local model would coincide in this case. 

For the triple consisting of $G := \Res_{F/F_0} GSp_{2n}$,  the natural special parahoric subgroup,  and the natural minuscule cocharacter $\mu$, one obtains in the analogous way the following matrix equation, cf.\ \cite[12.5]{P-R2}, 
% \begin{eqnarray}
% \begin{aligned}
% P=\{ & A=\begin{pmatrix} a&b\\ 0&~^ta\end{pmatrix}\in {\rm Mat}_{2ne\times 2ne}\mid a,b\in
% {\rm Mat}_{ne\times ne}\ ,\ ~^tb=-b\ ,\\
%  &  \det(X\cdot I-a)\equiv \prod_\varphi (X-\varphi(\pi))^{n},Q(a)=0\ \ \} \ .
%  \end{aligned}
% \end{eqnarray}
\begin{equation}
   P = \Biggl\{\,
          \biggl(\begin{matrix}
             a & b\\
             0 & {}^ta
          \end{matrix}\biggr) \in {\rm Mat}_{2ne\times 2ne}
   \Biggm|
   \twolinestight{$a,b\in
{\rm Mat}_{ne\times ne}$, ${}^tb=-b$, $Q(a)=0$,}
      {$\det(X\cdot I-a)\equiv \prod_\varphi (X-\varphi(\pi))^{n}$}\,\Biggr\}.
\end{equation}
\begin{conjecture}[{\cite[12.5]{P-R2}}]
The scheme $P$ is flat over $\Spec{\cal O}_{F_0}$, with reduced special fiber.
\end{conjecture}
If this conjecture held true, it would follow that in this case the naive local model is flat, i.e.,\ coincides with the local model --- which would constitute a special case of Conjecture \ref{st:ram_loc_mod_conj}.

Our next examples are related to the case of a ramified unitary group. Let $F/F_0$ be a ramified quadratic extension obtained by adjoining the square root of a uniformizer $\pi_0$ of $F_0$. The following matrix equations  arise in connection with the triple consisting of a group of unitary similitudes of size $n$ for $F/F_0$, a special maximal parahoric subgroup (in the case when $n$ is odd, the parahoric subgroup fixing a self-dual lattice, and in the case when $n$ is even, the parahoric subgroup fixing a lattice which is selfdual up to a scalar $\sqrt \pi_0$), and a minuscule cocharacter given by $(r, s)$ with $r+s=n$. Consider the following schemes of matrices.
  
\noindent {When $n$ is odd},  
% \begin{eqnarray}\label{uniteqodd}
% \begin{aligned}
%  N=\{  A\in{\rm Mat}_{n\times n}\mid A^2=\pi_0\cdot I,\  A^t=HAH, \\
% \hfill{\rm char}_A(T)=(T-\sqrt \pi_0)^s(T+\sqrt \pi_0)^r\ , \wedge^{s+1}A=0,  \wedge^{r+1}A=0\ \} .
% \end{aligned}
% \end{eqnarray}
\begin{equation}\label{uniteqodd}
   N = \left\{A \in {\rm Mat}_{n\times n}\left|\,
   \threelinestight{$A^2=\pi_0\cdot I,\  A^t=HAH$,}
      {${\rm char}_A(T)=(T-\sqrt \pi_0)^s(T+\sqrt \pi_0)^r$,}
      {$\wedge^{s+1}A=0,\  \wedge^{r+1}A=0$}
      \right.\right\}.
\end{equation}
\noindent { When  both $n=2m$ and $s$ are even}, 
% \begin{eqnarray}\label{uniteqev}
% \begin{aligned}
%  N=\{A\in{\rm Mat}_{n\times n}\mid A^2=\pi_0\cdot I,\  A^t=-JAJ,\\ 
%  {\rm char}_A(T)=(T-\sqrt \pi_0)^s(T+\sqrt \pi_0)^r\ , \wedge^{s+1}A=0,  \wedge^{r+1}A=0 \} ,
%  \end{aligned}
% \end{eqnarray}
\begin{equation}\label{uniteqev}
   N = \left\{A\in{\rm Mat}_{n\times n} \left|\,
   \threelinestight{$A^2=\pi_0\cdot I,\  A^t=-JAJ,$}
      {${\rm char}_A(T)=(T-\sqrt \pi_0)^s(T+\sqrt \pi_0)^r$,}
      {$\wedge^{s+1}A=0,\  \wedge^{r+1}A=0$}
      \right.\right\},
\end{equation}
where the conditions on wedge powers are imposed only when $r\neq s$.
 Here, as in the beginning of \S\ref{s:examples}.\ref{ss:GSp_2g}, $H=H_n$ denotes the antidiagonal unit matrix, and $J=J_{n}$ the skew-symmetric matrix 
with  square blocks $0_m$ on the diagonal and $H_m$, resp.\ $-H_m$, above the diagonal, resp.\ below the diagonal. 
\begin{conjecture}[{\cite[\S5]{P-R4}}]\label{conjmatunit}
 The schemes $N$  above are flat over $\O_E$, with reduced special fiber (which is then normal, with rational singularities). 
\end{conjecture}
If this conjecture were true, it would follow that for the local models mentioned above, the wedge local model contains the local model as an open subscheme, cf.\ Remark \ref{locverswedge} (a corrected version of \cite[Rem.\ 5.3]{P-R4}). 
\begin{Remark}
There should be similar matrix equations related to local models for orthogonal groups. This does not seem to have been investigated so far. 
\end{Remark}

\section{Local models and quiver Grassmannians}\label{s:lm&qg}

In a few cases, the special fibers of local models can be identified with certain quiver Grassmannians in the sense of Zelevinsky and others, cf.\ \cite{Z}. In this section we discuss this in rough outline. 

\subsection{Quiver Grassmannians}\label{ss:qg}

Let $Q$ be a quiver, with set of vertices $Q_0$ and set of arrows $Q_1$. Then $Q$ is in the obvious way a category. Let $(V, \varphi)$ be a representation of $Q$ over the field $k$, in other words, a functor from the category $Q$ to the category of finite-dimensional 
vector spaces over $k$. To any such representation there is associated its dimension vector ${\bold d}(V)\in (\mZ_{\geq 0})^{Q_0}$ with
${\bold d}(V)_i=\dim V_i$. Let 
${\bold e}\in (\mZ_{\geq 0})^{Q_0}$ such that ${\bold e}\leq {\bold d}$, i.e.,\ each component of ${\bold e}$ is less than or equal to the corresponding  component of ${\bold d}$.  The {\it quiver Grassmannian} associated to these data is the projective variety (comp.\ e.g.\ \cite[\S1]{CR}) 
\begin{equation}\label{quivergr}
{\rm Gr}_{\bold e}(V) = \bigl\{\,\F_i\in {\rm Gr}(e_i, V_i),\ \forall i\in Q_0 \bigm| \varphi_{i, j}(\F_i)\subset \F_j,\ \forall (i,j)\in Q_1\,\bigr\} \ .
\end{equation}
The subgroup $G_V$ of elements in $\prod_{i\in Q_0} GL(V_i)$ which respect the homomorphisms  $\{\varphi_{i, j}\}_{(i, j)\in Q_1}$ acts in the obvious way on ${\rm Gr}_{\bold e}(V)$. Most often, there are infinitely many orbits. 

\subsection{Relations to local models}\label{ss:qg&lm} 

We consider the situation of \S\ref{s:examples}.\ref{ss:ResGL_n}, i.e.,\ take $G=GL_n$ over $F$, $\mu=(1^{(r)}, 0^{n-r)})$, and a periodic lattice chain in $F^n$ extracted from the 
standard lattice chain $\Lambda_i,$ $i\in \mZ$,  by keeping those $\Lambda_i$ with $i$ congruent to an element  in a fixed non-empty subset $I\subset \mZ/n\mZ$.  
Let $\ol \Lambda_i=\Lambda_i\otimes_\O k$, with the linear maps $\ol \Lambda_i\to \ol \Lambda_{i+1}$ induced by the inclusions $\Lambda_i\subset \Lambda_{i+1}$. Using the identification $\Lambda_{i+n}=\pi\Lambda_i$, we may identify $\ol\Lambda_{i+n}$ with 
$\ol\Lambda_i$, and therefore define unambiguously $\ol\Lambda_i$ for $i\in \mZ/n\mZ$. By keeping only those $\Lambda_i$ with $i\in I$, we obtain a 
representation $\ol \Lambda_I$ of the quiver of type $\wt A_{|I|}$.  Here an extended Dynkin diagram of type $\tilde A$ defines a quiver by choosing the clockwise orientation of its bonds. This representation is characterized up to isomorphism by the following conditions:
\begin{altenumerate}
\item $\dim \ol\Lambda_i=n$ for all $ i\in I . $
\item  $\dim \Ker \varphi_{i, i'}=i'-i$ for all $i \leq i' \leq i+n.$

\end{altenumerate}
From \S\ref{s:examples}.\ref{ss:GL_n} it is plain that  a point of the local model $M=M^{\rm loc}_{G, \mu, I}$ with values in a $k$-scheme 
$S$ corresponds to a $S$-valued point $\F_i$ of the Grassmannian of subspaces of dimension $n-r$ of $\ol\Lambda_i$, one for each $i\in I$ such that, under $\varphi_{i,i'}$, the subspace $\F_i$ is carried into a subspace of $\F_{i'}$. 

Comparing with (\ref{quivergr}),  we see that  $M\otimes_{\O_F}k$ can be identified with the quiver Grassmannian ${\rm Gr}_{\bold n-\bold r}(\ol \Lambda_I)$ of subspaces with scalar dimension vector ${\bold n}-{\bold r}$ of the representation $\ol\Lambda_I$ of the quiver of type $\wt A_{|I|}$. Furthermore, under this identification, the action of the loop group $L^+P_I$ on  $M\otimes_{\O_F}k$ from Example \ref{eg:GL_n_emb} coincides  with the action of the automorphism group $G_V$ of the quiver $\ol\Lambda_I$ from \S\ref{s:lm&qg}.\ref{ss:qg}. In particular, in this case the $G_V$-action has only finitely many orbits. 

From this perspective, the local model  $M$ is a deformation over $\O_F$ of a quiver Grassmannian over $k$. 

\begin{Remark}
In \cite{CR}  and other papers in the area of representations of algebras, quiver Grassmannians are considered as varieties, i.e.,\ nilpotent elements are neglected.  It follows from G\"ortz's Theorem \ref{st:GL_n_loc_mod_flat} that  the quiver Grassmannians of type $\wt A$ are reduced. For other quiver Grassmannians this question does not seem to have been considered in the literature. 
\end{Remark}
\begin{Remark} It is not clear which local models can be described in this way. 
\begin{altenumerate}
\item  In \S\ref{s:examples}.\ref{ss:ResGL_n}, we mentioned the splitting model $\M=\M_{G, \{\mu\}, \L}$ from \cite{P-R2} for $G={\rm Res}_{F/F_0}(GL_n)$, where  $F/F_0$ is a totally ramified extension. Similar to the above identification, the special fiber of $\M$ can be described as a {\it subvariety}  of a quiver {\it flag variety} of a representation of a quiver of type $\wt A$ (defined by the condition that the nilpotent operator induced by $\pi$ induces the zero endomorphism on a certain associated graded vector space), cf.\ \cite{P-R2}.

\item Recall from \S\ref{s:examples}.\ref{ss:GSp_2g} the local model  corresponding to the triple $(GSp_{2g}, \{\mu\}, \L)$, where $\{\mu\}$ is the unique conjugacy class of nontrivial minuscule coweights of $GSp_{2g}$, and where $\L$ is a self-dual periodic lattice chain. In fact, to simplify matters, let us assume that $\L$ is maximal. 
By choosing the symplectic form as in  \S \ref{s:examples}.\ref{ss:GSp_2g}, and taking for $\L$ the standard lattice chain, we see that $\widehat\Lambda_i=\Lambda_{-i}$. Using again the notation $\ol\Lambda_i$ for $\Lambda_i\otimes_{\O_F}k$, we see that we obtain a non-degenerate pairing
\begin{equation}\label{pairing}
\ol\Lambda_i\times \ol\Lambda_{-i}\to k  .
\end{equation}
Now a point of the special fiber $M\otimes_{\O_F}k$  is given by a subspace $\F_i$ of dimension $g$ of $\ol\Lambda_i$, one for each $i$,  such that,  under each map $\varphi_{i, i+1}\colon\ol\Lambda_i\to \ol\Lambda_{i+1}$,  the subspace $\F_i$ is mapped into a subspace of $\F_{i+1}$, and such under the natural pairing 
(\ref{pairing})  the subspaces $\F_i$ and $\F_{-i}$ are 
perpendicular to each other, for all $i$. However,  this kind of   object  has apparently not been considered in the context of quiver Grassmannians. 
\end{altenumerate}
\end{Remark}

\section{Local models and wonderful completions}\label{s:wc}

In this section, which is of an (even)  more informal nature, we will explain various relations between the theory of local models  and  
the so-called wonderful compactifications of symmetric spaces. This extends to also give a relation of  local models  for $GL_n$
with Lafforgue's compactifications of the quotients $(PGL_r)^s/PGL_r$.  At the moment we do not have a very good understanding
of the scope of these connections between  the theory of local models and those theories; they appear somewhat sporadic. As a result we will mainly concentrate on several illustrative examples and explanations 
that, we hope, are enough to explain why one should expect such a connection in the first place. We also hope
that this will motivate readers
to explore possible generalizations.  

An instance of a connection
between some local models and wonderful completions was first observed by Faltings \cite{Fa1}
(also \cite{Fa2}). Faltings starts by considering  certain schemes given by matrix equations.
These schemes are given by embedding symmetric spaces in projective spaces defined by homogeneous line bundles
and considering their closures. 
In several cases, these give affine charts of local models in the sense of \S \ref{s:matrix}. 
Faltings then uses constructions from the theory of wonderful completions of symmetric spaces
 to produce birational modifications of these schemes. In many cases, these also give 
modifications of the corresponding local models
which are regular and have as special fiber a divisor with (possibly non-reduced) normal crossings. 

In this section, after a quick review of wonderful completions (\S \ref{s:wc}.\ref{ss:wc}), we  will explain  (in  \S\S \ref{s:wc}.\ref{ss:exawcGLn}, \ref{ss:wcsympl}) a different and more direct  relation between local models and wonderful completions, 
based on some unpublished notes \cite{P2} by the first author. This was inspired by Faltings' paper. The goal in this approach is to  relate  
local models to closures of orbits of parabolic subgroups in the wonderful completion; such parabolic orbit closures
have been studied by Brion and others \cite{Br2,BrP,BrTh}. In some cases this  gives 
an alternative construction of the local models. Then, in \S \ref{s:wc}.\ref{ss:wcfaltings}, we give some comments on Faltings' methods.

Contrary to our notation earlier in the paper, in this section we shall use the symbols $\Lambda_0$, $\Lambda_1,\dotsc$ to denote arbitrary lattices in a vector space, not lattices in the standard lattice chain \eqref{disp:std_lattice_chain}.

\subsection{Wonderful completions}\label{ss:wc}

For a more complete overview of this ``wonderful" theory and its connections to classical 
algebraic geometry we refer the reader to \cite{dC,dCP}. We also refer to    \cite{dCS} and \cite[\S2]{Fa1} 
for details on the actual constructions
in the generality we require. The basic set-up is as follows. 

Let $G$ be an   adjoint semi-simple algebraic group  over a field $F$ 
of characteristic $\neq 2$ which is equipped with  an involution $\theta$
defined over $F$. Let $H=G^\theta$ be the fixed points of the involution
which is  then a reductive group over $F$; it is connected when $G$ is simply connected, cf.\ 
\cite[\S 1]{dCS}.
The corresponding symmetric space is the affine quotient $X=G/H$ over $F$. 
The {\it wonderful completion} $\ov X$ of $X$  is a smooth projective variety over $F$
which contains $X$ as a dense Zariski open subset.
It supports a left action of $G$ that extends the
translation action on $X$. In addition, it has the following property: 
The complement $\ov X-X$ is a divisor with normal crossings
which is the union of a finite set of smooth irreducible $G$-stable divisors
such that all their partial intersections are transversal;  the closures of 
the $G$-orbits in $X$ are precisely these intersections.

One basic example is obtained by  taking the group to be the product $G\times G$ with $\theta(g_1, g_2)=(g_2,g_1)$, 
so that $X=(G\times G)/G\simeq G$. Then  $\ov X=\ov G$ is a compactification of the 
group $G$. Another well-studied example is given by taking the group $PGL_n$ with involution given by
$\theta(g)=(^t\! g)^{-1}$. Then $H=PGO_n$, and $X$ is the variety of invertible symmetric matrices and $\ov X$ is the 
classical variety of ``complete quadrics,'' see \cite{dCP,Laksov,dCGMP}.

\subsection{The example of the general linear group} \label{ss:exawcGLn}

%\begin{example}\label{ex:wcGLn}
In this subsection, we explain the method of \cite{P2}. We will concentrate on two classes of examples.
For simplicity, we only consider the equal characteristic case\footnote {An extension to the mixed characteristic case depends 
on defining wonderful completions over ${\mathbb Z}_p$.
This should not present any problems (provided $p$ is odd) and 
is roughly sketched in \cite{Fa1} and 
\cite{dCS},  but we prefer to leave the details for another occasion.},
i.e.,\  the local models will be schemes over the discrete
valuation ring $\O=k[[t]]$ with uniformizer $\pi=t$. Let $F=k((t))$.
   Suppose that $ \La_0$, $ \La_1$ are two $\O$-lattices 
in the vector space $V=F^{n}$ such that
$\La_0\subset \La_1\subset t^{-1}\La_0$. 
Choose a $\O$-basis of  $e_1,\ldots , e_n$ of $\La_1$ such that $\La_0$ has $\O$-basis formed by $e_1,\ldots, e_m, te_{m+1},\ldots , te_n$, for some $m\leq n-1$. Fix an integer $0<r<n$. 

Recall from \S \ref{s:examples}.\ref{ss:GL_n} that the naive local model 
$M=M^{\rm naive}$ corresponding to the triple consisting of $GL_n$, of the minuscule cocharacter $\mu=(1^{(r)}, 0^{(n-r)})$, and of the above lattice 
chain,  is the scheme over $\Spec \O$
whose $S$-points parametrize pairs $(\F_0, \F_1)$ of $\O_S$-subbundles of rank $n-r$  of 
$\La_0\otimes_{\O}\O_S$, resp.~$\La_1\otimes_{\O}\O_S$,
 such that
the following diagram commutes 
% \begin{equation}
% \begin{matrix}
% \La_0\otimes_{\O}\O_S
% &
% \longrightarrow
% &
% \La_{1}\otimes_{\O}\O_S
% &
% \longrightarrow
% &
% t^{-1}\La_{0}\otimes_{\O}\O_S
% \\
% \rotatebox{90}{$\subset$}
% &&
% \rotatebox{90}{$\subset$}
% &&\rotatebox{90}{$\subset$}
% \\
% \mathcal F_0 & \longrightarrow & \mathcal F_1
% & \longrightarrow &  t^{-1}\F_{0} \ .
% \end{matrix}
% \end{equation}
\begin{equation}
   \xymatrix@R=2.5ex@C-1ex{
      \La_0\otimes_{\O}\O_S \ar[r] \ar@{}[d] |*{\rotatebox{90}{$\subset$}}
         &  \La_{1}\otimes_{\O}\O_S \ar[r] \ar@{}[d] |*{\rotatebox{90}{$\subset$}}
         &  t^{-1}\La_{0}\otimes_{\O}\O_S \ar@{}[d] |*{\rotatebox{90}{$\subset$}}\\
      \mathcal F_0 \ar[r]
         &  \mathcal F_1 \ar[r]
         &  t^{-1}\F_{0}.
   }
\end{equation}
Here the horizontal maps on the top row are induced by the inclusions 
$\La_0\subset \La_1\subset t^{-1}\La_0$ and $t^{-1}\F_{0}$
stands for image of $\F_0$
under the isomorphism $\La_0\otimes_{\O}\O_S\xrightarrow{\sim} 
t^{-1}\La_0\otimes_{\O}\O_S$ induced by multiplication by $t^{-1}$.
(Recall that in this case, by  G\"ortz's Theorem \ref{st:GL_n_loc_mod_flat}, the naive local model is
equal to the local model $M^{\rm loc}=M$.)

Of course, $M$ is realized as a closed subscheme of the product of Grassmannians
 ${{\rm Gr}}(n-r, n)_{\O}\times_\O {\rm Gr}(n-r, n)_{\O}$.
The generic fiber is isomorphic to ${\rm Gr}(n-r, n)_{F}$, and in the generic fiber  the embedding of $M\otimes_\O F$ in the product ${{\rm Gr}}(n-r, n)_{F}\times_F {\rm Gr}(n-r, n)_{F}$
is described 
by $\F_1\mapsto (A\cdot \F_1, \F_1)$
with $A={\rm diag}(1^{(m)}, t^{(n-m)})$ ($m$ copies of $1$ and $n-m$ copies of $t$ placed along the diagonal.)  This allows us to view $M$ 
as a deformation over $\O$ of what is essentially the diagonal embedding  ${\rm Gr}(n-r, n)_F\hookrightarrow {\rm Gr}(n-r, n)_F\times_F {\rm Gr}(n-r, n)_F$.
Such deformations have been considered by Brion \cite{Br1} (following work by Thaddeus \cite{Tha} and
others on Kapranov's ``Chow quotients" \cite{Ka}). Brion views such deformations as parametrized
by a part of the Hilbert scheme of subschemes of ${\rm Gr}(n-r, n)_\O\times_\O {\rm Gr}(n-r, n)_\O$.

Here is how this is related to the wonderful completion of $G=PGL_n$.
Set $P=P_\mu$ for the standard parabolic subgroup of  $G$ such that ${\rm Gr}(n-r, n)=G/P$. Let $\overline G$ 
be the wonderful completion  of $G$  and denote by $\overline P\subset \overline G$ the Zariski closure of $P$
in $\overline G$. The product $G\times G$ acts on  $\overline G$
in a way that extends the action $(g_1, g_2)\cdot g=g_1gg_2^{-1}$ of $G\times G$ on $G$. This also restricts to an action of $P\times P$   on $\overline P$. Over $\overline G$, we can construct a   family of closed subschemes of ${\rm Gr}(n-r, n)\times {\rm Gr}(n-r, n)$ as follows. Consider the commutative diagram 
% $$
% \xymatrix@C=20mm{
% \M\ \ \ar[r]^{ \iota}\ar[d]^{\phi}
% \ar[dr]_{\pi} & \ \ \ \  G/P\times G/P\times \overline G\ar[d]^{{\rm pr}_2}  \\
% G/P\times G/P & \ \ \ \ \overline G\ . 
% }$$
\[
   \xymatrix{
      \M \ar[r]^-{\iota} \ar[d]_-{\phi} \ar[dr]_-{\pi}
         &  G/P\times G/P\times \ol G \ar[d]^{{\rm pr}_{\ol G}}\\
      G/P\times G/P
         & \ol G.
   }
\]
Here $\M=(G\times G){\buildrel{P\times P}\over{\times}}\overline P$ 
is the ``contracted product'' which is given as the quotient of $G\times G\times \overline P$ by the right action of $P\times P$   
by 
$$
\bigl((g_1, g_2), x\bigr) \cdot (p_1, p_2) = \bigl(g_1p_1, g_2p_2, (p^{-1}_1, p^{-1}_2)\cdot x\bigr).
$$
The morphism $\iota$ is given 
by
$$
\iota \bigl((g_1, g_2), x\bigr) = \bigl(g_1P, g_2P, (g_1, g_2)\cdot x\bigr)\,  
$$
and the morphism $\phi$ by $\phi\bigl((g_1, g_2), x\bigr)=(g_1P, g_2P)$.
It is easy to  see that $\iota$ is a closed immersion; hence we can view $\pi\colon \M\to \overline G$ as a family of closed subschemes of $G/P\times G/P$ over the base $\overline G$. It follows from  \cite{Br1,BrP}
that this is a flat family. Now  the matrix $A$ gives, by
the valuative criterion of properness,  a well-defined point $[A]\colon \Spec k[[t]]\to \overline G$. 
Then it is not hard to see that the base change $\M\times_{\overline G, [A]}\O$ of $\pi$ along $[A]$ can be identified 
with the flat closure of the generic fiber $M\otimes_\O{F}$ in the naive local model $M$. 
By definition, this is the local model $M^{\rm loc}$ for our situation, and so we obtain
the following result.

\begin{thm}
In the situation described above,  there is an isomorphism 
$$
 M^{\rm loc}\simeq \M\times_{\overline G, [A]}\Spec \O.
$$
\end{thm}

Indeed, using the flatness result above,  it is enough to check that this base change is a closed subscheme of the naive local model $M$.
This can be easily verified (see the proof of Theorem \ref{faltwonder} below for a more detailed
explanation of a more interesting case).

\begin{Remark}
a) As was pointed out above,  the naive local model $M$ is flat by G\"ortz's theorem,  and so $M^{\rm loc}=M$.
However, the above construction of the local model is independent of G\"ortz's result.

b) Note that the morphism $\phi$  above is a (Zariski) locally trivial fibration with fibers 
isomorphic to $\overline P$. Hence, the singularities of the total space $\M$ are smoothly equivalent
to the singularities of $\overline P$. The singularities 
of parabolic orbit closures in complete symmetric varieties (such as $\overline G$) 
have been studied by Brion and others (e.g \cite{Br2,BrP}); we can
then obtain  results on the singularities of the local models in question. For example, one can 
 deduce from this approach that the special fibers   of these local models are reduced
and Cohen-Macaulay. Of course, these results can also be obtained by the method of embedding 
the local models in affine flag varieties described in the previous sections (see \cite{Go1}), cf.\ Theorem \ref{st:GL_n_loc_mod_flat}. (In the case of this lattice chain, which consists of multiples of two lattices,
one can obtain that the special fiber, as a whole, is Cohen-Macaulay,
cf.\ Remark \ref{rem: 2.1.3}. See also Remark \ref{rem:2.2.5} and Theorem \ref{Thm6.1.6} 
for similar results in the symplectic case.)
\end{Remark}

\subsection{Other examples, some symplectic and orthogonal cases}\label{ss:wcsympl}

Here we will explain how the method of  \S  \ref{s:wc}.\ref{ss:exawcGLn}
can be extended to relate certain local models for maximal 
parahoric subgroups in the cases of symplectic and even orthogonal groups
to wonderful completions of the corresponding symmetric spaces.

Suppose that $n=2m$ is even. We assume that  $V=F^n=k((t))^n$ is equipped with 
a perfect form $h\colon V\times V\to F$ which is  
alternating, resp.\ symmetric, the two cases leading to a description
of local models with respect to the symplectic, resp.\ the orthogonal group.
When $h$ is symmetric, we assume ${\rm char}(k)\neq 2$.
 We will assume that the form $h$ is split and ``standard"
i.e  it satisfies $\bigl(h(e_i, e_j)\bigr)_{i,j}=J_{n}$, resp.\  $\bigl(h(e_i, e_j)\bigr)_{i,j}=H_n$, for the standard basis 
$\{e_i\}_i$
of $V=F^n$, with the matrices $J_{n}$, resp.\ $H_n$ as in  (\ref{disp:sympl}) and (\ref{disp:antidiag_1}).    
Denote by $S$ the matrix of the form $h$ 
so that $S=J_{n}$, resp.\ $S=H_n$.

Recall that we denote by $G(V, h)$ 
the group of similitudes of the form $h$. By the above, this
is ${ GSp}_n(F)$, resp.\ ${ GO}_n(F)$. We consider 
the minuscule cocharacter 
$
\mu=(1^{(m)}, 0^{(m)})
$
for $G(V,h)$ expressed as a cocharacter for the standard 
torus in $GL(V)=GL_n$.

For $0\leq r\leq m$, we consider the lattice
$$
\Lambda=\sum^{r}_{j=1}\pi\O e_j+\sum^n_{j=r+1} \O e_j.
$$
Then $\La\subset \wh\La$. We will denote  by 
$\alpha$ the inclusion 
$\La\subset \wh\La$. The form $h$ restricts 
to give an $\O$-bilinear form $ \La\times \La\to \O$ and 
a perfect $\O$-bilinear form $ \La\times \wh\La\to \O$;
we will also write $h$ for these forms.
We also denote by $h'$ the (different)   alternating, resp.\ symmetric, form
on $\La$ given on the standard basis $\{\pi e_1,\ldots, \pi e_r, e_{r+1}, \ldots ,e_{n}\}$
of $\La$ by the matrix $S$ . Denote by $L$ the $\O$-submodule of rank $m$
of $\La$ generated by the first $m$ standard basis elements of $\La$ as listed above;
 it is totally isotropic for both forms $h'$ and $h$. 

In this case, the local model $M^{\rm loc}$
can be described as follows. 
Let us first give the ``naive" local model  $M=M^{\rm naive}$ for this situation.
Consider the functor over $\Spec \O$ 
whose points with values in an $\O$-scheme 
$S$ are given by $\O_S$-submodules 
$\F\subset \La\otimes_{\O}\O_S$, 
which are $\O_S$-locally free direct 
summands of rank $m$ such that 
$$
(\alpha_{\O_S})(\F )\subset \F ^\perp.
$$
Here the perpendicular $\F^\perp\subset \wh\La\otimes_{\O}\O_S$ is by definition
the kernel of the $\O_S$-homomorphism 
$ 
(\La\otimes_{\O}\O_S)^*=\wh\La\otimes_\O\O_S\to \F^*
$ 
which is the dual of the inclusion $\F \subset \La\otimes_\O\O_S$. 
This condition is equivalent to 
$$
(h\otimes_\O\O_S)(\F , \F )\equiv 0.
$$
This functor is representable by a projective scheme $M$ over $\Spec \O$ which is a closed subscheme
of the Grassmannian ${\rm Gr}(m, n)$ over $\Spec \O$. The generic fiber 
of $M$ can be identified with the Langrangian, resp.\ 
(disconnected) orthogonal Grassmannian of isotropic $m$-subspaces
in $n$-space. The local model $M^{\rm loc}$ is by definition the (flat) Zariski closure
of the generic fiber  in $M$.
\smallskip

Consider the involution $\theta$ on $G={\rm PGL}_n$ given by $\theta(g)= S^{-1}(^tg)^{-1}S$.
The fixed points $H={\rm PGL}_n^\theta$ can be identified with the groups
$PGSp_{2m}$, resp.\ $PGO_{n}$. Let us consider the symmetric space $X=G/H$.
The morphism
\begin{equation*}
gH\mapsto A_g=(^tg)^{-1}\cdot S\cdot g^{-1}.
\end{equation*}
identifies $X$ with the quotients
\begin{gather*}
\bigl\{\, A\in {\rm Mat}_{n\times n}  \bigm|  {}^t\!A=-A,\ \det(A)\neq 0\,\bigr\}   \big/{\mathbb G}_m\, , \\
\bigl\{\, A\in {\rm Mat}_{n\times n} \bigm| {}^t\!A=A,\ \det(A)\neq 0\,\bigr\} \big/{\mathbb G}_m
\end{gather*}
of antisymmetric, resp.\ symmetric $n\times n$ invertible matrices 
up to homothety. Consider  the  wonderful completion
$\ov X$ of the symmetric space $X=G/H$. 
By the construction of $\ov X$, it follows that there is a morphism 
\begin{equation*}
T\colon \ov X\xrightarrow{ \ \ } {\mathbb P}^{n^2-1}=({\rm Mat}_{n\times n}-\{0\})/{\mathbb G}_m
\end{equation*}
which extends the natural  inclusion $G/H\hookrightarrow ({\rm Mat}_{n\times n}-\{0\})/{\mathbb G}_m$.
The morphism $T$ factors through the closed subscheme given by matrices 
which are antisymmetric, resp.\ symmetric. Now let us consider the 
parabolic $P$ of $G$ that corresponds to $\mu$, so that ${\rm Gr}(m, n)=G/P$.
Let us also consider the Zariski closure $\ov P{\,\rm mod\, } H=\ov{P/P\cap H}$ of the orbit
of $1\cdot H$ by the action of $P\subset G$ in $\ov X$.

There is a diagram
\begin{equation}
G/P \buildrel{q}\over\longleftarrow 
G\times^{P}(\ov P{\,\rm mod\, } H)\xrightarrow {\pi} \ov{G/H}.
\end{equation}
Here $G\times^{P}(\ov P{\,\rm mod\, } H)=(G\times\ov P{\,\rm mod\, } H)/P$
where the quotient is for the right $P$-action given by 
$(g, x)\cdot p= (gp, p^{-1}\cdot x)$. We have $q(g, x)=gP$
and $\pi $ is given by $\pi (g,x)=g\cdot x$, via the action of $G$ on $\ov{G/H}$.
There is also a morphism
\begin{equation}
\iota\colon G\times^{P}(\ov P{\,\rm mod\, } H)\to G/P\times  \ov{G/H}\ ,
\end{equation}
given by $\iota(z)=\bigl(q(z), \pi (z)\bigr)$. These fit in a diagram
% $$
% \xymatrix@C=20mm{
% \M\ \ \ar[r]^{ \iota}\ar[d]^{q}
% \ar[dr]_{\pi} & \ \ \ \  G/P \times \overline {G/H}\ar[d]^{{\rm pr}_2}  \\
% G/P  & \ \ \ \ \overline {G/H}\ , 
% }$$
\[
   \xymatrix{
      \M \ar[r]^-{\iota} \ar[d]_-{q} \ar[dr]_-{\pi}
         & G/P \times \ol {G/H} \ar[d]^-{{\rm pr}_2}\\
      G/P
         & \ol{G/H},
   }
\]
where $\M=G\times^{P}(\ov P{\,\rm mod\, } H)$. As in \cite{Br1}, one can see that:

\begin{altitemize}
\item [a)]\   The morphism $q$ is an \'etale locally trivial $G$-equivariant 
fibration with fibers isomorphic to  $\ov P{\,\rm mod\, } H$.

\item [b)]\   The morphism $\iota$ is a closed immersion which identifies $G\times^{P}(\ov P{\,\rm mod\, } H)$
with the closed subscheme of $G/P\times  \ov{G/H}$ whose points
$(gP, x)$ satisfy the ``incidence" condition $x\in g\cdot (\ov P{\,\rm mod\, } H)$. 

\item [c)]\  The morphism $\pi\colon \M\to \ov{G/H}$ is flat.
\end{altitemize} 
 
Now consider the matrix $H_\La=\bigl(h(e_i, e_j)\bigr)_{ij}\in {\rm Mat}_{n\times n}(\O)$ obtained 
by the restriction of our form $h$ to $\La\times \La$. Since $h\otimes_\O F$ is perfect,
this matrix $H_\Lambda$ gives an $F$-valued point of $G/H$. By properness, this extends to a point
\begin{equation*}
[H_\Lambda]\colon \Spec \O\to \ov{G/H} .
\end{equation*}
After these preparations we can finally give the description of the local model. 
\begin{thm}[{\cite{P2}}] \label{faltwonder}
Under our assumptions, there is an isomorphism 
$$
 M^{\rm loc}\simeq \M\times_{\overline {G/H}, [H_\La]}\Spec \O.
$$
\end{thm}
\begin{proof}(Sketch)
Denote by $M'$ the base change in the statement of the theorem:
% \begin{equation}
% \begin{CD}
% M'@>>> G\times^{P}(\ov P{\,\rm mod\, } H) \\
% @VVV @V \pi VV\\
% \Spec \O@>[H_\La]>> \ov{G/H}.
% \end{CD}
% \end{equation}
\begin{equation}
   \xymatrix{
   M' \ar[r] \ar[d]  &  G\times^{P}(\ov P{\,\rm mod\, } H) \ar[d]^-{\pi}\\
   \Spec \O \ar[r]^-{[H_\La]}  &  \ol{G/H}.
   }
\end{equation}
By c), $M'\to \Spec \O$ is flat. Since, by definition,  the local model $M^{\rm loc}$
is the Zariski closure of the generic fiber $M\otimes_\O F$ in the naive local model $M$,
it remains to show that $M'$ is a closed subscheme of $M$ and has the same
generic fiber, i.e.,\  $M'\otimes_\O F=M\otimes_\O F$.
Let us identify $G/P$ with the Grassmannian
using $g P\mapsto \F =gL $. Recall that  we can identify $M $ with 
a closed subscheme of $G/P$: this is the subscheme 
of points $g P$ for which 
$[H_\La]\in g\cdot (\ov P{\,\rm mod\, } H)$, or 
equivalently $g^{-1}\cdot [H_\La]\in \ov P{\,\rm mod\, } H$.
Using that  $L $ is an isotropic subspace for the
form $h'$,  we now obtain that the image of $\ov {P}{\,\rm mod\, } H$
under the morphism $T$ is contained in the closed subscheme with affine cone 
the antisymmetric or symmetric matrices 
$A$ for which
\begin{equation}\label{isoA}
^tv\cdot A\cdot w=0\quad\text{for all}\quad v, w\in L  \ .
\end{equation}
Now suppose that $g P$ is in $M'$, i.e.,\  
$g^{-1}\cdot [H_\La]\in \ov P{\,\rm mod\, } H$. By applying
$T$ we find that $A=(^tg)T([H_\La])g$ satisfies (\ref{isoA}).
Since,  by definition, the $\O$-valued point $T([H_\La])$ is equal to $H_\La=\bigl(h(e_i,e_j)\bigr)_{ij}$,  we obtain that
\begin{equation}
^tv\cdot ^tgT([H_\La])g\cdot w=h(g v, gw)=0   \quad\text{for all}\quad v , w\in L \ .
\end{equation}
Since $\F =gL $, this shows that $\F $ is isotropic for $h$. Hence $M'$
is a closed subscheme of $M$. Now it is not hard to show
that the generic fibers of $M'$ and $M$ are equal, and the claim follows. 
\end{proof}

\begin{Remark} \label{rem:wcothercases} 
This approach  can also be applied to the  local model studied by Chai and Norman (\cite{CN}, cf.\ (\ref{sympsp}))
and to certain orthogonal local models corresponding to pairs of lattices. The relevant symmetric space 
is the one corresponding to the symplectic, resp.\ the orthogonal group. In the interest 
of brevity we omit discussing these examples.
\end{Remark}

 \subsection{Wonderful completions and resolutions}\label{ss:wcfaltings}

%\begin{Remark}\label{rem:wccharts}
In what follows, we will  first explain in rough outline some of the constructions of 
\cite{Fa1}  in the case 
of the  local model  $M=M^{\rm loc}$ for $GL_n$  considered in \S \ref{s:wc}.\ref{ss:exawcGLn} above.  Faltings' approach also applies
to cases of other groups, see Remark \ref{rem:wcsympl}. Then we sketch the method of \cite{Fa2} to include more general parahoric level structures (defined by more than two lattices). Similar constructions
also appear in  the work of Genestier  \cite{Ge1,Ge3}.  
The main goal of all these papers is to produce  resolutions of a local model $M^{\rm loc}$
which are regular and have as special fiber a divisor with normal crossings. 
More precisely, this goal may be formulated as follows. 
\smallskip

Recall that, in all cases that it is successfully constructed,
 the local model $M^{\rm loc}_{G,\{\mu\}, {\mathcal L}}$ supports an action of the 
parahoric group scheme ${\mathcal G}\otimes_{\O}\O_E$.

\begin{Definition}\label{equivmod}
An  {\sl equivariant modification}  of $M^{\rm loc}_{G,\{\mu\}, {\mathcal L}}$ consists of a  proper $\O_E$-scheme
 that supports an action of ${\mathcal G}\otimes_{\O}\O_E$
and a ${\mathcal G}\otimes_{\O}\O_E$-equivariant proper birational morphism
$\pi\colon N\to M^{\rm loc}_{G,\{\mu\}, {\mathcal L}}$, 
which is an isomorphism on the generic fibers. We can obtain such modifications
by blowing-up ${\mathcal G}\otimes_{\O}\O_E$-invariant subschemes
of $M^{\rm loc}_{G,\{\mu\}, {\mathcal L}}$ which are supported in the special fiber.
\end{Definition}
 
It is reasonable to conjecture that there   always exists an equivariant 
modification $N\to M^{\rm loc}_{G,\{\mu\}, {\mathcal L}}$
such that $N$ is regular and has as special fiber a divisor 
with (possibly non-reduced) normal crossings \cite{P1}.
 \smallskip

Let us return to the local model $M=M^{\rm loc}$
for $GL_n$ and $\mu=(1^{(r)}, 0^{(n-r)})$ considered in \S \ref{s:wc}.\ref{ss:exawcGLn}. Consider, as in \S\ref{s:matrix},  the $GL_r\times GL_r$-torsor 
$$
\widetilde M\xrightarrow{ \ } M
$$
 given by choosing bases for $\Lambda_0/\F_0$ and $\Lambda_1/\F_1$, 
$$
\widetilde M(S)=\bigl\{(\F_0, \F_1)\in M(S),\ \alpha_0\colon  \Lambda_{0, S}/\F_0\xrightarrow{\sim}\O_S^r,\ \ \alpha_1\colon   \Lambda_{1, S}/\F_1\xrightarrow{\sim}\O_S^r\bigr\} \ .
$$
The scheme $\widetilde M$ affords a morphism $q\colon \widetilde M\to Y$, where 
$Y$ is the $\O$-scheme of matrices
\begin{equation}
Y=\bigl\{\, (A, B)\in {\rm Mat}_{r\times r}\times {\rm Mat}_{r\times r} \bigm| A\cdot B=B\cdot A=\pi\cdot I \,\bigr\} , 
\end{equation}
comp. \eqref{gencirceq}. The morphism $q$ is
given by sending $(\F_0,\F_1; \alpha_0, \alpha_1)$ to the pair 
of matrices that describe   the maps 
$\La_{0, S}/\F_0\to \La_{1, S}/\F_1$, resp.\
$\La_{1, S}/\F_1\to \La_{0, S}/\F_0$ induced by 
$\La_0\otimes_{\O}\O_S\to \La_1\otimes_{\O}\O_S$, resp.\
$\La_1\otimes_{\O}\O_S\to \La_0\otimes_{\O}\O_S$.
It is not hard to  see that the morphism $q\colon \widetilde M\to Y$
is smooth, comp.\ \cite[Th.\ 4.2]{P-R1}.  The scheme $Y$ supports 
an action of $GL_r\times GL_r$ given by 
$$
(g_1, g_2)\cdot (A, B)=(g_1Ag_2^{-1}, g_2Bg_1^{-1})\ ,
$$
such that $q$ is $GL_r\times GL_r$-equivariant.
Hence we obtain a smooth morphism of algebraic stacks
\begin{equation}
M \xrightarrow{ \ } [(GL_r\times GL_r)\backslash Y].
\end{equation}
Now consider the following variant of $Y$,  
$$
Y_1=\bigl\{\,(A, B, a) \bigm| A\cdot B=B\cdot A=a\cdot {\rm Id}\,\bigr\}\subset {\rm Mat}_{r\times r}\times {\rm Mat}_{r\times r}\times {\mathbb A}^1\ ,
$$
regarded as a $k$-variety with
$GL_r\times GL_r$-action. Following \cite[p.\ 194]{Fa1}, (see also \cite[\S 2.2]{Ge3}), we can now see 
  that the open subset $Y_1-\{a=0\}$ 
is a ${\mathbb G}_m^2$-bundle over $PGL_r\simeq (PGL_r\times PGL_r)/PGL_r$
and actually $Y_1$ can be viewed as a double affine cone over the projective
variety $X$  in ${\mathbb P}({\rm Mat}_{r\times r})\times {\mathbb P}({\rm Mat}_{r\times r})$
given by the closure  of the image of the map $A\mapsto (A, A^{\rm adj})$, where $A^{\rm adj}$ denotes the adjugate matrix of $A$.
As in loc.\ cit.\ we see that the total space of the corresponding affine  
bundle obtained by pulling back  by $\ov{PGL_r}\to X$ provides
a resolution $\ti Y_1\to Y_1$. 
\begin{comment}One sees that 
The torus $T={\mathbb G}_m^2$  acts
on $Y_1$ by 
$$
(\lambda,\mu)\cdot (A,B,a)=(\lambda\cdot A, \mu\cdot B,  \lambda\mu\cdot  a).
$$
Note that the quotient $Y_1/T$ contains $PGL_r\simeq (PGL_r\times PGL_r)/PGL_r$ 
as an open subset. We can also see that the wonderful completion $\ov{PGL_r}$ maps
to $Y_1/T$. The total space of an ${\mathbb A}^1\times {\mathbb A}^1$-bundle over $\ov{PGL_r}$ now provides
a resolution $\ti Y\to Y_1$ and
\end{comment}
 By intersecting  $\ti Y_1$ with $a-\pi=0$, we obtain
a resolution $\wt Y\to Y$.
Explicitly, $\ti Y$ can be obtained by successively blowing up ideals 
obtained from minors of $A$ and $B$.
This can now be used to obtain that in this very special case:

\begin{thm}\label{thm:resolution}
There exists an equivariant modification $\pi\colon N\to M^{\rm loc}$ 
such that $N$ is regular and has as special fiber a divisor with simple normal crossings.\qed
\end{thm}

%(See \cite{Fa1}, \cite{Fa2} and also \cite{Ge2}, and  \cite{Ge3} for more details.)

\begin{Remark}\label{rem:Lafforgue}
One can attempt to generalize this method of resolution, as well as the method of \cite{P2},  to general parahoric 
level subgroups. Let us start with a lattice chain
$$
\Lambda_0\subset \Lambda_1\subset \cdots \subset \Lambda_{s-1}\subset t^{-1}\Lambda_0
$$
in $V=F^n$ and consider the corresponding local model for $G=GL_n$ and $\mu=(1^{(r)}, 0^{(n-r)})$.
The natural replacement for $\overline {PGL_r}=\overline{(PGL_r\times PGL_r)/PGL_r}$ to accomodate more than two lattices would be a suitable
completion $\ov X$ of the quotient $X=(PGL_r)^s/PGL_r$. Unfortunately, there 
is no easy ``wonderful" choice for such a completion. Indeed, the equivariant compactifications of such quotients
have a very complicated theory,  which was developed by Lafforgue 
\cite{Laf}. To transpose the theory of \cite{P2},   one can then  consider the corresponding closures in the completion of a product 
of parabolics and attempt to obtain local models as pull-backs of the corresponding universal families.
The details of such a general construction have not been worked out.
On the other hand, as far as constructing resolutions of local models in the style of Theorem \ref{thm:resolution} are concerned, an  approach using completions of $(PGL_r)^s/PGL_r$ is given in \cite{Fa2}, see also   \cite{Ge3}. 
\end{Remark} 

\begin{Remark}
A somewhat different but related point of view which also connects with Lafforgue's
completions is explained in \cite{Fa2}. Using it, Faltings constructs a resolution of the local models of Remark 
\ref{rem:Lafforgue} when $r=2$ (when $r=1$
the local models themselves  have the desired properties, cf.\ the second part of Theorem \ref{st:GL_n_loc_mod_flat}).
We will not attempt to fully reproduce his (ingenious!) construction in this survey, 
but here is an idea.

Faltings  starts with the following observation: If $R$ is a discrete valuation ring, 
then an $R$-valued point $\F(R):=\{\F_i(R)\}_i$
of the local model $M^{\rm loc}$ gives a sequence of free $R$-modules
$$
\F_0 \subset \F_1 \subset \cdots \subset \F_{s-1}
$$
of rank $r$. This sequence can be viewed as a lattice arrangement
in a vector space of dimension $r$. Note that the local model $M^{\rm loc}$
and all its proper modifications $N$ that share the same generic fiber, also share
the same set of $R$-valued points for   a discrete valuation ring $R$. We can now view
the search for a suitable birational modification $N\to M^{\rm loc}$ as a search for a suitable 
parameter
space of lattice arrangements  as above. 
Observe that to any such lattice arrangement 
we can associate its corresponding {\it Deligne scheme} $D=D(\F(R))$ over $\Spec R$
(which is a type of local model
on its own) cf.\ \cite[\S 5]{Fa2}, \cite{Mustafin}. A suitable blow-down of the Deligne scheme gives a ``minimal" model $D^{\rm min}$
with toroidal singularities. Now parameters for  a space of lattice arrangements 
can be obtained by looking at moduli of these minimal Deligne schemes. 
More specifically, Faltings constructs  a  universal family of lattice arrangements
that supports a universal minimal Deligne scheme. The base of this family is a projective equivariant embedding of the homogeneous space $({PGL}_r)^s/{ PGL}_r$.
It turns out that this embedding is of the kind  considered by Lafforgue. As  explained in Remark \ref{rem:Lafforgue} above,  this can then be used to obtain modifications of the local models.  
For example, when $r=2$, the Deligne scheme is a projective flat curve over $\Spec R$ with 
generic fiber ${\mathbb P}^1$ and special fiber a chain of ${\mathbb P}^1$'s
intersecting transversely (our first local model for $\Gamma_0(p)$ in Example \ref{exn=2} is
such an example of a Deligne scheme).  The minimal model 
$D^{\rm min}$ now gives a semi-stable curve over $\Spec R$ and we
can parametrize the lattice arrangement by a corresponding
point of the moduli space of  genus $0$ semi-stable marked curves.
In this case, Faltings' construction produces   a 
smooth compactification of $(PGL_2)^s/PGL_2$
and hence also a regular equivariant modification of the corresponding local model
for $\mu=(1^{(2)}, 0^{(n-2)})$ and the periodic  lattice chain with $s$ members.
See also \cite{KT} for some more recent developments
in this circle of ideas.
 \end{Remark} 
%\end{example}

\begin{Remark} \label{rem:wcsympl} 
In \cite{Fa1}, Faltings gives a  construction of resolutions of local models in some cases related to 
other groups. This is done by working with explicit  schemes of matrices
that give affine charts for $M^{\rm loc}$, and  relating   those  
to wonderful completions. This then leads 
to resolutions for these affine charts. One can then obtain
 equivariant resolutions of the corresponding local models
as in Theorem \ref{thm:resolution} (note however that the special fibers of these resolutions are
not always reduced).  

We conclude this section with a list of some matrix equations which are  among those  investigated by Faltings \cite[\S4]{Fa1}. Before doing so, we make two remarks. First of all, even though some of the matrix equations that Faltings writes down are among the ones discussed in \S \ref{s:matrix} (e.g., $Z$ in \eqref{sympsp} appears in the middle of p.~194 in \cite{Fa1}, and $N$ of \eqref{uniteqodd}, resp.\ \eqref{uniteqev} occurs in the middle of p.~195 in \cite{Fa1}), and therefore are closely related to local models, this is less clear for others. In fact, his list arises from embedding symmetric spaces in projective spaces via homogeneous line bundles, and considering the singularities which occur in their closures---so there is a priori no connection to local models. Secondly, Faltings is less interested in questions of flatness, but rather allows himself to pass to the flat closure of the generic fiber, i.e.,\ to the affine variety with coordinate ring obtained by dividing out by $\pi$-torsion, and then tries to construct resolutions of those.  

Again we fix $\O$ with uniformizer $\pi$. 
One matrix equation  considered in \cite{Fa1}  is 
\begin{equation}
Z=\{\,A \in {\rm Mat}_{n\times n}\mid AA^{\rm ad}=A^{\rm ad}A=\pi\cdot I\, \} ,
\end{equation} 
where $A^{\rm ad}$ is the adjoint of $A$ with respect to a symmetric or a symplectic form. Faltings proves that, when $n$ is even, the flat closure of $Z\otimes_\O F$ inside $Z$ is Cohen-Macaulay with rational singularities. When $n$ is odd, the flat closure of $Z\otimes_\O F$ inside $Z$ is not Cohen-Macaulay, but its normalization is, with rational singularities.  Furthermore, he gives equivariant resolutions of these flat $\O$-schemes which have a normal crossings divisor as their special fibers, and computes the multiplicities of the irreducible components. 

In a similar vein, Faltings also analyzes  the intersection of $Z$ with the locus where $A=A^{\rm ad}$, i.e.,\
\begin{equation}
\{\,A \in {\rm Mat}_{n\times n}\mid A=A^{\rm ad}, A^2=\pi\cdot I\, \} .
\end{equation}
When $A^{\rm ad}$ is the adjoint of $A$ for a symmetric form,  this matrix equation relates to  
local models for the ramified unitary group, and the maximal parahoric subgroup  fixing  a selfdual lattice, comp. \eqref{uniteqodd}.
Similarly, he also considers the matrix equation
\begin{equation}
\{\,A \in {\rm Mat}_{n\times n}\mid A=-A^{\rm ad}, A^2=\pi\cdot I\, \} .
\end{equation} 
We refer to \cite[\S4]{Fa1} for further matrix equations, and results concerning them.

In \cite[Th.~13]{Fa2}, Faltings constructs resolutions of local models in the case of  
 the symplectic group  of genus $2$,  for more general parahoric subgroups 
(see also \cite{Ge3}).
\end{Remark}

\end{document}